\documentclass[12pt,a4paper]{article}
\usepackage{amsfonts}
\usepackage{mathrsfs}
\usepackage{color}
\usepackage{stmaryrd}
\usepackage{amsmath}
\usepackage{amssymb}
\usepackage{bbm}
\usepackage{graphicx}
\usepackage{enumerate}
\usepackage{theorem}
\usepackage{booktabs}
\usepackage{multirow}
\usepackage{hyperref}
\usepackage{booktabs}
\usepackage{marvosym}
\usepackage{multirow}

\makeatletter
\def\tank#1{\protected@xdef\@thanks{\@thanks
        \protect\footnotetext[0]{#1}}}
\def\bigfoot{

    \@footnotetext}
\makeatother

\newcommand{\ea}{\end{array}}
\newtheorem{theorem}{Theorem}[section]
\newtheorem{proposition}{Proposition}[section]
\newtheorem{claim}{Claim}[section]

\newtheorem{lemma}{Lemma}[section]
\newtheorem{definition}{Definition}[section]
\newtheorem{remark}{Remark}[section]

{\theorembodyfont{\rmfamily}
\newtheorem{example}{Example}[section]
}

\newenvironment{proof}{Proof.}

\allowdisplaybreaks

\begin{document}
\title{\Large \bf Stochastic generalized porous media equations over $\sigma$-finite
  measure spaces with non-continuous diffusivity function}

\author{{Michael R\"{o}ckner}$^{a,d}$\footnote{E-mail:roeckner@math.uni-bielefeld.de (M. R\"{o}ckner)}~~~{Weina Wu}$^{b,a}$\footnote{E-mail:wuweinaforever@163.com (W.N. Wu)}~~~ {Yingchao Xie}$^c$\footnote{E-mail:ycxie@jsnu.edu.cn (Y.C. Xie)}
\\
 \small  a. Faculty of Mathematics, University of Bielefeld, D-33615 Bielefeld, Germany.\\
 \small  b. Nanjing University of Finance and Economics, Nanjing, Jiangsu 210023, China.\\
 \small  c. School of Mathematics and Statistics, Jiangsu Normal University, Xuzhou, Jiangsu 221000, China.\\
 \small  d. Academy of Mathematics and Systems Science, Chinese Academy of Sciences, Beijing 100190, China.}\,
\date{}
\maketitle

\begin{center}
\begin{minipage}{130mm}
{\bf Abstract:} In this paper, we prove that stochastic porous media equations over $\sigma$-finite measure spaces $(E,\mathcal{B},\mu)$,
driven by time-dependent multiplicative noise, with the Laplacian
replaced by a self-adjoint transient Dirichlet operator $L$ and the diffusivity function given by a maximal monotone multi-valued function $\Psi$ of polynomial growth, have a unique solution. This generalizes previous results in that we work on general measurable state spaces, allow non-continuous monotone functions $\Psi$, for which, no further assumptions (as e.g. coercivity) are needed, but only that their multi-valued extensions are maximal monotone and of at most polynomial growth. Furthermore, an $L^p(\mu)$-It\^{o} formula in expectation is proved, which is not only crucial for the proof of our main result, but also of independent interest. The result in particular applies to fast diffusion stochastic porous media equations (in particular self-organized criticality models) and cases where $E$ is a manifold or a fractal, and to non-local operators $L$, as e.g. $L=-f(-\Delta)$, where $f$ is Bernstein function.

\vspace{3mm} {\bf Keywords:} Wiener process; porous media equation; Dirichlet form; maximal monotone graph; Yosida approximation; $L^p$-It\^{o} formula in expectation.

\vspace{3mm} {\bf AMS 2020 Mathematics Subject Classification:} 60H15; 76S05
\end{minipage}
\end{center}

\section{Introduction}\label{introduction}
\setcounter{equation}{0}
 \setcounter{definition}{0}
The purpose of this paper is to solve multi-valued stochastic porous media equations
(SPMEs) on $(E,\mathcal{B},\mu)$ of the following type:

\begin{equation}\label{eq:1}
\left\{ \begin{aligned}
&dX(t)-L\Psi(X(t))dt\ni B(t,X(t))dW(t),\ \text{in}\ [0,T]\times E,\\
&X(0)=x \text{~on~} E~ (x\in \mathscr{F}_e^*),
\end{aligned} \right.
\end{equation}
where $(E,\mathcal{B})$ is a standard measurable space (see \cite{P67}) with a $\sigma$-finite measure $\mu$ and $T\in(0,\infty)$ is fixed. $(L,D(L))$ is the generator of a symmetric strongly continuous contraction sub-Markovian semigroup on $L^2(\mu)$, which additionally is assumed to be the generator of transient Dirichlet form, $\mathscr{F}_e^*$ is the dual space of the corresponding extended transient Dirichlet space $\mathscr{F}_e$ (cf. Section \ref{Dirichletspaces} below).
$\Psi(\cdot):\mathbb{R}\rightarrow2^{\mathbb{R}}$ denotes a maximal
monotone graph with polynomial growth (cf. \textbf{(H1)} in Section \ref{Assumptions} below, $\Psi$ is called diffusivity function, see \cite{V}). $B$ is a Hilbert-Schmidt
operator-valued map fulfilling certain Lipschitz and growth conditions (cf. \textbf{(H2)} and \textbf{(H3)} in Section \ref{Assumptions} below). $W$ is an $L^2(\mu)$-valued
cylindrical $\mathscr{F}_t$-adapted Wiener process on a probability
space $(\Omega,\mathscr{F},\mathbb{P})$ with normal filtration
$(\mathscr{F}_t)_{t\geq0}$. Explicit assumptions and more explanations will be given in
Section \ref{Assumptions}.

At least since \cite{DRJEE}, there has been a lot of papers concerning stochastic porous media equations, e.g., about strong solutions (\cite{BDPR081, BDR2, BDR3, BDR, BRR, DRRW, Kim}), stochastic variational inequalities (\cite{GRSIAM,N}) or finite time extinction of solutions (\cite{BDR3, BDR, BRR, GCMP}) to
\begin{equation}\label{eq:1.2}
dX(t)-\Delta\Psi(X(t))= B(t,X(t))dW(t)~~ \text{on}~~\mathcal{O},
\end{equation}
with maximal monotone (multi-valued) diffusivity $\Psi$, where $\mathcal{O}$ is an open and bounded subset of the Euclidean space $\Bbb{R}^d$ (see also \cite{BDR1, LR} and references there in). In the classical deterministic case, i.e., $B\equiv0$ and $\Psi(r)=|r|^{m-1}r$, $r\in\Bbb{R}$, $m\geq1$ (see \cite{ADG}), for $x=$ a probability density on $\mathcal{O}$, its solution $X(t)$, $t\geq0$, describes the time evolution for the density of a substance in a porous medium. Heuristically (because $\Psi$ is in general not even assumed to be (single-valued) continuous and $\xi\mapsto X(t)(\xi)$ is not $C^2$) applying the chain rule we have
\begin{eqnarray}\label{eq:1.20}
\Delta\Psi(X(t))=\Psi'(X(t))\Delta X(t)+\Psi''(X(t))|\nabla X(t)|^2.
\end{eqnarray}
This shows that $\Psi'(X(t))$ is the (solution dependent) diffusion coefficient of the equation. This explains the name ``diffusivity (function)" for $\Psi$ and why $\Psi$ must be assumed to be increasing. If $\Psi$ is strictly increasing, which corresponds to $\Psi'>0$ on $\Bbb{R}$ in \eqref{eq:1.20}, this would mean that we have local strict ellipticity in \eqref{eq:1.20}, hence we would be in the nondegenerate case. We stress, however, that in this paper we do not assume this, so the degenerate case is covered. \eqref{eq:1.20} also reveals why it is important to include multivalued diffusivities, because it implies that we can cover non-continuous $\Psi$ (See Example \ref{example for psi} below). This means that its generalized derivative $\Psi'$ (in the sense of Schwartz distributions) would be a weighted Dirac measure $\delta_{r_0}$ at a point of discontinuity $r_0$ of $\Psi$. So, if we consider the time evolution $t\mapsto X(t)(\xi)$ of the density of the substance at a point $\xi\in E$ and if it ``hits" such a discontinuity point $r_0\in\Bbb{R}$ of $\Psi$, the diffusion coefficient $\Psi'(X(t)(\xi))$ would jump to $+\infty$, describing a ``very large" diffusion of ``the system" at that moment, which is an interesting case of high relevance, e.g. in physics. This is also the reason why the solution to \eqref{eq:1.2} are sometimes called singular diffusions. A prominent example is the so-called self-organized criticality (SOC) model developed by Bak, Tang and Wiesenfeld \cite{BTW88}, which can also be used to model the dynamics of phase-transition (including melting and solidification processes) as well as for the description of a large class of other diffusion problems. It remains to ``justify" the type of noise in \eqref{eq:1.2}. For a general explanation of this we refer \cite[Page:3, Section 1.2]{LR}, which explains that the noise must have this term under some reasonable assumptions.

One interesting direction of research is to replace $-\Delta$ by a nonlocal pseudodifferential operator (e.g., $(-\Delta)^\alpha$, $\alpha\in(0,1]$, cf. \cite{T} or more generally $f(-\Delta)$, where $f$ is a Bernstein function, see \cite{SSV}) in \eqref{eq:1.2}, or even more generally by a self-adjoint operator $L$ satisfying certain properties, as e.g., being the generator of a transient Dirichlet form. But there exist many interesting operators on $\Bbb{R}^d$ which are not self-adjoint on $L^2(\Bbb{R}^d,dx)$ (where $dx$ denotes the Lebesgue
measure), but on $L^2(\Bbb{R}^d,\mu)$ for some other measure $\mu$ replacing $dx$, as for instance the Friedrichs extension of the operator
$L_0=\Delta +2\frac{\nabla\rho}{\rho}\cdot\nabla$ on $L^2(\mathbb{R}^d, \rho^2dx)$ (\cite{RWX}), where $\rho\in H^1(\mathbb{R}^d)$ and $H^1$ is the usual
Sobolev space. It is also interesting to change $\Bbb{R}^d$ (the ``state space") and replace it by a smooth or even not smooth Riemannian manifold as e.g., a fractal, or allow infinite dimensional state spaces, e.g., the Wiener space. Furthermore, it is very desirable to do all this for multivalued diffusivity functions $\Psi$, as explained above. So, all this motivates the study of \eqref{eq:1}, in such a general form.

There are few results about \eqref{eq:1} on general measure spaces. One paper known to the authors about \eqref{eq:1} with $\Psi$ being a multi-valued graph is \cite[Example 7.3]{GTJMPA}, in which the existence and uniqueness of (limit) solutions and the ergodicity for \eqref{eq:1} are proved. However, in \cite{GTJMPA}, $(E, \mathcal{B}, \mu)$ is assumed to be a finite measure space and $L$ to have compact resolvent, and $\Psi$ is of linear growth, which simplifies the situation substantially. Furthermore, in \cite{GTJMPA} the extended transient Dirichlet space $\mathscr{F}_e$ is the same as the Dirichlet space $F_{1,2}$, which results in a Gelfand triple $L^2(\mu)\subset \mathscr{F}_e^*(=F_{1,2}^*)\subset (L^2(\mu))^*$. The proof of the well posedness theorem in \cite{GTJMPA} heavily relies on this Gelfand triple. We work on general $\sigma$-finite measure spaces with no further conditions on $L$, so it turns out that we must construct solutions in the smaller state space $\mathscr{F}_e^*$. Also the idea of the proof in \cite{GTJMPA} is based on a viscosity approximation, while our proof is based on the Yosida approximation. We would also like to mention another paper, where \eqref{eq:1.2} was studied with multivalued $\Psi$ and with L\'{e}vy noise replacing the Wiener noise in \eqref{eq:1.2}, namely in Section 6 of \cite{LS}. However, this is done only on $(\mathcal{O},\mathcal{B}(\mathcal{O}),dx)$ and $\mathcal{O}$ is assumed to have finite Lebesgue measure.

The present work was also motivated by papers \cite{BRR}, \cite{RRW} and \cite{RWX}. In \cite{BRR} existence and uniqueness of solutions for \eqref{eq:1.2} with linear multiplicative noise was proved for multivalued $\Psi$ with $E:=\mathcal{O}=\Bbb{R}^d$, $d\geq3$, and one of our aims here is to generalize this result to equation \eqref{eq:1} on general measure spaces and include more examples of $L$, where e.g. it is the generator of a transient Dirichlet form not only on $\Bbb{R}^d$, but also on e.g. a manifold or a fractal (cf. Examples \ref{example for local E}-\ref{nonlocal E} in Section \ref{Applications}). In \cite{RRW}, by constructing a suitable Gelfand triple with $\mathscr{F}_e^*$ as the piv\^ot space and using the variational framework (\cite{PR, LR}), the first named author of the present paper and his collaborators proved existence and uniqueness for the following stochastic generalized porous media equation in the state space $\mathscr{F}^*_e$:
\begin{eqnarray}\label{exp2}
dX(t)=(L\Psi(t,X(t))+\Phi(t,X(t)))dt+B(t,X(t))dW(t),
\end{eqnarray}
where $L$ is as above and $\Psi$ is continuous, single-valued and maximal monotone, additionally satisfying a number of other somewhat restrictive conditions (see \cite[Page: 135, condition (A1)]{RRW}, in particular, ($\Psi2$) and ($\Psi3$), where the appearing Young function $N$ is assumed to be $\Delta_2$-regular), which, in particular, imply that $\Psi(r)\rightarrow\infty$ as $r\rightarrow\infty$. In the case where $\Phi\equiv0$ and $\Psi$ is a Lipschitz increasing function which is independent of $t$ and $\omega$, the well-posedness of strong solutions to \eqref{exp2} in $F^*_{1,2}$ was proved in \cite{RWX}. In contrast to that, in the present paper, $\Psi$ is not assumed to be either Lipschitz or single-valued. In particular, we can cover the SOC model which is not included in \cite{RRW} or \cite{RWX}. A second aim of our paper is to generalize the results in \cite{RRW} and \cite{RWX} to multi-valued diffusivities $\Psi$, being just maximal monotone and of at most polynomial growth, with no further assumptions such as e.g. that $\lim_{r\rightarrow\infty}\Psi(r)=\infty$, which allows our framework to apply to the SOC model (cf. Example \ref{example for psi} in Section \ref{Applications}). Our method is completely different from \cite{RRW}, but it is a generalization of that in \cite{BRR}. Because of our much more general situation, the methods and techniques available for investigating SPMEs on $\Bbb{R}^d$ are insufficient for covering the case of general measure spaces and cannot be extended to general measure spaces in a straightforward way. Let us now describe our method. Though the overall strategy is borrowed from \cite{BRR}, some severe obstacles had to be overcome in our case, one of which was to find a proper version of It\^{o}'s formula for the $L^p$-norm of solutions for processes taking values in $L^p(E,\mathcal{B},\mu)$ with no further assumption on $(E,\mathcal{B},\mu)$, so that our results apply to general state spaces $E$, as those mentioned above.

As a first step, we consider Yosida approximating equations to \eqref{eq:1} of the following form with initial value $X_\lambda(0)\in\mathscr{F}_e^*$:
\begin{eqnarray}\label{eq:333}
dX_\lambda(t)-L(\Psi_\lambda(X_\lambda(t))+\lambda X_\lambda(t))dt=B(t,X_\lambda(t))dW(t),\ \ t\in(0,T).
\end{eqnarray}
Here $\lambda>0$ and
\begin{eqnarray*}
\Psi_\lambda(x)=\frac{1}{\lambda}\big(x-(1+\lambda\Psi)^{-1}(x)\big)\in
\Psi\big((1+\lambda\Psi)^{-1}(x)\big)
\end{eqnarray*}
is the Yosida approximation of $\Psi$, which is Lipschitz continuous. One of the main points in this paper is first to prove the well posedness of \eqref{eq:333} in $\mathscr{F}_e^*$. To this end, we consider the following approximating equations to \eqref{eq:333}
\begin{eqnarray}\label{eq:3333}
dX^\nu_\lambda(t)+(\nu-L)(\Psi_\lambda(X^\nu_\lambda(t))+\lambda X^\nu_\lambda(t))dt=B(t,X^\nu_\lambda(t))dW(t),\ \ t\in(0,T), 0<\nu\leq1,
\end{eqnarray}
with initial value $X^\nu_\lambda(0)\in F^*_{1,2}$ (which contains $\mathscr{F}_e^*$, see Section \ref{Dirichletspaces} below). This approach allows us to estimate $\|X^\nu_\lambda\|_{F^*_{1,2,\nu_0}}$, $0<\nu_0\leq1$, where $\|\cdot\|_{F^*_{1,2,\nu_0}}$ is an equivalent norm on $F^*_{1,2}$ (cf. Section \ref{Dirichletspaces} below). The key point here is to relate $\|\cdot\|_{F^*_{1,2,\nu_0}}$ with $\|\cdot\|_{\mathscr{F}^*_e}$ to prove the convergence of $X^\nu_\lambda$ to $X_\lambda$ in $\mathscr{F}^*_e$ as $\nu\rightarrow0$ (see Proposition \ref{theorem2} and the last part of the proof of Theorem \ref{Theorem}).

To prove the convergence of solutions to \eqref{eq:3333} as $\nu\rightarrow0$ to those of \eqref{eq:333} and in turn that solutions to \eqref{eq:333} as $\lambda\rightarrow0$ converge to those of \eqref{eq:1}, since $\Psi$ has a growth of at most order $m\geq1$ (see \textbf{(H1)} in Section \ref{Assumptions}), one has to control the $\|\cdot\|_{L^{2m}}$ norm of the solutions uniformly in the approximation parameters. To obatin such bounds (see \eqref{4.4} and \eqref{eqnarray12}) we need to apply It\^{o}'s formula to $|X^\nu_\lambda(t)|_{2m}^{2m}$. But this is not possible directly since $X^\nu_\lambda$ is not (right) continuous in $L^{2m}(\mu)$. Therefore, we consider the  following approximating equation to \eqref{eq:3333}:
 \begin{equation} \label{eq:33333}
dX^{\nu,\varepsilon}_\lambda(t)+A^{\nu,\varepsilon}_\lambda(X^{\nu,\varepsilon}_\lambda(t))dt=B(t, X^{\nu,\varepsilon}_\lambda(t))dW(t),\ \text{in}\ (0,T)\times E,\\
\end{equation}
where $\varepsilon\in (0,1)$ and
\begin{eqnarray*}
A^{\nu,\varepsilon}_\lambda(x)=\frac{1}{\varepsilon}\big(x-(I+\varepsilon
A^\nu_\lambda)^{-1}(x)\big),
\end{eqnarray*}
is the Yosida approximation of the operator $A^\nu_\lambda(x):=(\nu-L)(\Psi_\lambda(x)+\lambda I(x))$ and $I$ denotes the identity map on the respective space (see more details in Section 4 below). In \cite[Section 4]{BRR} the authors can directly use the $L^p$-It\^{o} formula ($p\geq2$) proved in \cite{NK} for the case $E=\Bbb{R}^d$, which is not possible here. Also one cannot expect to prove such a formula on arbitrary $\sigma$-finite measure spaces without further assumptions. The reason is that in \cite{NK}, approximations by convolution with smooth functions were crucial, which heavily depend on the linear structure of $\Bbb{R}^d$. To overcome this difficulty, we prove an $L^p(\mu)$-It\^{o} formula ($p\geq2$) in expectation (see Subsection \ref{itoformula}) to get the crucial a priori $L^{2m}(\mu)$-estimates first for the solution to \eqref{eq:33333} uniformly in the approximation parameters $\varepsilon, \nu, \lambda$, subsequently, by letting $\varepsilon\rightarrow0$ and then $\nu\rightarrow0$ we get a solution to \eqref{eq:333} satisfying \eqref{eqnarray12} (-\eqref{eqnarray14}) (see Theorem \ref{Theorem} below), which then allows us finally to take $\lambda\rightarrow0$ to obtain well-posedness of \eqref{eq:1} (see Theorem \ref{theorem1}). So, the $L^p$-It\^{o} formula in expectation (see Theorem \ref{ito} below), which we think is also of its own interest, is the main new tool that leads to the desired generalization of \cite{BRR}. Apart from other smaller obstacles that had to be overcome in comparison to \cite{BRR} (see e.g. Proposition \ref{Proposition7.1} in the Appendix), there is one other novel point we would like to mention here. Since our aim was to be able to treat non-local operators (such as pseudo-differential operators) $L$ in \eqref{eq:1} in our approach, we had to include non-local extended Dirichlet forms in our framework. We identified the crucial condition on the square field operators of the non-local Dirichlet forms needed to implement our approach. This is condition \textbf{(H4)(ii)} which is fulfilled for an abundance of (local and) non-local Dirichlet forms (see the examples in Section \ref{Applications}), in particular for such with generator $L=-f(-\Delta)$, where $f:[0,\infty)\rightarrow\Bbb{R}$ is a Bernstein function (in the sense of \cite{SSV}). Such examples for the operator $L$ together with discontinuous diffusivities in \eqref{eq:1} are completely new.

\smallskip
This paper is organized as follows. In Section \ref{Notations}, we introduce some
notations and recall some known results for preparation. In
addition, we prove some necessary technical auxiliary results, which will be
used to construct the solutions to \eqref{eq:1} in
$\mathscr{F}^*_e$. In Section \ref{Assumptions}, we present our assumptions and the two main results for \eqref{eq:1} and \eqref{eq:333}. A detailed proof of the existence result for \eqref{eq:333} will be given in Section \ref{ProofofTheorem3.2}, while the existence and uniqueness result for \eqref{eq:1} will be given in Section \ref{ProofofTheorem3.1}. A number of examples that are covered by our framework will be presented in Section \ref{Applications}, including local (nonlocal) operators $L$ on manifold or fractals. In order to make the main structure of the proofs more transparent, we shift the proofs of some estimates to Appendix \ref{Auxiliaryresults}. In addition, we present a detailed proof of the mentioned $L^p(\mu)$-It\^{o} formula ($p\geq2$) in expectation in Appendix \ref{itoformula}.

\section{Notations and preliminaries}\label{Notations}
\setcounter{equation}{0}
 \setcounter{definition}{0}
\subsection{Dirichlet spaces}\label{Dirichletspaces}
Let $(E,\mathcal{B},\mu)$ be a $\sigma$-finite measure space, which we fix in the entire paper. We assume that $(E, \mathcal{B})$ is a standard measurable space (i.e., $\sigma$-isomorphic to a Polish space, see \cite{P67}). This assumption is used in the proof of the $L^p(\mu)$-It\^{o} formula ($p\geq2$) in expectation, but also in the proof of Lemma \ref{lemma1} below, where we apply \cite[Lemma 5.1]{RW}, in which this assumption on $(E, \mathcal{B})$ was crucially used. Let $(P_t)_{t\geq0}$ be a strongly continuous, symmetric, sub-Markovian contraction semigroup on $L^2(\mu)$. Let $(L, D(L))$ be its infinitesimal generator (see e.g. \cite{FOT,MR}), which is a negative
definite self-adjoint operator on $L^2(\mu)$. We
use $\langle\cdot, \cdot\rangle_2$ and $|\cdot|_2$ for the inner
product and the norm in $L^2(\mu)$ respectively. More generally, we
set $\langle f, g\rangle_2:=\mu(fg):=\int fg d\mu$ for any two
measurable functions $f$, $g$ such that $fg\in L^1(\mu)$. For the rest of this paper we fix $(P_t)_{t\geq0}$ with generator $(L,D(L))$ on $L^2(\mu)$ with $(E,\mathcal{B},\mu)$ as above.

Consider the $\Gamma$-transform $V_r(r>0)$ of $(P_t)_{t\geq0}$
\begin{eqnarray*}
V_ru=\Gamma(\frac{r}{2})^{-1}\int_0^\infty
s^{\frac{r}{2}-1}e^{-s}P_suds, ~r>0,~u\in L^2(\mu).
\end{eqnarray*}
From \cite{Fukushima}, we can define the Bessel-potential space $(F_{1,2},\|\cdot\|_{F_{1,2}})$ by
$$
F_{1,2}:=V_1(L^2(\mu)),~\text{with~norm}~\|u\|_{F_{1,2}}=|f|_2,~~\text{for}~~u=V_1f,~~ f\in L^2(\mu),
$$
where the norm $|\cdot|_2$ is defined as $|f|_2=\langle f,f\rangle_2:=(\int_E |f|^2d\mu)^{\frac{1}{2}}$,
consequently,
$$
V_1=(1-L)^{-\frac{1}{2}},~~\text{so~that}~~F_{1,2}=D\big((1-L)^{\frac{1}{2}}\big)~~\text{and}~~\|u\|_{F_{1,2}}=|(1-L)^{\frac{1}{2}}u|_2.
$$
The dual space of $F_{1,2}$ is denoted by $F^*_{1,2}$ and $F^*_{1,2}=D((1-L)^{-\frac{1}{2}})$, it is equipped with the norms
\begin{eqnarray*}\label{eqnarray5}
\|\eta\|_{F^*_{1,2,\nu}}:=\langle \eta, (\nu-L)^{-1}\eta\rangle_2^{\frac{1}{2}},~~\eta\in F^*_{1,2},~~0<\nu<\infty.
\end{eqnarray*}
Denote the duality between $F^*_{1,2}$ and $F_{1,2}$ by
$_{F^*_{1,2}}\langle \cdot,\cdot\rangle_{F_{1,2}}$.

\vspace{2mm}
Consider the Dirichlet form $(\mathscr{E}, D(\mathscr{E}))$ on
$L^2(\mu)$ associated with $(L, D(L))$, i.e.,
\begin{eqnarray*}
   && D(\mathscr{E}):= F_{1,2},~~ \text{and} \\
   && \mathscr{E}(u,v):=\mu(\sqrt{-L}u \sqrt{-L}v),~~ u,v\in F_{1,2}.
\end{eqnarray*}
Let $D(\mathscr{E})$ be equipped with the inner product
$\mathscr{E}_1:=\mathscr{E}+\langle\cdot,\cdot\rangle_2$.

If $(\mathscr{E},D(\mathscr{E}))$ is a transient Dirichlet space, that is,
there exists $g\in L^1(\mu)\cap L^\infty(\mu)$, $g>0$, such that
$\mathscr{F}_e\subset L^1(g\cdot\mu)$ continuously, let $(\mathscr{E},\mathscr{F}_e)$ be the corresponding extended Dirichlet space (see \cite{FOT}), which
is the completion of $F_{1,2}$, with respect to the norm
$$\|\cdot\|_{\mathscr{F}_e}:=\mathscr{E}(\cdot,\cdot)^{\frac{1}{2}}.$$
Then $F_{1,2}=\mathscr{F}_e\cap L^2(\mu)$. Let $\mathscr{F}^*_e$ be its dual space with inner product
$\langle\cdot,\cdot\rangle_{\mathscr{F}^*_e}$ and corresponding norm
$\|\cdot\|_{\mathscr{F}^*_e}$, which is induced by the Riesz map
$\mathscr{F}_e\ni u\mapsto \mathscr{E}(\cdot,u)\in \mathscr{F}^*_e$.
Denote the duality between $\mathscr{F}^*_e$ and $\mathscr{F}_e$ by
$_{\mathscr{F}^*_e}\langle \cdot,\cdot\rangle_{\mathscr{F}_e}$. Both $\mathscr{F}_e$ and $\mathscr{F}^*_e$
are Hilbert spaces. For more background knowledge on Dirichlet forms, we refer to \cite{FOT, MR}. From now on we assume:

\vspace{2mm}
\textbf{(L.1)} The symmetric Dirichlet form $(\mathscr{E},D(\mathscr{E}))$ associated with $(L,D(L))$ is transient.

\vspace{2mm}
Consider the inner product
$\mathscr{E}_\nu:=\mathscr{E}+\nu \langle\cdot, \cdot \rangle_2,~ \nu\in(0,\infty)$, on
$F_{1,2}$, i.e.,
\begin{eqnarray}\label{eqnal1}
  \|v\|^2_{F_{1,2,\nu}}:=\mathscr{E}(v,v)+\nu\int|v|^2d\mu=\|v\|^2_{\mathscr{F}_e}+\nu\int|v|^2d\mu,\
\text{for}\ v\in F_{1,2},
\end{eqnarray}
and
\begin{eqnarray*}
 \|l\|_{F^*_{1,2,\nu}}:=_{F^*_{1,2}}\langle l, (\nu-L)^{-1}l\rangle^{\frac{1}{2}}_{F_{1,2}}:=\sup_{\substack{v\in
F_{1,2}\\ \|v\|_{F_{1,2,\nu}}\leq1}}l(v),~l\in F_{1,2}^*,
\end{eqnarray*}
\begin{eqnarray*}
 \|l\|_{\mathcal{F}^*_e}:=\sup_{\substack{v\in
\mathscr{F}_e\\ \|v\|_{\mathscr{F}_e}\leq1}}l(v),~l\in \mathscr{F}^*_e.
\end{eqnarray*}
Since $F_{1,2}\subset\mathcal{F}_e$ continuously and densely, we have
\begin{eqnarray*}
\mathscr{F}^*_e\subset F^*_{1,2}\ \text{continuously\ and\ densely}.
\end{eqnarray*}

\begin{proposition}\label{theorem2}
Let $l\in\mathscr{F}_e^*$. Then $\nu\mapsto\|l\|_{F^*_{1,2,\nu}}$ is decreasing,
\begin{eqnarray}
 &&\lim_{\nu\rightarrow0}\|l\|_{F^*_{1,2,\nu}}=\sup_{\nu>0}\|l\|_{F^*_{1,2,\nu}}=\|l\|_{\mathscr{F}_e^*}, \label{eqnal6}\\
&&\|l\|_{F^*_{1,2}}\leq \|l\|_{F^*_{1,2,\nu}}\leq
\frac{1}{\sqrt{\nu}}\|l\|_{F^*_{1,2}},~~\forall~0<\nu\leq1.\label{eqnal19}
\end{eqnarray}
\end{proposition}
\begin{proof}
Firstly, note that for all $l\in F_{1,2}^*$ and $0<\nu'\leq \nu<\infty$, we have
\begin{eqnarray*}
\|l\|_{F^*_{1,2,\nu}}=:
\sup_{\substack{{v\in
F_{1,2}}\\{\|v\|_{F_{1,2,\nu}\leq1}}}}l(v)\leq \sup_{\substack{{v\in
F_{1,2}}\\{\|v\|_{F_{1,2,\nu'}\leq1}}}}l(v)=\|l\|_{F^*_{1,2,\nu'}},
\end{eqnarray*}
i.e., $\forall~l\in F^*_{1,2}$, $\|l\|_{F^*_{1,2,\nu}}$ is decreasing in $\nu$. In particular, the first equality in \eqref{eqnal6} and the first inequality in \eqref{eqnal19} hold.

Let $l\in \mathscr{F}_e^*$. Since $\mathscr{F}_e^*\subset F^*_{1,2}$ continuously and densely, we have $l\in
F^*_{1,2}$ and
\begin{eqnarray*}
  \|l\|_{F^*_{1,2,\nu}}=\sup_{\substack{{v\in
F_{1,2}}\\{\|v\|_{F_{1,2,\nu}}\leq1}}}l(v)\leq \sup_{\substack{{v\in
\mathscr{F}_e}\\{\|v\|_{\mathscr{F}_e}\leq1}}}l(v)=\|l\|_{\mathscr{F}_e^*}.
\end{eqnarray*}
Hence $\forall\ l\in \mathscr{F}_e^*$,
\begin{eqnarray}\label{eqnal5}
  \lim_{\nu\rightarrow0}\|l\|_{F^*_{1,2,\nu}}=\sup_{\nu>0}\|l\|_{F^*_{1,2,\nu}}\leq
\|l\|_{\mathscr{F}_e^*}.
\end{eqnarray}

To prove the converse inequality of \eqref{eqnal5}, fix $l\in \mathscr{F}_e^*$ and let
$\varepsilon,\ \delta\in(0,1)$. Then there exists $v_\varepsilon\in
F_{1,2}$ with $\|v_\varepsilon\|_{\mathscr{F}_e}=1$ and
$$l(v_\varepsilon)\geq \|l\|_{\mathscr{F}_e^*}-\varepsilon.$$
Let $\nu_0:=\frac{\delta^2}{1+|v_\varepsilon|_2^2}$. From \eqref{eqnal1}, we see that
\begin{eqnarray*}
  \|v_\varepsilon\|_{F_{1,2,\nu_0}}=\sqrt{\|v_\varepsilon\|^2_{\mathscr{F}_e}+\nu_0|v_\varepsilon|^2_2}\leq
\sqrt{1+\delta^2}\leq1+\delta,
\end{eqnarray*}
so for $\bar{v_\varepsilon}:=\frac{v_\varepsilon}{1+\delta}$, we
have
$$\|\bar{v_\varepsilon}\|_{F_{1,2,\nu_0}}\leq1.$$

Consequently,
\begin{eqnarray*}\label{eqnal3}
  \lim_{\nu\rightarrow0}\|l\|_{F^*_{1,2,\nu}}&&\!\!\!\!\!\!\!\!=\sup_{\nu>0}\|l\|_{F^*_{1,2,\nu}}\nonumber\\
  &&\!\!\!\!\!\!\!\!\geq\|l\|_{F^*_{1,2,\nu_0}}\geq l(\bar{v_\varepsilon})=\frac{1}{1+\delta}l(v_\varepsilon)\geq
  \frac{1}{1+\delta}(\|l\|_{\mathscr{F}^*_e}-\varepsilon),
\end{eqnarray*}
letting $\delta\rightarrow0$, $\varepsilon\rightarrow0$, yields the desired converse inequality. Hence \eqref{eqnal6} is proved.

It remains to prove the second inequality in \eqref{eqnal19}. Note that
$$\|v\|^2_{F_{1,2,\nu}}=\mathscr{E}(v,v)+\nu\langle v,v\rangle_2~~~\text{and}~~~\|v\|^2_{F_{1,2}}=\mathscr{E}(v,v)+\langle v,v\rangle_2,$$
hence for $0<\nu\leq1$,
\begin{eqnarray*}
\|l\|_{F^*_{1,2}}=\sup_{\|v\|_{F_{1,2}}\leq1}l(v)
=\sqrt{\nu}\sup_{\sqrt{\nu}\|v\|_{F_{1,2}}\leq1}l(v)
\geq\sqrt{\nu}\sup_{\|v\|_{F_{1,2,\nu}}\leq1}l(v)
=\sqrt{\nu}\|l\|_{F^*_{1,2,\nu}}.
\end{eqnarray*}
\end{proof}

\vspace{2mm}
\subsection{Gelfand triples}\label{Gelfandtriples}
Let $H$ be a separable Hilbert space with inner product
$\langle\cdot, \cdot\rangle_H$ and let $H^*$ be its dual space. Let $V$ be a
reflexive Banach space, such that $V\subset H$ continuously and
densely. Then for its dual space $V^*$ it follows that $H^*\subset
V^*$ continuously and densely. Identifying $H$ and $H^*$ via the
Riesz isomorphism we have that
\begin{eqnarray*}
V\subset H\subset V^*
\end{eqnarray*}
continuously and densely. Let $_{V^*}\langle\cdot, \cdot\rangle_V$
denote the dualization between $V^*$ and $V$ (i.e. $_{V^*}\langle
z,v \rangle_V:=z(v)$ for $z\in V^*$, $v\in V$). Then it follows that
\begin{eqnarray}\label{eqnarray48}
_{V^*}\langle z,v \rangle_V=\langle z,v\rangle_H,\ \text{for\ all}\
z\in H,\ v\in V.
\end{eqnarray}
$(V,H,V^*)$ is called a Gelfand triple.

\vspace{2mm}

A Gelfand triple with $V:=L^2(\mu)$, $H:=F^*_{1,2}$ was constructed in \cite{RWX}, the Riesz map which identifies $F_{1,2}$ and $F^*_{1,2}$ is $(1-L)^{-1}:F^*_{1,2}\rightarrow F_{1,2}$. We need the following two lemmas.
\begin{lemma}\label{rwx1}(\cite[Lemma 2.1.]{RWX})
The map $(1-L): F_{1,2}\rightarrow F^*_{1,2}$ is an isometric
isomorphism. In particular,
\begin{equation}\label{equation2}
\big\langle (1-L)u,~ (1-L)v \big\rangle_{F^*_{1,2}}=\langle u,~ v
\rangle_{F_{1,2}} \ \ \text{for all}\ \ u, v \in F_{1,2}.
\end{equation}
Furthermore, $(1-L)^{-1}: F^*_{1,2}\rightarrow F_{1,2}$ is the Riesz
isomorphism for $F^*_{1,2}$, i.e., for every $u \in F^*_{1,2}$,
\begin{equation}\label{equation3}
\langle u,~\cdot\rangle_{F^*_{1,2}}=\!\!_{F_{1,2}}\langle
(1-L)^{-1}u,~ \cdot \rangle_{F^*_{1,2}}.
\end{equation}
\end{lemma}

\begin{lemma}\label{rrw}(\cite[Lemma 3.3(i)]{RRW})
The map $\bar{L}:\mathscr{F}_e\rightarrow \mathscr{F}_e^*$
defined by
\begin{eqnarray*}\label{eqnarray35}
 \bar{L}v:=-\mathscr{E}(v,\cdot),\ v\in \mathscr{F}_e
\end{eqnarray*}
(i.e. the Riesz isomorphism of $\mathscr{F}_e$ and $\mathscr{F}_e^*$
multiplied by (-1)) is the unique continuous linear extension of the
map
\begin{eqnarray*}\label{eqnarray36}
  D(L)\ni v\mapsto \mu(Lv\cdot)\in \mathscr{F}_e^*.
\end{eqnarray*}
\end{lemma}

If there is no danger of confusion, we write $L$ instead of $\bar{L}$ below. For a Hilbert space $\mathbb{H}$, throughout the paper, let $L^2([0,T]\times\Omega;\mathbb{H})$ denote the space of all $\mathbb{H}$-valued square-integrable functions on
$[0,T]\times\Omega$, $C([0,T];\mathbb{H})$ the space of all
continuous $\mathbb{H}$-valued functions on $[0,T]$, and $L^\infty([0,T],\Bbb{H})$ the space of all $\Bbb{H}$-valued uniformly bounded measurable functions on $[0,T]$. For two
Hilbert spaces $H_1$ and $H_2$, the space of Hilbert-Schmidt
operators from $H_1$ to $H_2$ is denoted by $L_2(H_1, H_2)$. For
simplicity, the positive constants $C$, $C_i$, $i=1,...,5$, used in this paper may change from line to line. We would like
to refer to \cite{LR,PR} for more background information and results on SPDEs, \cite{ADG,V} on PMEs and \cite{BDR1} on
SPMEs.

\section{Assumptions and main results}\label{Assumptions}
\setcounter{equation}{0}
 \setcounter{definition}{0}

Let $K:=L^1(\mu)\cap L^\infty(\mu)\cap \mathscr{F}^*_e$. In addition to condition \textbf{(L.1)} above, we study Eq.\eqref{eq:1} under the following
assumptions.

\medskip
\noindent \textbf{(H1)} $\Psi(\cdot): \mathbb{R}\rightarrow\!
2^\mathbb{R}$ is a maximal monotone graph such that $0\in \Psi(0)$ and there exist $C\in(0,\infty)$ and $m\in[1,\infty)$ such that
\begin{eqnarray}\label{eqnarray1}
   &&\inf\{|\eta|;\eta \in \Psi(r)\}\leq C(|r|^m+1_{\{\mu(E)<\infty\}}),\ \ \forall~ r\in
\mathbb{R}.
\end{eqnarray}

\vspace{1mm}

\noindent \textbf{(H2)} $B\!: [0, T]\times K\times
\Omega\rightarrow L_2(L^2(\mu), L^2(\mu))$ is progressively
measurable (i.e. for any $t\in[0,T]$, this mapping restricted to
$[0,t]\times K\times \Omega$ is measurable w.r.t.
$\mathscr{B}([0,t])\times\mathscr{B}(K)\times \mathscr{F}_t$,
where $\mathscr{B}(\cdot)$ is the Borel $\sigma$-field for a
topological space). Furthermore, $B(t,u)$ satisfies the following conditions:
\noindent \textbf{(i)}~~ There exists $C_1\in[0, \infty)$ such that
for all $\nu\in(0,1]$,
$$\|B(\cdot, u)-B(\cdot, v)\|_{L_2(L^2(\mu), F^*_{1,2,\nu})}\leq C_1\|u-v\|_{ F^*_{1,2,\nu}},~~\forall u, v\in K\ \text{on}\ [0, T]\times \Omega,$$
where for simplicity, we write $B(t,u)$ meaning the mapping $\omega\mapsto B(t,u,\omega)$.

\noindent \textbf{(ii)}~~ There exists $C_2\in(0, \infty)$ such that
for all $\nu\in(0,1]$,
$$\|B(\cdot, u)\|_{L_2(L^2(\mu), F^*_{1,2,\nu})}\leq C_2\|u\|_{ F^*_{1,2,\nu}},~~\forall u\in K\ \text{on}\ [0, T]\times \Omega.$$

\vspace{1mm}
\medskip
\noindent \textbf{(H3)}\noindent \textbf{(i)}~~There exists $C_3\in(0, \infty)$ satisfying
$$\|B(\cdot, u)\|_{L_2(L^2(\mu), L^2(\mu))}\leq C_3|u|_2,~~\forall u\in K\ \text{on}\ [0, T]\times \Omega.$$
\noindent \textbf{(ii)}~~Let $m$ be as in \eqref{eqnarray1}, there exist an orthonormal basis $\{e_k\}_{k\geq1}$ of $L^2(\mu)$ and $C_4\in(0, \infty)$ satisfying
$$\int_E\big(\sum_{k=1}^{\infty}|B(\cdot,u)e_k|^2\big)^md\mu\leq C_4|u|_{2m}^{2m},~~\forall u\in K\ \text{on}\ [0, T]\times\Omega.$$

\vspace{2mm}
\noindent \textbf{(H4)}~~There exists a symmetric, positive, bilinear mapping $\Gamma:F_{1,2}\times F_{1,2}\rightarrow L^1(\mu)$ satisfying:

 \noindent \textbf{(i)}
\begin{eqnarray*}
  \mathscr{E}(u,u)=\int \frac{1}{2}\Gamma(u,u)d\mu,~~\forall u\in
F_{1,2}.
\end{eqnarray*}
\vspace{2mm}

\noindent \textbf{(ii)}
There exists a constant $C_5\in(0,\infty)$ such that
$$\Gamma(\varphi(u),\varphi(u))\leq C_5\Gamma(u,\varphi(u)),~\forall u\in F_{1,2},$$
for every non-decreasing Lipschitz function $\varphi:\Bbb{R}\rightarrow\Bbb{R}$ with $\varphi(0)=0$.
\begin{remark}\label{mainremark}
\textbf{(i)} \eqref{eqnal6} and \textbf{(H2)(i)} imply that for all $u,v\in K$,
\begin{eqnarray}\label{eqnal0}
\|B(\cdot, u)-B(\cdot, v)\|^2_{L_2(L^2(\mu), \mathscr{F}^*_e)}\leq C_1\|u-v\|^2_{\mathscr{F}^*_e}\ \ \text{on}\ [0, T]\times \Omega.
\end{eqnarray}

\textbf{(ii)} We emphasize that \textbf{(H4)(ii)} is automatically fulfilled, if $(\mathcal{E},D(\mathcal{E}))$ is a local Dirichlet form.

\textbf{(iii)} By \textbf{(L.1)} there exists $g\in L^1(\mu)\cap L^\infty(\mu)$, $g>0$, $\mu$-a.e., such that $\mathscr{F}_e\subset L^1(g\cdot \mu)$ continuously and it was proved in \cite{RRW} (see the last part of the proof of Proposition 3.1 in \cite{RRW}) that the linear space
$$\mathscr{G}:=\{h\cdot g|h\in L^\infty(\mu)\}$$
is dense in $\mathscr{F}^*_e$. Furthermore, obviously $\mathscr{G}\subset L^1(\mu)\cap L^\infty(\mu)$. Hence it follows that $K$ (defined in \textbf{(H2)}) is dense in $\mathscr{F}_e^*$, and hence in $F^*_{1,2,\nu_0}$ for every $\nu_0\in(0,1]$. Here $F^*_{1,2,\nu_0}$ denotes the space $F^*_{1,2}$ equipped with the norm $\|\cdot\|_{F^*_{1,2,\nu_0}}$. Therefore, by \textbf{(H2)(i)} the map
\begin{eqnarray*}
K\ni u\longrightarrow B(t,u)\in L_2(L^2(\mu), F^*_{1,2,\nu_0})
\end{eqnarray*}
can be extended uniquely to a Lipschitz continuous map on all of $F^*_{1,2,\nu_0}$. Furthermore, \textbf{(H2)(ii)} trivially also holds for this extension, as well as \eqref{eqnal0}. We shall use this extension below without further notice.

\textbf{(iv)} From \cite[Proposition 3.1]{RRW}, we also know that $\mathscr{G}$ is dense in $L^p(\mu)$ ($1<p<\infty$), since $\mathscr{G}\subset L^1(\mu)\cap L^\infty(\mu)$, $\mathscr{G}\subset\mathcal{F}_e^*$, $L^1(\mu)\cap L^\infty(\mu)$ is dense in $L^p(\mu)$, hence $K$ is dense in $L^p(\mu)$, so \textbf{(H3)(i)} is also true for all $u\in L^2(\mu)$ and \textbf{(H3)(ii)} is true for all $u\in L^{2m}(\mu)$. Note that if $m=1$, then \textbf{(H3)(i)} and \textbf{(H3)(ii)} are the same.
\end{remark}

%definition

\begin{definition}\label{definition1}
Let $x\in \mathscr{F}_e^*$. An
$\mathscr{F}_e^*$-valued adapted process $X=X(t)$ is called strong
solution to \eqref{eq:1} if there exists $q\in [2,\infty)$ such that the following conditions hold:
\begin{eqnarray*}
   &&X\ \text{is}\ \mathscr{F}_e^*-{valued\ continuous\ on}\ [0,T],\
\mathbb{P}-a.s.;  \\\label{eqnarray7}
   &&X\in L^q(\Omega \times (0,T)\times E);\label{eqnarray8}
\end{eqnarray*}
there is $\eta \in L^{\frac{q}{m}}(\Omega\times (0,T)\times E)$ such
that
\begin{eqnarray*}\label{eqnarray9}
&&\eta \in \Psi(X),\ dt \otimes \mathbb{P} \otimes d\mu-a.e.\
\text{on}\ \Omega \times (0,T)\times E;
\end{eqnarray*}
and $\mathbb{P}$-$a.s$.,
\begin{eqnarray}\label{eqnarray3.2}
\int_0^\cdot \eta(s)ds\in C([0,T];\mathscr{F}_e),
\end{eqnarray}
\begin{equation*}\label{equa4}
X(t)=x+L\int_0^t\eta(s)ds+\int_0^tB(s,X(s))dW(s)\ \ \text{for\ all}\
t\in [0,T].
\end{equation*}
\end{definition}

Theorem \ref{theorem1} below is the main existence and uniqueness result for Eq.\eqref{eq:1}.

%Theorem 3.3

\begin{theorem}\label{theorem1}
Assume that \textbf{(L.1)}, \textbf{(H1)}-\textbf{(H4)} are satisfied and let m be as in \eqref{eqnarray1}. Let $x\!\in L^2(\mu)\cap L^{2m}(\mu)\cap
\mathscr{F}_e^*$. Then there is a unique strong solution $X$
to \eqref{eq:1} such that
\begin{eqnarray*}\label{eqnarray10}
  &&X\in L^2\big(\Omega; C([0,T];\mathscr{F}_e^*)\big)\cap L^{\infty}\big([0,T];(L^2\cap L^{2m})(\Omega\times E)\big).
\end{eqnarray*}
\end{theorem}

Theorem \ref{theorem1} will be proved in Section \ref{ProofofTheorem3.1}. The proof is based on an approximating equation of \eqref{eq:1}. More precisely, in Section \ref{ProofofTheorem3.2} we shall establish the existence of solutions for the following Yosida approximating equation of \eqref{eq:1}
\begin{equation}\label{eq:3}
\left\{ \begin{aligned}
&dX_\lambda-L(\Psi_\lambda(X_\lambda)+\lambda X_\lambda)dt=B(t,X_\lambda)dW(t),\ \ t\in[0,T],\\
&X_\lambda(0)=x \ \text{on} \ E.
\end{aligned} \right.
\end{equation}
Here $\lambda>0$ and
\begin{eqnarray*}
\Psi_\lambda(x)=\frac{1}{\lambda}\big(x-(1+\lambda\Psi)^{-1}(x)\big)\in
\Psi\big((1+\lambda\Psi)^{-1}(x)\big)
\end{eqnarray*}
is the Yosida approximation of $\Psi$. We recall that (see \cite[page:38, Proposition 2.2]{B}) $\Psi_\lambda$ is single-valued, monotone and for all $r\in\Bbb{R}$
\begin{eqnarray}\label{eqnal01}
|\Psi_\lambda(r)|\leq \inf|\Psi(r)|.
\end{eqnarray}

We have the following result for Eq.\eqref{eq:3}.
\begin{theorem}\label{Theorem}
Assume that \textbf{(L.1)}, \textbf{(H1)}-\textbf{(H4)} are satisfied and let $m$ be as in \eqref{eqnarray1}. Let $\lambda\in(0,1)$, and $x\in L^2(\mu)\cap L^{2m}(\mu)\cap \mathscr{F}_e^*$. Then \eqref{eq:3} has a strong solution
\begin{eqnarray}\label{eqnarray11}
X_\lambda\in L^2(\Omega; C([0,T];\mathscr{F}_e^*))\cap L^{\infty}\big([0,T];(L^2\cap L^{2m})(\Omega\times E)\big),
\end{eqnarray}
satisfying
\begin{equation}\label{equ:2.2}
\int_0^\cdot \Psi_\lambda(X_\lambda(s))+\lambda X_\lambda(s)ds\in C([0,T];\mathscr{F}_e)~
\Bbb{P}\text{-a.s.},
\end{equation}
and $\Bbb{P}$-a.s.,
\begin{equation}\label{equ:2.3}
X_\lambda(t)=x+L\int_0^t\Psi_\lambda(X_\lambda(s))+\lambda X_\lambda(s)ds+\int_0^tB(s,X_\lambda(s))dW(s),\ \forall t\in[0,T].
\end{equation}
Moreover, there exists $C\in(0,\infty)$ such that for all $\lambda,\ \lambda'\in(0,1)$, $t\in[0,T]$,
\begin{eqnarray}
&&\mathbb{E}|X_\lambda(t)|_{2m}^{2m}\leq C|x|_{2m}^{2m},
 \label{eqnarray12}\\
   && \mathbb{E}\int_0^T
    \int_E|\Psi_\lambda(X_\lambda(t))|^2d\mu dt\leq
    C(|x|_{2m}^{2m}+\mu(E)\cdot1_{\{\mu(E)<\infty\}}), \label{eqnarray13}\\
        && \mathbb{E}\Big[\sup_{0\leq t\leq
    T}\|X_\lambda(t)\|^2_{\mathscr{F}_e^*}\Big]\leq
    C\big(\|x\|^2_{\mathscr{F}_e^*}+|x|_{2m}^{2m}+|x|_2^2+\mu(E)\cdot 1_{\{\mu(E)<\infty\}}\big),\label{eqnarray15}\\
  && \mathbb{E}\Big[\sup_{0\leq t\leq
T}\|X_\lambda(t)-X_{\lambda'}(t)\|^2_{\mathscr{F}_e^*}\Big]\leq
C(\lambda+\lambda')(|x|_2^2+|x|_{2m}^{2m}+\mu(E)\cdot1_{\{\mu(E)<\infty\}}).\label{eqnarray14}
\end{eqnarray}
\end{theorem}

\section{Proof of Theorem \ref{Theorem}}\label{ProofofTheorem3.2}
\setcounter{equation}{0}
 \setcounter{definition}{0}

\begin{proof}
For each fixed $\lambda$, firstly we consider the following approximating equation for \eqref{eq:3}
\begin{equation}\label{eq:4}
\left\{ \begin{aligned}
&dX^\nu_\lambda(t)+(\nu-L)\big(\Psi_\lambda(X^\nu_\lambda(t))+\lambda X^\nu_\lambda(t)\big)dt=B(t,X^\nu_\lambda(t))dW(t),\ \text{in}\ (0,T)\times E,\\
&X^\nu_\lambda(0)=x\in  L^2(\mu)\cap L^{2m}(\mu),
\end{aligned} \right.
\end{equation}
where $\nu\in(0,1]$. Since $\Psi_\lambda+\lambda I$ is Lipschitz (where here and below $I$ denotes the identity
map on the respective space), by \cite[Lemma 3.1]{RWX}, \eqref{eq:4} has a unique $(\mathscr{F}_t)_{t\geq0}$-adapted solution in the sense that $X^\nu_\lambda\in L^2([0,T]\times \Omega;L^2(\mu))\cap L^2(\Omega;C([0,T];F^*_{1,2}))$,
\begin{eqnarray*}
X^\nu_\lambda(t)+(\nu-L)\int_0^T\Psi_\lambda(X^\nu_\lambda(s))+\lambda X^\nu_\lambda(s)ds=x+\int_0^TB(s,X^\nu_\lambda(s))dW(s),\ \text{holds\ in}\ F^*_{1,2},\ \Bbb{P}-a.s.,
\end{eqnarray*}
and there exists a positive constant $C\in(0,\infty)$ such that for all $\lambda\in(0,1)$, $\nu\in(0,1]$,
\begin{eqnarray*}
\Bbb{E}\Big[\sup_{t\in[0,T]}|X^\nu_\lambda(t)|_2^2\Big]\leq e^{CT}|x|_2^2.
\end{eqnarray*}

To prove that \eqref{eqnarray11}-\eqref{eqnarray15} hold with $X^\nu_\lambda$ replacing $X_\lambda$, with a constant $C$ independent of $\nu$ and $\lambda$, we consider the following approximating equation for \eqref{eq:4}.
\begin{equation} \label{eq:2}
\left\{ \begin{aligned}
&dX^{\nu,\varepsilon}_\lambda(t)+A^{\nu,\varepsilon}_\lambda(X^{\nu,\varepsilon}_\lambda(t))dt=B(t, X^{\nu,\varepsilon}_\lambda(t))dW(t),\ \text{in}\ (0,T)\times E,\\
&X^{\nu,\varepsilon}_\lambda(0)=x\in L^2(\mu)\cap L^{2m}(\mu),
\end{aligned} \right.
\end{equation}
where $A^{\nu,\varepsilon}_\lambda:F^*_{1,2}\rightarrow F^*_{1,2}$, defined by
\begin{eqnarray*}
A^{\nu,\varepsilon}_\lambda(x)=\frac{1}{\varepsilon}\big(x-(I+\varepsilon
A^\nu_\lambda)^{-1}(x)\big),~ x\in F^*_{1,2},~ \varepsilon\in (0,1),
\end{eqnarray*}
is the Yosida approximation of the operator $A^\nu_\lambda
(x):=(\nu-L)(\Psi_\lambda(x)+\lambda I(x))$ on $F^*_{1,2}$, $x\in
D(A^\nu_\lambda):=F_{1,2}$. Clearly, $I+\varepsilon A^\nu_\lambda:F_{1,2}\rightarrow F^*_{1,2}$ is a bijection, since so is $\Psi_\lambda+\lambda I:F_{1,2}\rightarrow F_{1,2}$. Furthermore, since by \eqref{eqnarray31} below, $A^\nu_{\lambda}$ with domain $F_{1,2}$ is monotone on $F^*_{1,2,\nu}$, it follows that $A^\nu_\lambda$ is maximal monotone on $F^*_{1,2,\nu}$. Fix $x\in F^*_{1,2}$ and set
$y:=J_\varepsilon(x):=(I+\varepsilon A^\nu_\lambda)^{-1}x\in
F_{1,2}$, i.e., $(I+\varepsilon A^\nu_\lambda)(y)=x$, equivalently,
\begin{eqnarray}\label{eqnarray18}
   &&y+\varepsilon(\nu-L)(\Psi_\lambda+\lambda I)(y)=x.
\end{eqnarray}
In particular, $(\Psi_\lambda+\lambda I)(y)\in D(L)$, if $x\in L^2(\mu)$.

%lemma 4.1

\vspace{1mm}
\medskip
Before giving the well-posedness result for \eqref{eq:2}, we need some preparations.
\begin{lemma}\label{lemma1}
Set
$$J_\varepsilon(x):=(I+\varepsilon A^\nu_\lambda)^{-1}x,\ \forall x\in F^*_{1,2}.$$
For all $0<\varepsilon<1$, we have
\begin{eqnarray}
&&\|J_\varepsilon (x)-J_\varepsilon
(\widetilde{x})\|_{F^*_{1,2,\nu}}\leq \|x-\widetilde{x}\|_{F^*_{1,2,\nu}},\forall x,\widetilde{x}\in F^*_{1,2}.\label{eqnarray41}\\
&&|J_\varepsilon (x)-J_\varepsilon (\widetilde{x})|_2\leq\frac{1}{\sqrt{\nu\varepsilon\lambda}}|x-\widetilde{x}|_2,~\forall x,\widetilde{x}\in L^2(\mu). \label{eqnarray20}\\
&&|J_\varepsilon (x)|_p\leq |x|_p,\ \ \forall x\in L^p(\mu)\cap
L^2(\mu),~ 2\leq p<\infty. \label{eqnarray19}
\end{eqnarray}
\end{lemma}
\begin{proof}
Firstly, let us prove \eqref{eqnarray41}. For $x,\ \tilde{x}\in
F^*_{1,2}$, set $y:=J_\varepsilon(x)$ and
$\widetilde{y}:=J_\varepsilon(\widetilde{x})$, we have
\begin{eqnarray*}\label{eqna2}
   &&y-\tilde{y}+\varepsilon A^\nu_\lambda(y)-\varepsilon
A^\nu_\lambda(\tilde{y})=x-\tilde{x}.
\end{eqnarray*}
Taking the scalar product of $y-\tilde{y}$ with both sides in $(F^*_{1,2},
\|\cdot\|_{F^*_{1,2,\nu}})$, we get
\begin{eqnarray}\label{eqnarray30}
   &&\langle y-\tilde{y},y-\tilde{y}\rangle_{F^*_{1,2,\nu}}+\varepsilon \langle A^\nu_\lambda(y)-
A^\nu_\lambda(\tilde{y}), y-\tilde{y}\rangle_{F^*_{1,2,\nu}}=\langle
x-\tilde{x},y-\tilde{y}\rangle_{F^*_{1,2,\nu}}.
\end{eqnarray}
For the second term in the left hand-side of \eqref{eqnarray30}, by \eqref{equation3}, we know
\begin{eqnarray}\label{eqnarray31}
&&\!\!\!\!\!\!\!\!\langle A^\nu_\lambda(y)-A^\nu_\lambda(\widetilde{y}),y-\widetilde{y}\rangle_{F^*_{1,2,\nu}}\nonumber\\
=&&\!\!\!\!\!\!\!\!\big\langle(\nu-L)((\Psi_\lambda+\lambda I)(y)-(\Psi_\lambda+\lambda I)(\widetilde{y})),y-\widetilde{y}\big\rangle_{F^*_{1,2,\nu}}\nonumber\\
=&&\!\!\!\!\!\!\!\!_{F_{1,2}}\big\langle(\Psi_\lambda+\lambda I)(y)-(\Psi_\lambda+\lambda I)(\widetilde{y}),y-\widetilde{y}\big\rangle_{F^*_{1,2}}\nonumber\\
=&&\!\!\!\!\!\!\!\!\big\langle(\Psi_\lambda+\lambda I)(y)-(\Psi_\lambda+\lambda I)(\widetilde{y}),y-\widetilde{y}\big\rangle_2\geq0,
\end{eqnarray}
since $y-\tilde{y}\in F_{1,2}\subset L^2(\mu)$. \\
\eqref{eqnarray30} and \eqref{eqnarray31} imply
\begin{eqnarray*}\label{eqna1}
  \|y-\tilde{y}\|^2_{F^*_{1,2,\nu}}\leq
\|x-\tilde{x}\|_{F^*_{1,2,\nu}}\cdot\|y-\tilde{y}\|_{F^*_{1,2,\nu}},
\end{eqnarray*}
from which \eqref{eqnarray41} follows.

Secondly, to prove the Lipschitz continuity of $J_\varepsilon$ in $L^2(\mu)$, we take $x, \tilde{x}\in L^2(\mu)$ and apply
$$_{F^*_{1,2}}\big\langle~ \cdot,~
(\Psi_\lambda+\lambda I)(y)-(\Psi_\lambda+\lambda I)(\tilde{y})\big\rangle_{F_{1,2}}$$
to both sides of \eqref{eqna2}. Then
\begin{eqnarray}\label{eqna4}
&&\!\!\!\!\!\!\!\!_{F^*_{1,2}}\big\langle y-\tilde{y},
(\Psi_\lambda+\lambda I)(y)-(\Psi_\lambda+\lambda I)(\tilde{y})\big\rangle_{F_{1,2}}\nonumber\\
&&\!\!\!\!\!\!\!\!+ _{F^*_{1,2}}\big\langle
\varepsilon A^\nu_\lambda(y)-\varepsilon A^\nu_\lambda(\tilde{y}),(\Psi_\lambda+\lambda I)(y)-(\Psi_\lambda+\lambda I)(\tilde{y})\big\rangle_{F_{1,2}}\nonumber\\
=&&\!\!\!\!\!\!\!\!_{F^*_{1,2}}\big\langle x-\tilde{x},
(\Psi_\lambda+\lambda I)(y)-(\Psi_\lambda+\lambda I)(\tilde{y})\big\rangle_{F_{1,2}}.
\end{eqnarray}
For the second term in the left hand-side of \eqref{eqna4}, by \eqref{eqnarray48}-\eqref{equation3} (under the Gelfand triple $F_{1,2}\subset L^2(\mu)\subset F^*_{1,2}$), we obtain
\begin{eqnarray}\label{eqna3}
&&\!\!\!\!\!\!\!\!_{F^*_{1,2}}\big\langle \varepsilon A^\nu_\lambda(y)-\varepsilon A^\nu_\lambda(\widetilde{y}),(\Psi_\lambda+\lambda I)(y)-(\Psi_\lambda+\lambda I)(\widetilde{y})\big\rangle_{F_{1,2}}\nonumber\\
=&&\!\!\!\!\!\!\!\!_{F^*_{1,2}}\big\langle(1-L)(\varepsilon(\Psi_\lambda+\lambda I)(y)-\varepsilon(\Psi_\lambda+\lambda I)(\widetilde{y})),(\Psi_\lambda+\lambda I)(y)-(\Psi_\lambda+\lambda I)(\widetilde{y})\big\rangle_{F_{1,2}}\nonumber\\
&&\!\!\!\!\!\!\!\!+_{F^*_{1,2}}\big\langle\varepsilon(\nu-1)((\Psi_\lambda+\lambda I)(y)-(\Psi_\lambda+\lambda)(\widetilde{y})),(\Psi_\lambda+\lambda I)(y)-(\Psi_\lambda+\lambda I)(\widetilde{y})\big\rangle_{F_{1,2}}\nonumber\\
=&&\!\!\!\!\!\!\!\!\varepsilon\|(\Psi_\lambda+\lambda I)(y)-(\Psi_\lambda+\lambda I)(\widetilde{y})\|_{F_{1,2}}^2+\varepsilon(\nu-1)|(\Psi_\lambda+\lambda I)(y)-(\Psi_\lambda+\lambda I)(\widetilde{y})|_2^2\nonumber\\
\geq&&\!\!\!\!\!\!\!\!\nu\varepsilon|(\Psi_\lambda+\lambda I)(y)-(\Psi_\lambda+\lambda I)(\widetilde{y})|^2_2.
\end{eqnarray}
For the first term in the left hand-side of \eqref{eqna4}, since $\Psi_\lambda$ is
monotone, by \eqref{eqnarray48} (under the Gelfand triple $F_{1,2}\subset L^2(\mu)\subset F^*_{1,2}$), we know
\begin{eqnarray}\label{eqnarray64}
&&\!\!\!\!\!\!\!\!_{F^*_{1,2}}\big\langle y-\tilde{y},
(\Psi_\lambda+\lambda I)(y)-(\Psi_\lambda+\lambda I)(\tilde{y})\big\rangle_{F_{1,2}}\nonumber\\
&&\!\!\!\!\!\!\!\!=\big\langle y-\tilde{y},
(\Psi_\lambda+\lambda I)(y)-(\Psi_\lambda+\lambda I)(\tilde{y})\big\rangle_2\nonumber\\
&&\!\!\!\!\!\!\!\!\geq\lambda|y-\tilde{y}|^2_2.
\end{eqnarray}
Similarly, since $x,\widetilde{x}\in L^2(\mu)$, by \eqref{eqnarray48}, we have
\begin{eqnarray}\label{eqnarray64.1}
&&\!\!\!\!\!\!\!\!_{F^*_{1,2}}\langle x-\widetilde{x},(\Psi_\lambda+\lambda I)(y)-(\Psi_\lambda+\lambda I)(\widetilde{y})\rangle_{F_{1,2}}\nonumber\\
=&&\!\!\!\!\!\!\!\!\langle x-\widetilde{x},(\Psi_\lambda+\lambda I)(y)-(\Psi_\lambda+\lambda I)(\widetilde{y})\rangle_2.
\end{eqnarray}
Taking \eqref{eqna3}, \eqref{eqnarray64} and \eqref{eqnarray64.1} into \eqref{eqna4}, by Young's inequality, we obtain
\begin{eqnarray*}\label{eqna13}
&&\!\!\!\!\!\!\!\!\lambda|y-\widetilde{y}|_2^2+\nu\varepsilon\big|(\Psi_\lambda+\lambda I)(y)-(\Psi_\lambda+\lambda I)(\widetilde{y})\big|^2_2\nonumber\\
\leq&&\!\!\!\!\!\!\!\!|x-\widetilde{x}|_2\cdot\big|(\Psi_\lambda+\lambda I)(y)-(\Psi_\lambda+\lambda I)(\widetilde{y})\big|_2\nonumber\\
\leq&&\!\!\!\!\!\!\!\!\frac{1}{\nu\varepsilon}|x-\widetilde{x}|_2^2+\nu\varepsilon\big|(\Psi_\lambda+\lambda I)(y)-(\Psi_\lambda+\lambda I)(\widetilde{y})\big|^2_2,
\end{eqnarray*}
and therefore
\begin{eqnarray*}
&&|y-\widetilde{y}|_2^2\leq
\frac{1}{\nu\varepsilon\lambda}|x-\widetilde{x}|_2^2,
\end{eqnarray*}
which yields \eqref{eqnarray20} as claimed.

Now, let us prove \eqref{eqnarray19}. Let $x\in L^2(\mu)\cap L^p(\mu)$, $p\geq2$. Since the function $h(r):=r|r|^{p-2}(1+k|r|^{p-2})^{-1}$ is Lipschitz, and $h(0)=0$, we have $h(y)\in F_{1,2}$, because $y\in F_{1,2}$. Hence applying
$_{F^*_{1,2}}\big\langle~\cdot,~
y|y|^{p-2}(1+k|y|^{p-2})^{-1}\big\rangle_{F_{1,2}}, ~k>0$, to both sides of
\eqref{eqnarray18}, we obtain
\begin{eqnarray}\label{eqnarray21}
   &&\!\!\!\!\!\!\!\!_{F^*_{1,2}}\big\langle y, \frac{y|y|^{p-2}}{1+k|y|^{p-2}}\big\rangle_{F_{1,2}}+_{F^*_{1,2}}\big\langle\varepsilon(\nu-L)(\Psi_\lambda(y)+\lambda
y),\frac{y|y|^{p-2}}{1+k|y|^{p-2}}\big\rangle_{F_{1,2}}\nonumber\\
=&&\!\!\!\!\!\!\!\!_{F^*_{1,2}}\big\langle x,
\frac{y|y|^{p-2}}{1+k|y|^{p-2}}\big\rangle_{F_{1,2}}.
\end{eqnarray}
Under the Gelfand triple $F_{1,2}\subset L^2(\mu)\subset F^*_{1,2}$, by \eqref{eqnarray48}, \eqref{eqnarray21} yields
\begin{eqnarray}\label{eqnarray47}
   &&\!\!\!\!\!\!\!\!\big\langle y, \frac{y|y|^{p-2}}{1+k|y|^{p-2}}\big\rangle_2+_{F^*_{1,2}}\big\langle\varepsilon(\nu-L)(\Psi_\lambda(y)+\lambda
y),\frac{y|y|^{p-2}}{1+k|y|^{p-2}}\big\rangle_{F_{1,2}}\nonumber\\
=&&\!\!\!\!\!\!\!\!\big\langle x,
\frac{y|y|^{p-2}}{1+k|y|^{p-2}}\big\rangle_2.
\end{eqnarray}
For the second term in the left hand-side of \eqref{eqnarray47}, since $x\in L^2(\mu)$, $y\in F_{1,2}\subset L^2(\mu)$, from \eqref{eqnarray18} we deduce that
\begin{eqnarray*}
(\nu-L)(\Psi_\lambda(y)+\lambda y)\in L^2(\mu).
\end{eqnarray*}
Then by \eqref{eqnarray48}, we know
\begin{eqnarray*}
_{F^*_{1,2}}\big\langle\varepsilon(\nu-L)(\Psi_\lambda(y)+\lambda
y),\frac{y|y|^{p-2}}{1+k|y|^{p-2}}\big\rangle_{F_{1,2}}=\big\langle\varepsilon(\nu-L)(\Psi_\lambda(y)+\lambda
y),\frac{y|y|^{p-2}}{1+k|y|^{p-2}}\big\rangle_2.
\end{eqnarray*}
To estimate the term above, notice that for all Lipschitz and increasing function $g:\mathbb{R}\rightarrow
\mathbb{R}$ with $g(0)=0$, we have
\begin{eqnarray*}\label{eqnarray22}
   && \int_E(\nu-L)\big(\Psi_\lambda(y)+\lambda
y\big)\cdot g(y)d\mu\geq 0,
\end{eqnarray*}
because on one hand, $\Psi_{\lambda}$ is Lipschitz and monotone with $\Psi_\lambda(0)=0$, then obviously,
\begin{eqnarray*}\label{eqnarray23}
  &&\int_E\nu\big(\Psi_\lambda(y)+\lambda
y\big)\cdot g(y)d\mu\geq 0.
\end{eqnarray*}
On the other hand, we can prove the following term, i.e.,
\begin{eqnarray}\label{eqnarray24}
   &&\!\!\!\!\!\!\!\! \big\langle(-L)(\Psi_\lambda(y)+\lambda y), g(y)\big\rangle\nonumber\\
=&&\!\!\!\!\!\!\!\!\mathscr{E}\big(\Psi_\lambda(y)+\lambda y, g(y)\big)\nonumber\\
=&&\!\!\!\!\!\!\!\!\lim_{\varepsilon\rightarrow0}\mathscr{E}^{(\varepsilon)}\big(\Psi_\lambda(y)+\lambda y,
g(y)\big),
\end{eqnarray}
is non-negative. Indeed, by \cite[Lemma 5.1]{RW}, with $\kappa$ being the kernel corresponding to
$P:=(I-\varepsilon L)^{-1}$, we know, setting $f:=\Psi_\lambda+\lambda I$,
\begin{eqnarray*}\label{eqnarray25}
   \mathscr{E}^{(\varepsilon)}\big(f(y), g(y)\big):=&&\!\!\!\!\!\!\!\! \frac{1}{\varepsilon} \big\langle f(y), \big(I-(I-\varepsilon L)^{-1}\big)g(y)\big\rangle_2\nonumber\\
   =&&\!\!\!\!\!\!\!\! \frac{1}{2\varepsilon}\int_E\int_E\big((f(y(\widetilde{\xi})))-f(y(\xi))\big)\cdot\big(g(y(\tilde{\xi}))-g(y(\xi))\big)\kappa(\xi, d\tilde{\xi})\mu (d\xi)\nonumber \\
   &&\!\!\!\!\!\!\!\!+\frac{1}{\varepsilon}\int_E(1-P1(\xi))f(y(\xi))g(y(\xi))\mu(d\xi),
\end{eqnarray*}
since $f, g$ are monotone with $f(0)=g(0)=0$ and $P1\leq1$, we deduce that
\begin{eqnarray*}\label{eqnarray26}
  &&\mathscr{E}^{(\varepsilon)}\big(f(y), g(y)\big)\geq 0,
\end{eqnarray*}
which implies that \eqref{eqnarray24} is non-negative. As a short remark, the assumption that $(E,\mathcal{B})$ is a standard measurable space is needed in \cite[Lemma 5.1]{RW} to ensure the existence of the kernel $\kappa$ above.

Thus,
\begin{eqnarray*}\label{eqnarray27}
  &&\int_E \frac{|y|^{p}}{1+k|y|^{p-2}}d\mu\leq \int_E  \frac{xy|y|^{p-2}}{1+k|y|^{p-2}}d\mu.
\end{eqnarray*}
Letting $k\rightarrow 0$ and by H\"{o}lder's inequality, we obtain
\begin{eqnarray*}\label{eqnarray28}
  &&\!\!\!\!\!\!\!\!|y|_p^p\leq \int_E xy|y|^{p-2}d\mu \leq
|x|_p|y|^{p-1}_p.
\end{eqnarray*}
Hence, since $y= J_\varepsilon(x)$,
\begin{eqnarray*}
  &&|J_\varepsilon (x)|_p\leq |x|_p.
\end{eqnarray*}

\end{proof}

\vspace{2mm}

As shown in Lemma \ref{lemma1}, $J_\varepsilon$ is Lipschitz in both $L^2(\mu)$ and $F^*_{1,2}$. Since $A^{\nu,\varepsilon}_\lambda=\frac{1}{\varepsilon}(I-J_\varepsilon)$, $A^{\nu,\varepsilon}_\lambda$ is also Lipschitz in $L^2(\mu)$ and $F^*_{1,2}$. If $x\in F^*_{1,2}$, \eqref{eq:2} has a unique
adapted solution $X^{\nu,\varepsilon}_\lambda\in
L^2(\Omega;C([0,T];F^*_{1,2}))$ and by It\^{o}'s formula (see e.g. \cite[Theorem 4.2.5]{LR}) we have
\begin{eqnarray*}
  &&\!\!\!\!\!\!\!\!\mathbb{E}\|X^{\nu,\varepsilon}_\lambda(t)\|^2_{F^*_{1,2,\nu}}+2\mathbb{E}\int_0^t\big\langle A^{\nu,\varepsilon}_\lambda(X^{\nu,\varepsilon}_\lambda(s)),X^{\nu,\varepsilon}_\lambda(s)\big\rangle_{F^*_{1,2,\nu}}ds\nonumber\\
  =&&\!\!\!\!\!\!\!\!\|x\|^2_{F^*_{1,2,\nu}}+\mathbb{E}\int_0^t\|B(s,X^{\nu,\varepsilon}_\lambda(s))\|^2_{L_2(L^2(\mu),F^*_{1,2,\nu})}ds,~~t\in[0,T],
\end{eqnarray*}
which, by virtue of \textbf{(H2)(ii)} and the fact that the second term on the left hand-side is nonnegative by \eqref{eqnarray41}, yields
\begin{eqnarray*}
  \mathbb{E}\|X^{\nu,\varepsilon}_\lambda(t)\|^2_{F^*_{1,2,\nu}}\leq
e^{C_2T}\|x\|^2_{F^*_{1,2,\nu}},~ \forall \varepsilon>0,~t\in[0,T],~
x\in F^*_{1,2}.
\end{eqnarray*}
Similarly, if $x\in L^2(\mu)$, we know that $X^{\nu,\varepsilon}_\lambda\in
L^2(\Omega;C([0,T];L^2(\mu)))$ and again by It\^{o}'s formula we obtain
\begin{eqnarray*}
  &&\!\!\!\!\!\!\!\!\mathbb{E}|X^{\nu,\varepsilon}_\lambda(t)|^2_2+2\mathbb{E}\int_0^t\big\langle A^{\nu,\varepsilon}_\lambda(X^{\nu,\varepsilon}_\lambda(s)),X^{\nu,\varepsilon}_\lambda(s)\big\rangle_2ds\nonumber\\
  =&&\!\!\!\!\!\!\!\!|x|^2_2+\mathbb{E}\int_0^t\|B(s,X^{\nu,\varepsilon}_\lambda(s))\|^2_{L_2(L^2(\mu),L^2(\mu))}ds,
\end{eqnarray*}
which, by virtue of \textbf{(H3)(i)} and the fact that the second summand on the left hand-side is nonnegative by \eqref{eqnarray19} applied to $p=2$, yields
\begin{eqnarray}\label{eqna7}
  \mathbb{E}|X^{\nu,\varepsilon}_\lambda(t)|_2^2 \leq
e^{C_3T}|x|_2^2,~ \forall \varepsilon>0,~t\in[0,T],~
x\in L^2(\mu).
\end{eqnarray}

%
%lemma 4.2

\begin{lemma}\label{lemma4}
For  $x\in L^2(\mu)\cap L^{2m}(\mu)$, we have that
$X^{\nu,\varepsilon}_\lambda\in
L^\infty\big([0,T];L^{2m}(\Omega;L^{2m}(\mu))\big)$.
\end{lemma}
\begin{proof}
For $\alpha,R>0$, consider the set
\begin{eqnarray*}
\mathcal {K}_R=&&\!\!\!\!\!\!\!\!\big\{ X\in L^2\big([0,T];C([0,T];L^2(\mu))\big), e^{-2m\alpha t}\mathbb{E}|X(t)|_{2m}^{2m}\leq R^{2m}, \ t\in [0,T]\big\}.
\end{eqnarray*}
Since, by \eqref{eq:2}, $X^{\nu,\varepsilon}_\lambda$ is a fixed
point of the map
\begin{eqnarray*}
F:X\mapsto~e^{-\frac{\bullet}{\varepsilon}}x+\frac{1}{\varepsilon}\int_0^{\bullet}e^{-\frac{\bullet-s}{\varepsilon}}J_\varepsilon(X(s))ds+\int_0^{\bullet}e^{-\frac{\bullet-s}{\varepsilon}}B(s,X(s))dW(s),
\end{eqnarray*}
obtained by iteration in $L^2\big(\Omega;C([0,T];L^2(\mu))\big)$,
it suffices to show that $F$ leaves the set $\mathcal {K}_R$
invariant for $\alpha,R>0$ large enough. By \eqref{eqnarray19} we have that
for $X\in \mathcal {K}_R$, $t\geq0$
\begin{eqnarray}\label{eq:6}
&&\!\!\!\!\!\!\!\!\Big[e^{-2m\alpha t}\mathbb{E}\Big|e^{-\frac{t}{\varepsilon}}x+\frac{1}{\varepsilon}\int_0^te^{-\frac{t-s}{\varepsilon}}J_\varepsilon(X(s))ds\Big|_{2m}^{2m}\Big]^{\frac{1}{2m}}\nonumber \\
\leq&& \!\!\!\!\!\!\!\!e^{-\alpha
t}e^{-\frac{t}{\varepsilon}}|x|_{2m}+e^{-\alpha
t}\Big[\mathbb{E}\Big(\int_0^t\frac{1}{\varepsilon}e^{-\frac{t-s}{\varepsilon}}|X(s)|_{2m}ds\Big)^{2m}\Big]^{\frac{1}{2m}}\nonumber\\
\leq&& \!\!\!\!\!\!\!\!e^{-(\alpha+\frac{1}{\varepsilon})t}|x|_{2m}+e^{-\alpha t}\int_0^t\frac{1}{\varepsilon}e^{-\frac{t-s}{\varepsilon}}\big(\mathbb{E}|X(s)|_{2m}^{2m}\big)^{\frac{1}{2m}}ds \nonumber\\
\leq&&\!\!\!\!\!\!\!\!e^{-(\alpha+\frac{1}{\varepsilon})t}|x|_{2m}+\frac{R}{1+\alpha\varepsilon}.
\end{eqnarray}

Set
\begin{eqnarray*}
  Y(t)=\int_0^te^{-\frac{t-s}{\varepsilon}} B(s,X(s))dW(s),\
t\geq0.
\end{eqnarray*}
Then $Y$ is a solution to the following SDE on $L^2(\mu)$:
\begin{equation*}
\left\{ \begin{aligned}
&dY(t)+\frac{1}{\varepsilon}Y(t)dt=B(t,X(t))dW(t),\ t\geq0,\\
&Y(0)=0,
\end{aligned} \right.
\end{equation*}
equivalently,
$$d\big(e^{\frac{t}{\varepsilon}}Y(t)\big)=e^{\frac{t}{\varepsilon}}B(t,X(t))dW(t),~t\geq0,~Y(0)=0.$$
By Hypothesis \textbf{(H3)(ii)}, we may apply It\^{o}'s formula in expectation from Theorem \ref{ito} in the Appendix with $u(t)$ replaced by $e^{\frac{t}{\varepsilon}}Y(t)$. Then by H\"{o}lder's and Young's inequality and \textbf{(H3)(ii)}, we obtain for $t\in[0,T]$
\begin{eqnarray*}\label{eqnarray43}
&&\!\!\!\!\!\!\!\!\mathbb{E}\big|e^{\frac{t}{\varepsilon}}Y(t)\big|_{2m}^{2m}\nonumber\\
=&&\!\!\!\!\!\!\!\!m(2m-1)\mathbb{E}\int_0^t\int_E\big|e^{\frac{s}{\varepsilon}}Y(s)\big|^{2m-2}\cdot\sum_{k=1}^\infty\big|e^{\frac{s}{\varepsilon}}B(s,X(s))e_k\big|^2d\mu
ds\nonumber\\
\leq&&\!\!\!\!\!\!\!\!m(2m-1)\mathbb{E}\int_0^t\Big(\int_E|e^{\frac{s}{\varepsilon}}Y(s)|^{2m-2\cdot\frac{m}{m-1}}d\mu \Big)^{\frac{m-1}{m}}\cdot \Bigg(\int_E\Big(\sum_{k=1}^\infty|e^{\frac{s}{\varepsilon}}B(s,X(s))e_k|^2\Big)^{m}d\mu\Bigg)^{\frac{1}{m}}ds\nonumber\\
=&&\!\!\!\!\!\!\!\!m(2m-1)\mathbb{E}\int_{0}^{t}|e^{\frac{s}{\varepsilon}}Y(s)|^{2m-2}_{2m}\cdot\Bigg(\int_E\Big(\sum_{k=1}^{\infty}|e^{\frac{s}{\varepsilon}}B(s,X(s))e_k|^2\Big)^{m}d\mu\Bigg)^{\frac{1}{m}}ds\nonumber\\
\leq&&\!\!\!\!\!\!\!\!m(2m-1)\mathbb{E}\int_{0}^{t}\frac{\Big(|e^{\frac{s}{\varepsilon}}Y(s)|^{2m-2}_{2m}\Big)^{\frac{m}{m-1}}}{\frac{m}{m-1}}+\frac{\Big(\int_E\big(\sum_{k=1}^{\infty}|e^{\frac{s}{\varepsilon}}B(s,X(s))e_k|^2\big)^{m}d\mu\Big)^{\frac{1}{m}\cdot m}}{m}ds\nonumber\\
=&&\!\!\!\!\!\!\!\!(m-1)(2m-1)\mathbb{E}\int_0^t|e^{\frac{s}{\varepsilon}}Y(s)\big|_{2m}^{2m}ds+(2m-1)\Bbb{E}\int_{0}^{t}\int_E\Big(\sum_{k=1}^{\infty}|e^{\frac{s}{\varepsilon}}B(s,X(s))e_k|^2\Big)^{m}ds\nonumber\\
\leq&&\!\!\!\!\!\!\!\!(m-1)(2m-1)\mathbb{E}\int_0^t|e^{\frac{s}{\varepsilon}}Y(s)\big|_{2m}^{2m}ds+C_4(2m-1)\Bbb{E}\int_{0}^{t}|e^{\frac{s}{\varepsilon}}X(s)|_{2m}^{2m}ds,
\end{eqnarray*}
and therefore, by Gronwall's lemma, we obtain
\begin{eqnarray*}\label{eq:001}
\mathbb{E}\big|e^{\frac{t}{\varepsilon}}Y(t)\big|_{2m}^{2m}&&\!\!\!\!\!\!\!\!\leq C_4(2m-1)e^{(m-1)(2m-1)T}\int_0^t\Bbb{E}|e^{\frac{s}{\varepsilon}}X(s)|_{2m}^{2m}ds\nonumber\\
&&\!\!\!\!\!\!\!\!\leq C_4(2m-1)e^{(m-1)(2m-1)T}\int_0^tR^{2m}e^{(\frac{2m}{\varepsilon}+2m\alpha)s}ds\nonumber\\
&&\!\!\!\!\!\!\!\!\leq \frac{\varepsilon R^{2m}C_{T,m}}{2m(1+\varepsilon\alpha)}e^{\frac{2 mt(1+\varepsilon\alpha)}{\varepsilon}},
\end{eqnarray*}
which yields
\begin{eqnarray}\label{eq:002}
&&e^{-2m\alpha t}\mathbb{E}|Y(t)|_{2m}^{2m}\leq \frac{\varepsilon R^{2m} C_{T,m}
}{(1+\varepsilon\alpha)},~\forall t\in [0,T].
\end{eqnarray}
Then, by formulas \eqref{eq:6}, \eqref{eq:002}, we infer that for
$\alpha$ large enough and $R\geq2|x|_{2m}$, the map $F$ leaves
$\mathcal {K}_R$ invariant as claimed.
\end{proof}

%lemma 4.3

\begin{lemma}\label{lemma5}
For $x\in L^2(\mu)\cap L^{2m}(\mu)$, there exists $C\in (0,\infty)$ such that
\begin{eqnarray}\label{eqna11}
\sup_{t\in[0,T]}\mathbb{E}|X^{\nu,\varepsilon}_\lambda(t)|_{2m}^{2m}\leq
C|x|_{2m}^{2m}, \ \forall\varepsilon, \lambda\in(0,1), \nu\in(0,1].
\end{eqnarray}
\end{lemma}
\begin{proof}
Applying the It\^{o} formula in expectation to $\mathbb{E}|X^{\nu,\varepsilon}_\lambda(t)|_{2m}^{2m}$ from Theorem \ref{ito} in the Appendix, we obtain
\begin{eqnarray}\label{eqna8}
  \mathbb{E}|X^{\nu,\varepsilon}_\lambda(t)|_{2m}^{2m}=&&\!\!\!\!\!\!\!\!|x|_{2m}^{2m}-2m\mathbb{E}\int_0^t\int_EA^{\nu,\varepsilon}_\lambda(X^{\nu,\varepsilon}_\lambda(s))X^{\nu,\varepsilon}_\lambda(s)|X^{\nu,\varepsilon}_\lambda(s)|^{2m-2}d\mu ds\nonumber\\
&&\!\!\!\!\!\!\!\!+m(2m-1)\mathbb{E}\int_0^t\int_E|X^{\nu,\varepsilon}_\lambda(s)|^{2m-2}\cdot \sum_{k=1}^\infty|B(s,X^{\nu,\varepsilon}_\lambda(s))e_k|^2d\mu
ds.
\end{eqnarray}
Recall that
$A^{\nu,\varepsilon}_\lambda(X^{\nu,\varepsilon}_\lambda(s))=\frac{1}{\varepsilon}(X^{\nu,\varepsilon}_\lambda(s)-J_\varepsilon(X^{\nu,\varepsilon}_\lambda(s)))$, so we have
\begin{eqnarray}\label{eqna9}
&&\!\!\!\!\!\!\!\!\int_EA^{\nu,\varepsilon}_\lambda(X^{\nu,\varepsilon}_\lambda(s))X^{\nu,\varepsilon}_\lambda(s)|X^{\nu,\varepsilon}_\lambda(s)|^{2m-2}d\mu\nonumber\\
=&&\!\!\!\!\!\!\!\!\frac{1}{\varepsilon}\int_E|X^{\nu,\varepsilon}_\lambda(s)|^{2m}d\mu-\frac{1}{\varepsilon}\int_EJ_\varepsilon(X^{\nu,\varepsilon}_\lambda(s))X^{\nu,\varepsilon}_\lambda(s)|X^{\nu,\varepsilon}_\lambda(s)|^{2m-2}d\mu.
\end{eqnarray}
By H\"{o}lder's inequality and \eqref{eqnarray19}, we conclude
\begin{eqnarray}\label{eq:4.31}
    &&\!\!\!\!\!\!\!\!\frac{1}{\varepsilon}\int_E|X^{\nu,\varepsilon}_\lambda(s)|^{2m}d\mu-\frac{1}{\varepsilon}\int_EJ_\varepsilon(X^{\nu,\varepsilon}_\lambda(s))X^{\nu,\varepsilon}_\lambda(s)|X^{\nu,\varepsilon}_\lambda(s)|^{2m-2}d\mu\nonumber\\
\geq&&\!\!\!\!\!\!\!\!\frac{1}{\varepsilon}\int_E|X^{\nu,\varepsilon}_\lambda(s)|^{2m}d\mu-\frac{1}{\varepsilon}\Big[|J_\varepsilon(X^{\nu,\varepsilon}_\lambda(s))|_{2m}\cdot|X^{\nu,\varepsilon}_\lambda(s)|_{2m}^{2m-1}\Big]\nonumber\\
\geq&&\!\!\!\!\!\!\!\!\frac{1}{\varepsilon}\int_E|X^{\nu,\varepsilon}_\lambda(s)|^{2m}d\mu-\frac{1}{\varepsilon}|X^{\nu,\varepsilon}_\lambda(s)|_{2m}\cdot|X^{\nu,\varepsilon}_\lambda(s)|_{2m}^{2m-1}\nonumber\\
   =&&\!\!\!\!\!\!\!\!0.
\end{eqnarray}
By \eqref{eqna8}-\eqref{eq:4.31} and using a similar argument as in \eqref{eqnarray43}, we get
\begin{eqnarray*}
  \mathbb{E}|X^{\nu,\varepsilon}_\lambda(t)|_{2m}^{2m}&&\!\!\!\!\!\!\!\!\leq|x|_{2m}^{2m}+(2m-1)\mathbb{E}\int_0^t(m-1)|X^{\nu,\varepsilon}_\lambda(s)|^{2m}_{2m}+C_4|X^{\nu,\varepsilon}_\lambda(s)|^{2m}_{2m}ds\nonumber\\
&&\!\!\!\!\!\!\!\!=|x|_{2m}^{2m}+(2m-1)(m-1+C_4)\mathbb{E}\int_0^t|X^{\nu,\varepsilon}_\lambda(s)|^{2m}_{2m}ds.
\end{eqnarray*}
As a result, by Gronwall's lemma, there exists $C\in(0,\infty)$ such that
\begin{eqnarray*}
\emph{esssup}_{t\in[0,T]}\mathbb{E}|X^{\nu,\varepsilon}_\lambda(t)|_{2m}^{2m}\leq
C|x|_{2m}^{2m}, \ \forall~\varepsilon, \lambda\in(0,1), \nu\in(0,1].
\end{eqnarray*}
Since $t\mapsto |X^{\nu,\varepsilon}_\lambda(t)|_{2m}$ is lower semi-continuous and hence so is $t\mapsto \Bbb{E}|X^{\nu,\varepsilon}_\lambda(t)|_{2m}^{2m}$, \eqref{eqna11} follows.

\end{proof}

%lemma 4.4

\begin{lemma}\label{lemma6}
For $x\in L^2(\mu)\cap L^{2m}(\mu)$ and $X^{\nu,\varepsilon}_\lambda$ as above. Then as $\varepsilon\longrightarrow0$, we have
$$X^{\nu,\varepsilon}_\lambda\longrightarrow X^{\nu}_\lambda\ \text{strongly\ in}\ L^2(\Omega;C([0,T];F^*_{1,2})),$$
where $X^{\nu}_\lambda$ is the solution to \eqref{eq:4}. Furthermore, there exists $C\in(0,\infty)$ such that
\begin{eqnarray}\label{4.4}
\sup_{t\in[0,T]}\Bbb{E}|X^\nu_\lambda(t)|_{2m}^{2m}\leq C|x|_{2m}^{2m},~~\forall~\lambda\in(0,1), \nu\in(0,1].
\end{eqnarray}
\end{lemma}

\begin{proof}
We prove the lemma in two steps, which are given as two claims.

%
%claim 4.5

\begin{claim}\label{claim1}
For each $x\in L^2(\mu)$, the sequence
$\{X^{\nu,\varepsilon}_\lambda\}$, $0<\varepsilon<1$, is Cauchy in
$L^2(\Omega;C([0,T];F^*_{1,2}))$.
\end{claim}
\begin{proof}
Let $\varepsilon,\eta>0$. Applying the It\^{o} formula (\cite[Theorem 4.2.5]{LR} with $V:=L^2(\mu)$, $H:=F^*_{1,2,\nu}$, $\alpha=2$, $X_0=x$) to
$\|X^{\nu,\varepsilon}_\lambda-X^{\nu,\eta}_\lambda\|^2_{F^*_{1,2,\nu}}$,
we have
\begin{eqnarray}\label{eqn1}
&&\!\!\!\!\!\!\!\!d\|X^{\nu,\varepsilon}_\lambda(t)-X^{\nu,\eta}_\lambda(t)\|^2_{F^*_{1,2,\nu}}\nonumber\\
&&\!\!\!\!\!\!\!\!+2\big\langle(\nu-L)\big((\Psi_\lambda+\lambda I)(J_\varepsilon(X^{\nu,\varepsilon}_\lambda(t)))-(\Psi_\lambda+\lambda)(J_\eta(X^{\nu,\eta}_\lambda(t)))\big),X^{\nu,\varepsilon}_\lambda(t)-X^{\nu,\eta}_\lambda(t)\big\rangle_{F^*_{1,2,\nu}}dt\nonumber\\
=&&\!\!\!\!\!\!\!\!2\big\langle
X^{\nu,\varepsilon}_\lambda(t)-X^{\nu,\eta}_\lambda(t),
\big(B(t,X^{\nu,\varepsilon}_\lambda(t))-B(t,X^{\nu,\eta}_\lambda(t))\big)dW(t)\big\rangle_{F^*_{1,2,\nu}}\nonumber\\
&&\!\!\!\!\!\!\!\!+\big\|B(t,X^{\nu,\varepsilon}_\lambda(t))-B(t,X^{\nu,\eta}_\lambda(t))\big\|^2_{L_2(L^2(\mu),F^*_{1,2,\nu})}dt.
\end{eqnarray}
The second term in the left hand-side of the above equality, by \eqref{eqnarray18}, \eqref{equation3} and \eqref{eqnarray48}, is equal to
\begin{eqnarray}\label{eqn2}
&&\!\!\!\!\!\!\!\!2\big\langle(\Psi_\lambda+\lambda I)(J_\varepsilon(X^{\nu,\varepsilon}_\lambda(t)))-(\Psi_\lambda+\lambda I)(J_\eta(X^{\nu,\eta}_\lambda(t))),J_\varepsilon(X^{\nu,\varepsilon}_\lambda(t))-J_\eta(X^{\nu,\eta}_\lambda(t))\big\rangle_2dt\nonumber\\
+&&\!\!\!\!\!\!\!\!2\big\langle(\nu-L)\big((\Psi_\lambda+\lambda I)(J_\varepsilon(X^{\nu,\varepsilon}_\lambda(s)))\big)-(\nu-L)\big((\Psi_\lambda+\lambda I)(J_\eta(X^{\nu,\eta}_\lambda(s)))\big),\nonumber\\
&&\!\!\!\!\varepsilon(\nu-L)\big((\Psi_\lambda+\lambda I)(J_\varepsilon(X^{\nu,\varepsilon}_\lambda(s)))\big)-\eta(\nu-L)\big((\Psi_\lambda+\lambda I)(J_\eta(X^{\nu,\eta}_\lambda(s)))\big)\big\rangle_{F^*_{1,2,\nu}}\!\!\!\!dt.
\end{eqnarray}
Taking \eqref{eqn2} into \eqref{eqn1}, then taking expectation of both sides, we obtain for all $t\in[0,T]$
\begin{eqnarray}\label{4.5}
&&\!\!\!\!\!\!\!\!\mathbb{E}\sup_{r\in[0,t]}\|X^{\nu,\varepsilon}_\lambda(r)-X^{\nu,\eta}_\lambda(r)\|^2_{F^*_{1,2,\nu}}\nonumber\\
&&\!\!\!\!\!\!\!\!-2\mathbb{E}\Big[\sup_{r\in[0,t]}\Big|\int_0^r\big\langle
X^{\nu,\varepsilon}_\lambda(s)-X^{\nu,\eta}_\lambda(s),(B(s,X^{\nu,\varepsilon}_\lambda(s))-B(s,X^{\nu,\eta}_\lambda(s)))dW(s)\big\rangle_{F^*_{1,2,\nu}}\Big|\Big]\nonumber\\
&&\!\!\!\!\!\!\!\!+2\mathbb{E}\int_0^t\big\langle(\Psi_\lambda+\lambda I)\big(J_\varepsilon(X^{\nu,\varepsilon}_\lambda(s))\big)-(\Psi_\lambda+\lambda I)\big(J_\eta(X^{\nu,\eta}_\lambda(s))\big)
,J_\varepsilon(X^{\nu,\varepsilon}_\lambda(s))-J_\eta(X^{\nu,\eta}_\lambda(s))\big\rangle_2 ds\nonumber\\
\leq&&\!\!\!\!\!\!\!\!2\mathbb{E}\int_0^T\big|\big\langle(\nu-L)\big((\Psi_\lambda+\lambda I)(J_\varepsilon(X^{\nu,\varepsilon}_\lambda(s)))\big)-(\nu-L)\big((\Psi_\lambda+\lambda I)(J_\eta(X^{\nu,\eta}_\lambda(s)))\big),\nonumber\\
&&~~~~~~~~~~\varepsilon(\nu-L)\big((\Psi_\lambda+\lambda I)(J_\varepsilon(X^{\nu,\varepsilon}_\lambda(s)))\big)-\eta(\nu-L)\big((\Psi_\lambda+\lambda I)(J_\eta(X^{\nu,\eta}_\lambda(s)))\big)\big\rangle_{F^*_{1,2,\nu}}\big|ds\nonumber\\
&&\!\!\!\!\!\!\!\!+\mathbb{E}\int_0^t\big\|B(s,X^{\nu,\varepsilon}_\lambda(s))-B(s,X^{\nu,\eta}_\lambda(s))\big\|^2_{L_2(L^2(\mu),F^*_{1,2,\nu})}ds\nonumber\\
\leq&&\!\!\!\!\!\!\!\!3(\varepsilon+\eta)\mathbb{E}\int_0^T\|(\nu-L)(\Psi_\lambda+\lambda I)(J_\varepsilon(X^{\nu,\varepsilon}_\lambda(s)))\|^2_{F^*_{1,2,\nu}}+\|(\nu-L)(\Psi_\lambda+\lambda I)(J_\eta(X^{\nu,\eta}_\lambda(s)))\|^2_{F^*_{1,2,\nu}}ds\nonumber\\
&&\!\!\!\!\!\!\!\!+\mathbb{E}\int_0^t\big\|B(s,X^{\nu,\varepsilon}_\lambda(s))-B(s,X^{\nu,\eta}_\lambda(s))\big\|^2_{L_2(L^2(\mu),F^*_{1,2,\nu})}ds\nonumber\\
\leq&&\!\!\!\!\!\!\!\!3(\varepsilon+\eta)(\frac{1}{\lambda}+\lambda+C_5)e^{C_3T}|x|^2_2+\mathbb{E}\int_0^tC_1\|X^{\nu,\varepsilon}_\lambda(s)-X^{\nu,\eta}_\lambda(s)\|^2_{F^*_{1,2,\nu}}ds,
\end{eqnarray}
where we used Proposition \ref{Proposition7.1} (see Appendix) and \textbf{(H2)(i)} in the last inequality.
For the second term in the left hand-side of \eqref{4.5},
by using the Burkholder-Davis-Gundy (BDG) inequality for $p=1$, we obtain for all $t\in[0,T]$,
\begin{eqnarray}\label{eqna15}
&&\!\!\!\!\!\!\!\!\mathbb{E}\Big[\sup_{r\in[0,t]}\Big|\int_0^r\big\langle X^{\nu,\varepsilon}_\lambda(s)-X^{\nu,\eta}_\lambda(s),(B(s,X^{\nu,\varepsilon}_\lambda(s))-B(s,X^{\nu,\eta}_\lambda(s)))dW(s)\big\rangle_{F^*_{1,2,\nu}}\Big|\Big]\nonumber\\
\leq&&\!\!\!\!\!\!\!\!\mathbb{E}\Big[\int_0^t\|X^{\nu,\varepsilon}_\lambda(s)-X^{\nu,\eta}_\lambda(s)\|^2_{F^*_{1,2,\nu}}\cdot
C_1\|X^{\nu,\varepsilon}_\lambda(s)-X^{\nu,\eta}_\lambda(s)\|^2_{F^*_{1,2,\nu}}ds\Big]^{\frac{1}{2}}\nonumber\\
\leq&&\!\!\!\!\!\!\!\!\mathbb{E}\Big[\sup_{r\in[0,t]}\|X^{\nu,\varepsilon}_\lambda(r)-X^{\nu,\eta}_\lambda(r)\|^2_{F^*_{1,2,\nu}}\cdot
C_1\int_0^t
\|X^{\nu,\varepsilon}_\lambda(s)-X^{\nu,\eta}_\lambda(s)\|^2_{F^*_{1,2,\nu}}ds\Big]^{\frac{1}{2}}\nonumber\\
\leq&&\!\!\!\!\!\!\!\!\frac{1}{4}\mathbb{E}\sup_{r\in[0,t]}\|X^{\nu,\varepsilon}_\lambda(r)-X^{\nu,\eta}_\lambda(r)\|^2_{F^*_{1,2,\nu}}+C_1\mathbb{E}\int_0^t\|X^{\nu,\varepsilon}_\lambda(s)-X^{\nu,\eta}_\lambda(s)\|^2_{F^*_{1,2,\nu}}ds.
\end{eqnarray}
Substituting \eqref{eqna15} into \eqref{4.5}, we obtain
\begin{eqnarray*}
&&\!\!\!\!\!\!\!\!\frac{1}{2}\mathbb{E}\sup_{r\in[0,t]}\|X^{\nu,\varepsilon}_\lambda(r)-X^{\nu,\eta}_\lambda(r)\|^2_{F^*_{1,2,\nu}}\nonumber\\
&&\!\!\!\!\!\!\!\!+2\mathbb{E}\int_0^t\big\langle(\Psi_\lambda+\lambda I)\big(J_\varepsilon(X^{\nu,\varepsilon}_\lambda(s))\big)-(\Psi_\lambda+\lambda I)\big(J_\eta(X^{\nu,\eta}_\lambda(s))\big)
,J_\varepsilon(X^{\nu,\varepsilon}_\lambda(s))-J_\eta(X^{\nu,\eta}_\lambda(s))\big\rangle_2 ds\nonumber\\
\leq&&\!\!\!\!\!\!\!\!3(\varepsilon+\eta)(\frac{1}{\lambda}+\lambda+C_5)e^{C_3T}|x|^2_2+3C_1\mathbb{E}\int_0^t\sup_{r\in[0,s]}\|X^{\nu,\varepsilon}_\lambda(r)-X^{\nu,\eta}_\lambda(r)\|^2_{F^*_{1,2,\nu}}ds.
\end{eqnarray*}
By Gronwall's lemma, we obtain
\begin{eqnarray}\label{4.6}
&&\!\!\!\!\!\!\!\!\mathbb{E}\sup_{t\in[0,T]}\|X^{\nu,\varepsilon}_\lambda(t)-X^{\nu,\eta}_\lambda(t)\|^2_{F^*_{1,2,\nu}}\nonumber\\
&&\!\!\!\!\!\!\!\!+4\mathbb{E}\int_0^t\big\langle(\Psi_\lambda+\lambda I)\big(J_\varepsilon(X^{\nu,\varepsilon}_\lambda(s))\big)-(\Psi_\lambda+\lambda I)\big(J_\eta(X^{\nu,\eta}_\lambda(s))\big)
,J_\varepsilon(X^{\nu,\varepsilon}_\lambda(s))-J_\eta(X^{\nu,\eta}_\lambda(s))\big\rangle_2 ds\nonumber\\
\leq&&\!\!\!\!\!\!\!\!6(\varepsilon+\eta)(\frac{1}{\lambda}+\lambda+C)e^{(6C_1+C_3)T}|x|_2^2.
\end{eqnarray}
Since by the monotonicity of $\Psi_\lambda$ the second term on the left hand-side of inequality \eqref{4.6} is nonnegative, letting $\varepsilon, \eta\rightarrow0$, we see that $\{X^{\nu,\varepsilon}_\lambda\}$ is Cauchy in $L^2(\Omega;C([0,T];F^*_{1,2}))$.
\end{proof}

From Claim \ref{claim1}, we know there exists
$\widetilde{X}\in L^2(\Omega;C([0,T];F^*_{1,2}))$ such that
\begin{eqnarray}\label{4.7}
\lim_{\varepsilon\rightarrow0}X^{\nu,\varepsilon}_\lambda=\widetilde{X}~~\text{in}~L^2(\Omega;C([0,T];F^*_{1,2})),
\end{eqnarray}

%claim 4.6

\begin{claim}\label{claim2}
$\widetilde{X}=X^\nu_\lambda$.
\end{claim}
\begin{proof}
We have
\begin{eqnarray}\label{4.71}
    \lim_{\varepsilon\rightarrow0}\int_0^\cdot
    B(s,X^{\nu,\varepsilon}_\lambda(s))dW(s)=\int_0^{\cdot}
    B(s,\widetilde{X}(s))dW(s)~~\text{in}~L^2(\Omega;C([0,T];F^*_{1,2})),
    \end{eqnarray}
since by the BDG inequality for $p=1$ and \textbf{(H2)(i)}, we have
\begin{eqnarray*}
&&\!\!\!\!\!\!\!\!\mathbb{E}\sup_{r\in[0,T]}\Big\|\int_0^r(B(s,X^{\nu,\varepsilon}_\lambda(s))-B(s,\widetilde{X}(s)))dW(s)\Big\|^2_{F^*_{1,2,\nu}}\nonumber\\
\leq&&\!\!\!\!\!\!\!\!C\mathbb{E}\int_0^T\|B(s,X^{\nu,\varepsilon}_\lambda(s))-B(s,\widetilde{X}(s))\|^2_{L_2(L^2(\mu),F^*_{1,2,\nu})}ds\nonumber\\
\leq&&\!\!\!\!\!\!\!\!CT\mathbb{E}\sup_{s\in[0,T]}\|X^{\nu,\varepsilon}_\lambda(s)-\widetilde{X}(s)\|^2_{F^*_{1,2,\nu}}.
\end{eqnarray*}
Next we show that
$(\Psi_\lambda+\lambda)(\widetilde{X})\in
L^2((0,T);L^2(\Omega;F_{1,2}))$ and that \eqref{eq:4} is satisfied.
From \eqref{eqna7} we know that
$\{X^{\nu,\varepsilon}_\lambda\}$ is bounded in
$L^2((0,T)\times\Omega\times E)$ and therefore along a subsequence,
again denoted by $\{\varepsilon\}$, we have
\begin{eqnarray}\label{eq:15}
\lim_{\varepsilon\rightarrow0}X^{\nu,\varepsilon}_\lambda=\widetilde{X}~~\text{weakly
~in}~L^2((0,T)\times\Omega\times E).
\end{eqnarray}
From \eqref{eqnarray18} and \eqref{eqna19}, we know
\begin{eqnarray}\label{eq:13}
&&\!\!\!\!\!\!\!\!\mathbb{E}\int_0^T\|X^{\nu,\varepsilon}_\lambda(s)-J_\varepsilon(X^{\nu,\varepsilon}_\lambda(s))\|^2_{F^*_{1,2,\nu}}ds\nonumber\\
=&&\!\!\!\!\!\!\!\!\varepsilon^2\mathbb{E}\int_0^T\|(\nu-L)(\Psi_\lambda(J_\varepsilon(X^{\nu,\varepsilon}_\lambda(s)))+\lambda
J_\varepsilon(X^{\nu,\varepsilon}_\lambda(s)))\|^2_{F^*_{1,2,\nu}}ds\nonumber\\
\leq&&\!\!\!\!\!\!\!\!
\frac{\varepsilon}{2}(\frac{1}{\lambda}+\lambda+C_5)e^{C_3T}|x|^2_2,
\end{eqnarray}
which yields,
\begin{eqnarray*}
\lim_{\varepsilon\rightarrow0}J_\varepsilon(X^{\nu,\varepsilon}_\lambda)=\widetilde{X}~~\text{in}~L^2((0,T);L^2(\Omega;F^*_{1,2})).
\end{eqnarray*}
Recall from \eqref{eqnarray19} that
\begin{eqnarray*}
|J_\varepsilon(X^{\nu,\varepsilon}_\lambda(t))|_2\leq|X^{\nu,\varepsilon}_\lambda(t)|_2,~~\forall~
t\in [0,T].
\end{eqnarray*}
Therefore, we infer by \eqref{eq:15} and \eqref{eq:13} that
\begin{eqnarray}\label{4.8}
\lim_{\varepsilon\rightarrow0}J_\varepsilon(X^{\nu,\varepsilon}_\lambda)=\widetilde{X},
~\text{weakly~in}~L^2((0,T)\times\Omega\times E).
\end{eqnarray}
By the monotonicity of $\Psi_\lambda$, it follows from \eqref{4.6} that
$J_\varepsilon(X^{\nu,\varepsilon}_\lambda)$, $\varepsilon\in(0,1)$, is Cauchy in $L^2((0,T)\times\Omega\times E)$, so the convergence in \eqref{4.8} is strong and thus
\begin{eqnarray*}
\lim_{\varepsilon\rightarrow0}(\Psi_\lambda+\lambda I)(J_\varepsilon(X^{\nu,\varepsilon}_\lambda))=(\Psi_\lambda+\lambda I)(\widetilde{X})~~\text{in}~L^2((0,T)\times\Omega\times
E),
\end{eqnarray*}
since $\Psi_\lambda+\lambda I$ is Lipschitz.

From \eqref{eqna19}, we know that
$(\nu-L)(\Psi_\lambda+\lambda I)(J_\varepsilon(X^{\nu,\varepsilon}_\lambda))$, $\varepsilon\in(0,1)$, is bounded in $L^2([0,T]\times\Omega;F^*_{1,2})$, so
$(\Psi_\lambda+\lambda I)(J_\varepsilon(X^{\nu,\varepsilon}_\lambda))$
is bounded in $L^2([0,T];L^2(\Omega,F_{1,2}))$. Hence there exists a
subsequence, again denoted by $\{\varepsilon\}$ such that
\begin{eqnarray}\label{4.9}
\lim_{\varepsilon\rightarrow0}(\Psi_\lambda+\lambda I)(J_\varepsilon(X^{\nu,\varepsilon}_\lambda))=(\Psi_\lambda+\lambda I)(\widetilde{X})
\text{~weakly~in}~L^2([0,T]\times\Omega;F_{1,2}).
\end{eqnarray}
It is then easy to see that also
\begin{eqnarray*}
\lim_{\varepsilon\rightarrow0}\int_0^\cdot(\Psi_\lambda+\lambda I)(J_\varepsilon(X^{\nu,\varepsilon}_\lambda(s)))ds=\int_0^\cdot(\Psi_\lambda+\lambda I)(\widetilde{X}(s))ds
\end{eqnarray*}
weakly in $L^2([0,T]\times \Omega;F_{1,2}),$
and thus
\begin{eqnarray*}
\lim_{\varepsilon\rightarrow0}(\nu-L)\int_0^\cdot(\Psi_\lambda+\lambda I)J_\varepsilon(X^{\nu,\varepsilon}_\lambda(s))ds=(\nu-L)\int_0^\cdot(\Psi_\lambda+\lambda I)(\widetilde{X}(s))ds
\end{eqnarray*}
weakly in $L^2([0,T]\times \Omega;F^*_{1,2})$.

Consequently, taking into account \eqref{4.7}, \eqref{4.71}, as $\varepsilon\rightarrow0$, we can pass to the weak limit in $L^2([0,T]\times\Omega;F^*_{1,2})$ in the equation

\begin{eqnarray*}
X^{\nu,\varepsilon}_\lambda(t)=x+(\nu-L)\int_0^t(\Psi_\lambda+\lambda I)(J_\varepsilon(X^{\nu,\varepsilon}_\lambda(s)))ds+\int_0^tB(s,X^{\nu,\varepsilon}_\lambda(s))dW(s),
\end{eqnarray*}
and since each term is a $\Bbb{P}$-a.s. continuous path in $F^*_{1,2}$, we conclude that $\widetilde{X}$ is a strong solution to \eqref{eq:4} in the sense of the definition in Lemma 3.1 of \cite{RWX}. Furthermore, by the uniqueness part of \cite[Lemma 3.1]{RWX}, it follows that $X^\nu_\lambda=\widetilde{X}$ a.e. in $[0,T]\times\Omega\times E$.
\end{proof}

It remains to prove \eqref{4.4}. By \eqref{eqna7} and \eqref{eqna11}, we know that for fixed $\nu$ and $\lambda$, there exists a subsequence denoted as $\{X^{\nu,\varepsilon_n}_\lambda\}$, $0<\varepsilon_n<1$ and an element $Y^\nu_\lambda\in L^\infty([0,T];(L^2\cap L^{2m})(\Omega\times E))$ such that as $\varepsilon_n\rightarrow0$,
$$X^{\nu,\varepsilon_n}_\lambda\longrightarrow Y^\nu_\lambda~~\text{weakly~in}~L^\infty([0,T];(L^2\cap L^{2m})(\Omega\times E)),$$
hence also in $L^2([0,T]\times \Omega;L^2(\mu))$  and in $L^2([0,T]\times \Omega;F^*_{1,2})$,
but $\{X^{\nu,\varepsilon_n}_\lambda\}$ is Cauchy in $L^2(\Omega;C([0,T];F^*_{1,2}))$ by Claim \ref{claim1}, therefore we have $Y^\nu_\lambda=X^\nu_\lambda$ a.e. in $[0,T]\times\Omega\times E$. Note that the above weak convergence also holds in $L^{\infty}([0,T]; L^{2m}(\Omega\times E))$, hence by the (sequentially) weakly lower-semicontinuity of norms, letting $\varepsilon\rightarrow0$ in \eqref{eqna11}, we can get \eqref{4.4}. This completes the proof of Lemma \ref{lemma6}.

\end{proof}

%remark 4.7

\begin{remark}\label{remark2}
By Lemma \ref{lemma6} we know that
\begin{eqnarray*}
X^\nu_\lambda(t)=&&\!\!\!\!\!\!\!\!x+(\nu-L)\int_0^t\big(\Psi_\lambda(X^\nu_\lambda(s))+\lambda X^\nu_\lambda(s)\big)ds+\int_0^tB(s,X^\nu_\lambda(s))dW(s),~t\in[0,T].
\end{eqnarray*}
But, since $X^\nu_\lambda=\widetilde{X}$, by \eqref{4.9} we may interchange $(\nu-L)$ with the integral w.r.t. $ds$.
\end{remark}

\vspace{2mm}
Let us now continue to prove Theorem \ref{Theorem}. Choose
$0<\nu\leq\nu_0\leq1$, rewrite \eqref{eq:4} as
\begin{eqnarray*}
&&\!\!\!\!\!\!\!\!dX^\nu_\lambda(t)+(\nu_0-L)(\Psi_\lambda(X^\nu_\lambda(t))+\lambda
X^\nu_\lambda(t))dt\nonumber\\
=&&\!\!\!\!\!\!\!\!(\nu_0-\nu)(\Psi_\lambda(X^\nu_\lambda(t))+\lambda
X^\nu_\lambda(t))dt+B(t,X^\nu_\lambda(t))dW(t).
\end{eqnarray*}
Now by Remark \ref{remark2} we may apply It\^o's formula (\cite[Theorem 4.2.5]{LR} with $V:=L^2(\mu)$, $H:=F^*_{1,2,\nu_0}$) to $\|X^\nu_\lambda-X^{\nu'}_\lambda\|^2_{F^*_{1,2,\nu_0}}$, $\nu,~ \nu'\in(0,\nu_0]$, to obtain for all $t\in[0,T]$, $\lambda\in(0,1)$,
\begin{eqnarray}\label{eqnarray57}
&&\!\!\!\!\!\!\!\!\|X^\nu_\lambda(t)-X^{\nu'}_\lambda(t)\|^2_{F^*_{1,2,\nu_0}}\nonumber\\
&&\!\!\!\!\!\!\!\!+2\int_0^t\int_E\big(\Psi_\lambda(X^\nu_\lambda(s))+\lambda
X^\nu_\lambda(s)-\Psi_\lambda(X^{\nu'}_\lambda(s))-\lambda
X^{\nu'}_\lambda(s)\big)\cdot(X^\nu_\lambda(s)-X^{\nu'}_\lambda(s))d\mu
ds\nonumber\\
=&&\!\!\!\!\!\!\!\!2\int_0^t\big\langle\nu_0\big(\Psi_\lambda(X^\nu_\lambda(s))-\Psi_\lambda(X^{\nu'}_\lambda(s))+\lambda
X^\nu_\lambda(s)-\lambda
X^{\nu'}_\lambda(s)\big),X^\nu_\lambda(s)-X^{\nu'}_\lambda(s)\big\rangle_{F^*_{1,2,\nu_0}}ds\nonumber\\
&&\!\!\!\!\!\!\!\!+2\nu'\int_0^t\big\langle\Psi_\lambda(X^{\nu'}_\lambda(s))+\lambda
X^{\nu'}_\lambda(s),X^\nu_\lambda(s)-X^{\nu'}_\lambda(s)\big\rangle_{F^*_{1,2,\nu_0}}ds\nonumber\\
&&\!\!\!\!\!\!\!\!-2\nu\int_0^t\big\langle\Psi_\lambda(X^\nu_\lambda(s))+\lambda
X^\nu_\lambda(s),X^\nu_\lambda(s)-X^{\nu'}_\lambda(s)\big\rangle_{F^*_{1,2,\nu_0}}ds\nonumber\\
&&\!\!\!\!\!\!\!\!+\int_0^t\|B(s,X^\nu_\lambda(s))-B(s,X^{\nu'}_\lambda(s))\|^2_{L_2(L^2(\mu),F^*_{1,2,\nu_0})}ds\nonumber\\
&&\!\!\!\!\!\!\!\!+2\int_0^t\big\langle
X^\nu_\lambda(s)-X^{\nu'}_\lambda(s),\big(B(s,X^\nu_\lambda(s))-B(s,X^{\nu'}_\lambda(s))\big)dW(s)\big\rangle_{F^*_{1,2,\nu_0}}.
\end{eqnarray}
Since $\Psi_\lambda$ is $\frac{1}{\lambda}$-Lipschitz, we have $\Psi_\lambda(r)-\Psi_\lambda(r')(r-r')\geq\lambda|\Psi_\lambda(r)-\Psi_\lambda(r')|^2$, $\forall r, r'\in\Bbb{R}$, then
\begin{eqnarray}\label{4.1}
&&\!\!\!\!\!\!\!\!2\int_0^t\int_E\big(\Psi_\lambda(X^\nu_\lambda(s))+\lambda
X^\nu_\lambda(s)-\Psi_\lambda(X^{\nu'}_\lambda(s))-\lambda
X^{\nu'}_\lambda(s)\big)\cdot(X^\nu_\lambda(s)-X^{\nu'}_\lambda(s))ds\nonumber\\
\geq&&\!\!\!\!\!\!\!\!2\lambda\int_0^t|X^\nu_\lambda(s)-X^{\nu'}_\lambda(s)|_2^2 ds+2\lambda\int_0^t|\Psi_\lambda(X^\nu_\lambda(s))-\Psi_\lambda(X^{\nu'}_\lambda(s))|_2^2ds.
\end{eqnarray}
Since $L^2(\mu)\subset F^*_{1,2}$ continuously, by \eqref{eqnal19} we have $\|u\|_{F^*_{1,2,\nu_0}}\leq\frac{1}{\sqrt{\nu_0}}|u|_2$, $\forall u\in L^2(\mu)$, by Minkowski's inequality and Young's inequality, we get for all $t\in[0,T]$
\begin{eqnarray}\label{4.2}
&&\!\!\!\!\!\!\!\!2\int_0^t\big\langle\nu_0\big(\Psi_\lambda(X^\nu_\lambda(s))-\Psi_\lambda(X^{\nu'}_\lambda(s))+\lambda
X^\nu_\lambda(s)-\lambda
X^{\nu'}_\lambda(s)\big),X^\nu_\lambda(s)-X^{\nu'}_\lambda(s)\big\rangle_{F^*_{1,2,\nu_0}}ds\nonumber\\
\leq&&\!\!\!\!\!\!\!\!2\nu_0\int_0^t\|\Psi_\lambda(X^\nu_\lambda(s))-\Psi_\lambda(X^{\nu'}_\lambda(s))\|_{F^*_{1,2,\nu_0}}\cdot\|X^\nu_\lambda(s)-X^{\nu'}_\lambda(s)\|_{F^*_{1,2,\nu_0}}ds\nonumber\\
&&\!\!\!\!\!\!\!\!+2\nu_0\lambda\int_0^t\|X^\nu_\lambda(s)-X^{\nu'}_\lambda(s)\|^2_{F^*_{1,2,\nu_0}}ds\nonumber\\
\leq&&\!\!\!\!\!\!\!\!2\lambda\int_0^t|\Psi_\lambda(X^\nu_\lambda(s))-\Psi_\lambda(X^{\nu'}_\lambda(s))|_2^2ds+\frac{\nu_0}{2\lambda}\int_0^t\|X^\nu_\lambda(s)-X^{\nu'}_\lambda(s)\|^2_{F^*_{1,2,\nu_0}}ds\nonumber\\
&&\!\!\!\!\!\!\!\!+2\nu_0\lambda\int_0^t\|X^\nu_\lambda(s)-X^{\nu'}_\lambda(s)\|^2_{F^*_{1,2,\nu_0}}ds.
\end{eqnarray}
Using similar arguments as above and the fact that, by \textbf{(H1)} and \eqref{eqnal01}, $|\Psi_\lambda(r)|\leq C(|r|^m+1_{\{\mu(E)<\infty\}})$, $\forall r\in\Bbb{R}$ with $C$ independent of $\lambda$, we have for all $t\in[0,T]$
\begin{eqnarray}\label{4.3}
&&\!\!\!\!\!\!\!\!2\nu'\int_0^t\big\langle\Psi_\lambda(X^{\nu'}_\lambda(s))+\lambda
X^{\nu'}_\lambda(s),X^\nu_\lambda(s)-X^{\nu'}_\lambda(s)\big\rangle_{F^*_{1,2,\nu_0}}ds\nonumber\\
&&\!\!\!\!\!\!\!\!-2\nu\int_0^t\big\langle\Psi_\lambda(X^\nu_\lambda(s))+\lambda
X^\nu_\lambda(s),X^\nu_\lambda(s)-X^{\nu'}_\lambda(s)\big\rangle_{F^*_{1,2,\nu_0}}ds\nonumber\\
\leq&&\!\!\!\!\!\!\!\!\frac{2\nu'}{\nu_0}\int_0^t|\Psi_\lambda(X^{\nu'}_\lambda(s))|_2^2+\lambda|X^{\nu'}_\lambda(s)|_2^2ds+\int_0^t\|X^\nu_\lambda(s)-X^{\nu'}_\lambda(s)\|^2_{F^*_{1,2,\nu_0}}ds\nonumber\\
&&\!\!\!\!\!\!\!\!+\frac{2\nu}{\nu_0}\int_0^t|\Psi_\lambda(X^{\nu}_\lambda(s))|_2^2+\lambda|X^{\nu}_\lambda(s)|_2^2ds+\int_0^t\|X^\nu_\lambda(s)-X^{\nu'}_\lambda(s)\|^2_{F^*_{1,2,\nu_0}}ds\nonumber\\
\leq&&\!\!\!\!\!\!\!\!\frac{2\nu'}{\nu_0}\int_0^tC\big(|X^{\nu'}_\lambda(s)|_{2m}^{2m}+\mu(E)\cdot1_{\{\mu(E)<\infty\}}\big)+\lambda|X^{\nu'}_\lambda(s)|_{2}^{2}ds\nonumber\\
&&\!\!\!\!\!\!\!\!+\frac{2\nu}{\nu_0}\int_0^tC\big(|X^{\nu}_\lambda(s)|_{2m}^{2m}+\mu(E)\cdot1_{\{\mu(E)<\infty\}}\big)+\lambda|X^{\nu}_\lambda(s)|_{2}^{2}ds\nonumber\\
&&\!\!\!\!\!\!\!\!+2\int_0^t\|X^\nu_\lambda(s)-X^{\nu'}_\lambda(s)\|^2_{F^*_{1,2,\nu_0}}ds\nonumber\\
\leq&&\!\!\!\!\!\!\!\!\frac{2(\nu+\nu')(C+\lambda)}{\nu_0}\int_0^t|X^{\nu'}_\lambda(s)|_{2m}^{2m}+|X^{\nu}_\lambda(s)|_{2m}^{2m}+|X^{\nu'}_\lambda(s)|_{2}^{2}+|X^{\nu}_\lambda(s)|_{2}^{2}+\mu(E)\cdot1_{\{\mu(E)<\infty\}}ds\nonumber\\
&&\!\!\!\!\!\!\!\!+2\int_0^t\|X^\nu_\lambda(s)-X^{\nu'}_\lambda(s)\|^2_{F^*_{1,2,\nu_0}}ds.
\end{eqnarray}
Taking expectation to both sides of \eqref{eqnarray57}, by the BDG inequality for $p=1$, taking \textbf{(H2)(i)}, \eqref{4.1}-\eqref{4.3}, \eqref{eqna11} into account, we obtain for all $t\in[0,T]$
\begin{eqnarray*}
&&\!\!\!\!\!\!\!\!\frac{1}{2}\mathbb{E}\Big[\sup_{s\in[0,t]}\|X^\nu_\lambda(s)-X^{\nu'}_\lambda(s)\|^2_{F^*_{1,2,\nu_0}}\Big]+2\lambda\Bbb{E}\int_0^t
|X^\nu_\lambda(s)-X^{\nu'}_\lambda(s)|_2^2ds\nonumber\\
\leq&&\!\!\!\!\!\!\!\!\frac{(C+\lambda)(\nu+\nu')C_T}{\nu_0}\big(|x|_{2m}^{2m}+|x|_2^2+\mu(E)\cdot1_{\{\mu(E)<\infty\}}\big)\nonumber\\
&&\!\!\!\!\!\!\!\!+\Big(\frac{\nu_0}{2\lambda}+2\nu_0\lambda+C\Big)\mathbb{E}\int_0^t\sup_{r\in[0,s]}\|X^\nu_\lambda(r)-X^{\nu'}_\lambda(r)\|^2_{F^*_{1,2,\nu_0}}ds.
\end{eqnarray*}
Hence by Gronwall's lemma, we have for some $C_{T,\lambda}\in(0,\infty)$
\begin{eqnarray}\label{eqnarray57.2}
&&\!\!\!\!\!\!\!\!\mathbb{E}\Big[\sup_{s\in[0,T]}\|X^\nu_\lambda(s)-X^{\nu'}_\lambda(s)\|^2_{F^*_{1,2,\nu_0}}\Big]+2\lambda\Bbb{E}\int_0^T|X^\nu_\lambda(s)-X^{\nu'}_\lambda(s)|^2_2ds\nonumber\\
\leq&&\!\!\!\!\!\!\!\!\frac{C_{T,\lambda}(\nu+\nu')}{\nu_0}\big(|x|_{2m}^{2m}+|x|_2^2+\mu(E)\cdot1_{\{\mu(E)<\infty\}}\big),~~\forall~\nu,\nu'\in(0,\nu_0].
\end{eqnarray}

Hence there exists
an $(\mathscr{F}_t)_{t\geq0}$-adapted process
\begin{eqnarray}\label{space}
X_\lambda\in
L^2(\Omega;C([0,T];F^*_{1,2}))\cap L^2((0,T)\times\Omega\times E),
\end{eqnarray}
 such that
\begin{eqnarray}\label{eq:4.46}
\lim_{\nu\rightarrow0}\Big\{\mathbb{E}\Big[\sup_{s\in[0,T]}\|X^\nu_\lambda(s)-X_\lambda(s)\|^2_{F^*_{1,2,\nu_0}}\Big]+2\lambda\Bbb{E}\int_0^T|X^\nu_\lambda(s)-X_\lambda(s)|_2^2ds\Big\}=0.
\end{eqnarray}
Consequently, by \textbf{(H2)(i)} we can pass to the limit with $\nu\rightarrow0$ in \eqref{eq:4} to obtain
\begin{eqnarray}\label{4.48}
X_\lambda(t)=&&\!\!\!\!\!\!\!\!x-\lim_{\nu\rightarrow0}(\nu-L)\int_0^t\Psi_\lambda(X^\nu_\lambda(s))+\lambda X^\nu_\lambda(s)ds\nonumber\\
&&\!\!\!\!\!\!\!\!+\int_0^tB(s,X_\lambda(s))dW(s), ~~t\in[0,T],
\end{eqnarray}
where the limit exists in $L^2(\Omega;C([0,T];F^*_{1,2}))$. Furthermore, it follows by \eqref{eq:4.46}, since $\Psi_\lambda$ is Lipschitz, that
\begin{eqnarray}\label{4.49}
\lim_{\nu\rightarrow0}\int_0^\cdot\Psi_\lambda(X^\nu_\lambda(s))+\lambda X^\nu_\lambda(s)ds=\int_0^\cdot\Psi_\lambda(X_\lambda(s))+\lambda X_\lambda(s)ds,
\end{eqnarray}
in $L^2(\Omega;C([0,T];L^2(\mu)))$, hence in $L^2(\Omega;C([0,T];F^*_{1,2}))$.
\vspace{2mm}
Writing $\nu-L=(1-L)+(\nu-1)I$, \eqref{4.48} and \eqref{4.49} imply that the convergence in \eqref{4.49} holds even in $L^2(\Omega;C([0,T];F_{1,2}))$ and that the second term on the right hand-side of \eqref{4.48} is equal to
\begin{eqnarray*}
L\int_0^t\Psi_\lambda(X_\lambda(s))+\lambda X_\lambda(s)ds,
\end{eqnarray*}
which together with \eqref{space} show that $X_\lambda$ is a solution of \eqref{equ:2.3} in the sense of Definition 3.1 in \cite{RWX} with state space $F^*_{1,2}$.

Now let us prove that, since $x\in\mathscr{F}^*_{e}(\subset F^*_{1,2})$, which so far we have not used, that $X_\lambda$ is indeed a solution of \eqref{equ:2.3} on the smaller state space $\mathscr{F}_e^*$ and that \eqref{eqnarray12}-\eqref{eqnarray14} hold. Note that \eqref{equ:2.2} trivially holds, since the convergence in \eqref{4.49} is in $L^2(\Omega;C([0,T];F_{1,2}))$ and since $F_{1,2}\subset\mathscr{F}_e$ continuously.

To prove \eqref{eqnarray12}, we observe that by \eqref{eq:4.46} it follows that as $\nu\rightarrow0$, $X^\nu_\lambda\rightarrow X_\lambda$ in $dt\otimes \Bbb{P}\otimes d\mu$-measure. Hence we have by Fatou's lemma and \eqref{4.4} for all $\varphi\in L^1([0,T];\Bbb{R})$
\begin{eqnarray*}
\int_0^T|\varphi(t)|\Bbb{E}|X_\lambda(t)|_{2m}^{2m}dt&&\!\!\!\!\!\!\!\!\leq\lim_{\nu\rightarrow0}\inf\int_0^T|\varphi(t)|\Bbb{E}|X^\nu_\lambda(t)|_{2m}^{2m}dt\nonumber\\
&&\!\!\!\!\!\!\!\!\leq|\varphi|_{L^1([0,T];\Bbb{R})}C|x|_{2m}^{2m},
\end{eqnarray*}
which implies \eqref{eqnarray12}. Now \eqref{eqnarray13} follows by \textbf{(H1)}.

To prove \eqref{eqnarray15}, applying It\^{o}'s formula (\cite[Theorem 4.2.5]{LR} with $V:=L^2(\mu)$, $H:=F^*_{1,2,\nu_0}$) to $\|X^\nu_\lambda(t)\|^2_{F^*_{1,2,\nu_0}}$, for all $t\in[0,T]$, we have
\begin{eqnarray}\label{xnulambda}
&&\!\!\!\!\!\!\!\!\|X^\nu_\lambda(t)\|^2_{F^*_{1,2,\nu_0}}+2\int_0^t\int_E\big(\Psi_\lambda(X^\nu_\lambda(s))+\lambda X^\nu_\lambda(s)\big)\cdot X^\nu_\lambda(s)d\mu ds\nonumber\\
=&&\!\!\!\!\!\!\!\!\|x\|^2_{F^*_{1,2,\nu_0}}+2\int_0^t\big\langle(\nu_0-\nu)\big(\Psi_\lambda(X^\nu_\lambda(s))+\lambda X^\nu_\lambda(s)\big),X^\nu_\lambda(s)\big\rangle_{F^*_{1,2,\nu_0}}ds\nonumber\\
&&\!\!\!\!\!\!\!\!+2\int_0^t\big\langle X^\nu_\lambda(s),B(s,X^\nu_\lambda(s))dW(s)\big\rangle_{F^*_{1,2,\nu_0}}+\int_0^t\|B(s,X^\nu_\lambda(s))\|_{L_2(L^2(\mu),F^*_{1,2,\nu_0})}.
\end{eqnarray}
Since $\Psi_\lambda$ is monotone,
\begin{eqnarray}\label{Psilambda}
2\int_0^t\int_E\big(\Psi_\lambda(X^\nu_\lambda(s))+\lambda X^\nu_\lambda(s)\big)\cdot X^\nu_\lambda(s)d\mu ds\geq2\lambda\int_0^t|X^\nu_\lambda(s)|_2^2ds.
\end{eqnarray}
Since $\|u\|_{F^*_{1,2,\nu_0}}\leq\frac{1}{\sqrt{\nu_0}}|u|_2$, $\forall u\in L^2(\mu)$, by Minkowski's inequality, Young's inequality and $|\Psi_\lambda(r)|\leq C(|r|^m+1_{\{\mu(E)<\infty\}})$, $\forall r\in\Bbb{R}$ with $C$ independent of $\lambda$, we get for all $t\in[0,T]$, $\lambda\in(0,1)$, $\nu\in(0,\nu_0]$ with $\nu_0\in(0,1]$ that
\begin{eqnarray}\label{Psi}
&&\!\!\!\!\!\!\!\!2\int_0^t\big\langle(\nu_0-\nu)\big(\Psi_\lambda(X^\nu_\lambda(s))+\lambda X^\nu_\lambda(s)\big),X^\nu_\lambda(s)\big\rangle_{F^*_{1,2,\nu_0}}ds\nonumber\\
\leq&&\!\!\!\!\!\!\!\!2\int_0^t\frac{\nu_0-\nu}{\sqrt{\nu_0}}|\Psi_\lambda(X^\nu_\lambda(s))+\lambda X^\nu_\lambda(s)|_2\cdot\|X^\nu_\lambda(s)\|_{F^*_{1,2,\nu_0}}ds\nonumber\\
\leq&&\!\!\!\!\!\!\!\!2\int_0^t\sqrt{\nu_0}\big(|\Psi_\lambda(X^\nu_\lambda(s))|_2+\lambda|X^\nu_\lambda(s)|_2\big)\cdot\|X^\nu_\lambda(s)\|_{F^*_{1,2,\nu_0}}ds\nonumber\\
\leq&&\!\!\!\!\!\!\!\!2\int_0^t|\Psi_\lambda(X^\nu_\lambda(s))|^2_2+|X^\nu_\lambda(s)|^2_2ds+\int_0^t\|X^\nu_\lambda(s)\|^2_{F^*_{1,2,\nu_0}}ds\nonumber\\
\leq&&\!\!\!\!\!\!\!\!C\int_0^t|X^\nu_\lambda(s)|_{2m}^{2m}+|X^\nu_\lambda(s)|_2^2+\mu(E)\cdot1_{\{\mu(E)<\infty\}}ds+\int_0^t\|X^\nu_\lambda(s)\|^2_{F^*_{1,2,\nu_0}}ds.
\end{eqnarray}
Taking expectation to both sides of \eqref{xnulambda}, and taking \eqref{Psilambda} and \eqref{Psi} into \eqref{xnulambda}, by exactly the same arguments as in the proof of \eqref{eqnarray57.2}, except for using \textbf{(H2)(ii)} instead of \textbf{(H2)(i)}, we obtain
\begin{eqnarray*}\label{4.50}
&&\!\!\!\!\!\!\!\!\Bbb{E}\Big[\sup_{s\in[0,T]}\|X^\nu_\lambda(s)\|^2_{F^*_{1,2,\nu_0}}\Big]+\lambda\Bbb{E}\int_0^T|X^\nu_\lambda(s)|_2^2ds\nonumber\\
\leq&&\!\!\!\!\!\!\!\!C_T\big(\|x\|^2_{F^*_{1,2,\nu_0}}+|x|_{2m}^{2m}+|x|_2^2+\mu(E)\cdot1_{\{\mu(E)<\infty\}}\big),~~\forall~\lambda\in(0,1),~\nu\in(0,\nu_0].
\end{eqnarray*}
Hence we get by Fatou's lemma
\begin{eqnarray}\label{4.51}
&&\!\!\!\!\!\!\!\!\mathbb{E}\Big[\sup_{t\in[0,T]}\|X_\lambda(t)\|^2_{F^*_{1,2,\nu_0}}\Big]+\lambda\mathbb{E}\int_0^T|X_\lambda(s)|^2_2ds\nonumber\\
\leq&&\!\!\!\!\!\!\!\!C_T\big(\|x\|^2_{F^*_{1,2,\nu_0}}+|x|_{2m}^{2m}+|x|_2^2+\mu(E)\cdot1_{\{\mu(E)<\infty\}}\big),~~\forall~\lambda\in(0,1).
\end{eqnarray}
Letting $\nu_0\rightarrow0$ and taking \eqref{eqnal6} into account,
we get
\begin{eqnarray*}\label{eqnal7}
&&\!\!\!\!\!\!\!\!\mathbb{E}\Big[\sup_{t\in[0,T]}\|X_\lambda(t)\|^2_{\mathscr{F}^*_e}\Big]+\lambda
\mathbb{E}\int_0^T|X_\lambda(s)|_2^2ds\nonumber\\
\leq&&\!\!\!\!\!\!\!\!C_T\big(\|x\|^2_{\mathscr{F}^*_e}+|x|^{2m}_{2m}+|x|_2^2+\mu(E)\cdot1_{\{\mu(E)<\infty\}}\big),\ \forall~ \lambda\in(0,1),
\end{eqnarray*}
hence \eqref{eqnarray15} follows.

\vspace{2mm}
Now let us prove that $X_\lambda$ is a solution to \eqref{eq:3} with state space $\mathscr{F}_e^*$. By \eqref{equ:2.2} and Lemma \ref{rrw}, we have
\begin{eqnarray*}
L\int_0^\cdot\Psi_\lambda(X_\lambda(s))+\lambda X_\lambda(s)ds\in L^2(\Omega;C([0,T];\mathscr{F}_e^*)).
\end{eqnarray*}
Furthermore, letting $\nu\rightarrow0$ in \textbf{(H2)(ii)}, we conclude from \eqref{4.51} that the stochastic integral in \eqref{equ:2.3} is in $L^2(\Omega;C([0,T];\mathscr{F}_e^*))$ as well. Since $x\in\mathscr{F}_e^*$, \eqref{equ:2.3} (which holds in $F^*_{1,2}$) implies that $X_\lambda\in L^2(\Omega;C([0,T];\mathscr{F}^*_e))$. So, altogether this implies that $X_\lambda$ is a strong solution of \eqref{eq:3} with state space $\mathscr{F}_e^*$ in the sense of \eqref{eqnarray11}-\eqref{equ:2.3}.

\vspace{2mm}
Now finally we prove \eqref{eqnarray14}. Firstly, we have
\begin{eqnarray*}
  &&\!\!\!\!\!\!\!\!d(X^\nu_\lambda(t)-X^\nu_{\lambda'}(t))+(\nu_0-L)\big(\Psi_\lambda(X^\nu_\lambda(t))-\Psi_{\lambda'}(X^\nu_{\lambda'}(t))+\lambda X^\nu_\lambda(t)-\lambda'
X^\nu_{\lambda'}(t))\big)dt\\
&&\!\!\!\!\!\!\!\!+(\nu-\nu_0)\big(\Psi_\lambda(X^\nu_\lambda(t))-\Psi_{\lambda'}(X^\nu_{\lambda'}(t))+\lambda
X^\nu_\lambda(t)-\lambda'
X^\nu_{\lambda'}(t))\big)dt\\
=&&\!\!\!\!\!\!\!\!\big(B(t,X^\nu_\lambda(t))-B(t,X^\nu_{\lambda'}(t))\big)dW(t).
\end{eqnarray*}
By Remark \ref{remark2} we may apply It\^{o}'s formula (\cite[Theorem 4.2.5]{LR} with $V:=L^2(\mu)$, $H:=F^*_{1,2,\nu_0}$) to
$\frac{1}{2}\|X^\nu_\lambda-X^\nu_{\lambda'}\|^2_{F^*_{1,2,\nu_0}}$, to obtain for $\nu\in(0,\nu_0]$, $t\in[0,T]$,
\begin{eqnarray}\label{eqnal8}
&&\!\!\!\!\!\!\!\!\frac{1}{2}\|X^\nu_\lambda(t)-X^\nu_{\lambda'}(t)\|^2_{F^*_{1,2,\nu_0}}\nonumber\\
&&\!\!\!\!\!\!\!\!+\int_0^t\int_E\big(\Psi_\lambda(X^\nu_\lambda(s))+\lambda
X^\nu_\lambda(s)-\Psi_{\lambda'}(X^\nu_{\lambda'}(s))-\lambda'
X^\nu_{\lambda'}(s)\big)\cdot(X^\nu_\lambda(s)-X^\nu_{\lambda'}(s))d\mu ds\nonumber\\
&&\!\!\!\!\!\!\!\!+(\nu-\nu_0)\int_0^t\big\langle\Psi_\lambda(X^\nu_\lambda(s))+\lambda
X^\nu_\lambda(s)-\Psi_{\lambda'}(X^\nu_{\lambda'}(s))-\lambda'
X^\nu_{\lambda'}(s),X^\nu_\lambda(s)-X^\nu_{\lambda'}(s)\big\rangle_{F^*_{1,2,\nu_0}}ds\nonumber\\
=&&\!\!\!\!\!\!\!\!\frac{1}{2}\int_0^t\big\|B(s,X^\nu_\lambda(s))-B(s,X^\nu_{\lambda'}(s))\big\|^2_{L_2(L^2(\mu),F^*_{1,2,\nu_0})}ds\nonumber\\
&&\!\!\!\!\!\!\!\!+\int_0^t \big\langle
X^\nu_\lambda(s)-X^\nu_{\lambda'}(s),
(B(s,X^\nu_\lambda(s))-B(s,X^\nu_{\lambda'}(s)))dW(s)\big\rangle_{F^*_{1,2,\nu_0}}.
\end{eqnarray}
Since
$r=\lambda\Psi_\lambda(r)+(I+\lambda\Psi)^{-1}(r)$, for all $r\in \mathbb{R}$, we have for all $r'\in\Bbb{R}$
\begin{eqnarray}\label{eqnal10}
  (\Psi_\lambda(r)-\Psi_{\lambda'}(r'))(r-r')=&&\!\!\!\!\!\!\!\!\big[\Psi_\lambda(r)-\Psi_{\lambda'}(r')\big]\cdot\big[(I+\lambda\Psi)^{-1}(r)-(I+\lambda'\Psi)^{-1}(r')\big]\nonumber\\
  &&\!\!\!\!\!\!\!\!+\big[\Psi_\lambda(r)-\Psi_{\lambda'}(r')\big]\cdot\big[\lambda\Psi_\lambda(r)-\lambda'\Psi_{\lambda'}(r')\big].
\end{eqnarray}
Note that the first summand in the right hand-side is nonnegative
since $\Psi$ is maximal monotone and since $\Psi_\lambda(r)\in
\Psi((I+\lambda\Psi)^{-1}(r))$(see \cite[page:61]{B}). Plugging \eqref{eqnal10} into
\eqref{eqnal8}, and using that $\|\cdot\|_{F^*_{1,2,\nu_0}}\leq\frac{1}{\sqrt{\nu_0}}|\cdot|_2$ and \textbf{(H2)(i)},
we obtain for $\nu\in(0,\nu_0]$, $t\in[0,T]$
\begin{eqnarray*}
  &&\!\!\!\!\!\!\!\!\frac{1}{2}\|X^\nu_\lambda(t)-X^\nu_{\lambda'}(t)\|^2_{F^*_{1,2,\nu_0}}\nonumber\\
  &&\!\!\!\!\!\!\!\!+\int_0^t\int_E\big(\Psi_\lambda(X^\nu_\lambda(s))-\Psi_{\lambda'}(X^\nu_{\lambda'}(s))\big)\cdot\big(\lambda\Psi_\lambda(X^\nu_{\lambda}(s))-\lambda'\Psi_{\lambda'}(X^\nu_{\lambda'}(s))\big)d\mu ds\nonumber\\
  &&\!\!\!\!\!\!\!\!+\int_0^t\int_E\big(\lambda X^\nu_\lambda(s)-\lambda'X^\nu_{\lambda'}(s)\big)\cdot\big(X^\nu_\lambda(s)-X^\nu_{\lambda'}(s)\big)d\mu ds\nonumber\\
 \leq &&\!\!\!\!\!\!\!\!\frac{(\nu_0-\nu)}{\sqrt{\nu_0}}\int_0^t\big|\Psi_\lambda(X^\nu_\lambda(s))+\lambda
X^\nu_\lambda(s)-\Psi_{\lambda'}(X^\nu_{\lambda'}(s))-\lambda'
X^\nu_{\lambda'}(s)\big|_2\cdot\big\|X^\nu_\lambda(s)-X^\nu_{\lambda'}(s)\big\|_{F^*_{1,2,\nu_0}}ds\nonumber\\
&&\!\!\!\!\!\!\!\!+\frac{C_1}{2}\int_0^t\|X^\nu_\lambda(s)-X^\nu_{\lambda'}(s)\|^2_{F^*_{1,2,\nu_0}}ds\nonumber\\
&&\!\!\!\!\!\!\!\!+\int_0^t \big\langle
X^\nu_\lambda(s)-X^\nu_{\lambda'}(s),
(B(s,X^\nu_\lambda(s))-B(s,X^\nu_{\lambda'}(s)))dW(s)\big\rangle_{F^*_{1,2,\nu_0}}.
\end{eqnarray*}
By the BDG inequality for $p=1$ we get for all
$\lambda, ~\lambda'>0$, $t\in [0,T]$
\begin{eqnarray*}
  &&\!\!\!\!\!\!\!\!\frac{1}{4}\mathbb{E}\Big[\sup_{s\in [0,t]}\|X^\nu_\lambda(s)-X^\nu_{\lambda'}(s)\|^2_{F^*_{1,2,\nu_0}}\Big]\nonumber\\
\leq&&\!\!\!\!\!\!\!\!
C(\lambda+\lambda'+\nu_0)\mathbb{E}\int_0^t\big(|\Psi_\lambda(X^\nu_\lambda(s))|_2^2+|\Psi_{\lambda'}(X^\nu_{\lambda'}(s))|_2^2+|X^\nu_\lambda(s)|_2^2+|X^\nu_{\lambda'}(s)|_2^2\big)
  ds\nonumber\\
  &&\!\!\!\!\!\!\!\!+C\mathbb{E}\int_0^t\sup_{r\in[0,s]}\|X^\nu_\lambda(r)-X^\nu_{\lambda'}(r)\|^2_{F^*_{1,2,\nu_0}}ds.
\end{eqnarray*}
Hence by \textbf{(H1)}, \eqref{4.4} and Gronwall's lemma, there exists $C_T\in(0,\infty)$ independent of $\nu_0$, such that for all $\nu\in(0,\nu_0]$, $\lambda, \lambda'\in(0,1)$,
\begin{eqnarray*}\label{eqnarray63}
&&\!\!\!\!\!\!\!\!\mathbb{E}\Big[\sup_{t\in
[0,T]}\|X^\nu_\lambda(t)-X^\nu_{\lambda'}(t)\|^2_{F^*_{1,2,\nu_0}}\Big]\nonumber\\
\leq&&\!\!\!\!\!\!\!\!
C_T(\lambda+\lambda'+\nu_0)(|x|_2^2+|x|_{2m}^{2m}+\mu(E)\cdot1_{\{\mu(E)<\infty\}}).
\end{eqnarray*}
Then letting
$\nu\rightarrow0$, we obtain
\begin{eqnarray}\label{eqnarray63.1}
&&\!\!\!\!\!\!\!\!\mathbb{E}\Big[\sup_{t\in
[0,T]}\|X_\lambda(t)-X_{\lambda'}(t)\|^2_{F^*_{1,2,\nu_0}}\Big]\nonumber\\
\leq&&\!\!\!\!\!\!\!\!C_T(\lambda+\lambda'+\nu_0)(|x|_2^2+|x|_{2m}^{2m}+\mu(E)\cdot1_{\{\mu(E)<\infty\}}),
\end{eqnarray}
so by letting $\nu_0\rightarrow0$ in \eqref{eqnarray63.1} and taking into account \eqref{eqnal6} we obtain \eqref{eqnarray14}. Consequently, Theorem \ref{Theorem} is proved.
\end{proof}

%%%%%%%%%%%%%%%%%%%%%%%%%%%%%%%%%%%%%%%%%%%%%%%%%% Theorem 3.1

\section{Proof of Theorem \ref{theorem1}}\label{ProofofTheorem3.1}
\setcounter{equation}{0}
 \setcounter{definition}{0}

After all our preparations, to deduce that the solution $X_\lambda$, $\lambda\in(0,1)$, of equation \eqref{eq:3} as $\lambda\rightarrow0$ converges to the unique solution of equation \eqref{eq:1} is now in principle quite standard (at least for experts on stochastic porous media equations), maybe except for proving \eqref{eqnarray3.2}. Since, however, there is no proof in the literature that covers our general case, we give a complete presentation of the arguments in this section.
\vspace{2mm}

\begin{proof}
Let $X_\lambda$ be as in Theorem \ref{Theorem}. Then it follows by Theorem \ref{Theorem} that there exists a process $X\in L^2(\Omega;C([0,T];\mathscr{F}_e^*))$ such that, as $\lambda\rightarrow0$,
\begin{eqnarray}\label{eqnal14}
 X_\lambda\!\!\!\!\!\!\!\!&&\rightarrow X~~~\text{strongly~in~}L^2(\Omega;C([0,T];\mathscr{F}_e^*)),\nonumber\\
 X_\lambda\!\!\!\!\!\!\!\!&&\rightarrow X~~~\text{weak-star~in~} L^\infty([0,T];(L^2\cap L^{2m})(\Omega\times E))\subset L^\infty([0,T];(L^{m+1})(\Omega\times E)),\nonumber\\
 \lambda X_\lambda\!\!\!\!\!\!\!\!&&\rightarrow0~~~\text{strongly~in~}L^2([0,T]\times\Omega\times E),\\
 \Psi_\lambda(X_\lambda)\!\!\!\!\!\!\!\!&&\rightarrow\eta~~~\text{weakly~in~}L^{\frac{2}{m}}([0,T]\times\Omega\times E)\cap L^2([0,T]\times \Omega\times E)\subset L^{\frac{m+1}{m}}([0,T]\times\Omega\times E).\nonumber
\end{eqnarray}
By \eqref{eqnal14} and \textbf{(H2)(i)} we may take the limit $\lambda\rightarrow0$ in \eqref{equ:2.3} in $L^2(\Omega;C([0,T];\mathscr{F}_e^*))$ to obtain that
\begin{eqnarray}\label{5.1}
X=\!\!\!\!\!\!\!\!&&x+\lim_{\lambda\rightarrow0}L\int_0^\cdot\Psi_\lambda(X_\lambda(s))+\lambda X_\lambda(s)ds\nonumber\\
\!\!\!\!\!\!\!\!&&+\int_0^\cdot B(s,X(s))dW(s),
\end{eqnarray}
where we have used that by \eqref{eqnal6} we may take the limit $\nu_0\rightarrow0$ in \textbf{(H2)(i)}. By Lemma \ref{rrw} we conclude that
\begin{eqnarray}\label{5.2}
\lim_{\lambda\rightarrow0}\int_0^\cdot\Psi_\lambda(X_\lambda(s))+\lambda X_\lambda(s)ds
\end{eqnarray}
exists in $L^2(\Omega;C([0,T];\mathscr{F}_e))$, hence by \textbf{(L.1)} in $L^2(\Omega;C([0,T];L^1(g\cdot\mu)))$ for some $g\in L^1(\mu)\cap L^\infty(\mu)$, $g>0$. Hence the limit in \eqref{5.2} coincides with the limit in $dt\otimes d\Bbb{P}\otimes d\mu$-measure. Therefore, $\int_0^\cdot\eta(s)ds\in L^2(\Omega;C([0,T];\mathscr{F}_e))$ and \eqref{5.1} implies
\begin{eqnarray}\label{5.22}
X(t)=x+L\int_0^t\eta(s)ds+\int_0^tB(s,X(s))dW(s),~~t\in[0,T].
\end{eqnarray}
Hence $X(t)$, $t\in[0,T]$, is a solution to \eqref{eq:1} in the sense of Definition \ref{definition1} if we can show that
\begin{eqnarray*}
\eta\in\Psi(X),~~dt\otimes\Bbb{P}\otimes\mu-a.e..
\end{eqnarray*}
For this it suffices to show that
\begin{eqnarray}\label{eqnal16}
\lim_{\lambda\rightarrow0}\sup\mathbb{E}\int_0^T\int_E\Psi_\lambda(X_\lambda)X_\lambda
d\mu dt\leq \mathbb{E}\int_0^T\int_E\eta X d\mu dt.
\end{eqnarray}

\vspace{2mm}
Indeed, since $\Psi_\lambda=\partial j_\lambda$, i.e. $\Psi_\lambda$ is the subdifferential of $j_\lambda$ (cf. \cite[page:7]{B}), where $j_\lambda$ is as in \eqref{eq:005}, we have for all $\lambda\in(0,1)$
\begin{eqnarray}\label{eqnarray67}
\mathbb{E}\int_0^T\int_E\Psi_\lambda(X_\lambda)(X_\lambda-Z)d\mu
dt\geq \mathbb{E}\int_0^T\int_Ej_\lambda(X_\lambda)-j_\lambda(Z)d\mu dt,
\end{eqnarray}
for all $Z\in L^{m+1}((0,T)\times\Omega\times E)$, since $X_\lambda\in (L^2\cap L^{2m}((0,T)\times \Omega\times E))\subset L^{m+1}((0,T)\times \Omega\times E)), \Psi_\lambda(X_\lambda)\in (L^{\frac{2}{m}}\cap L^2)((0,T)\times\Omega\times E)\subset L^{\frac{m+1}{m}}((0,T)\times\Omega\times E)$.\\

Let $\Psi^0:\Bbb{R}\rightarrow\Bbb{R}$ be defined by (see \cite[page: 54]{B}),
\begin{eqnarray*}
\Psi^0(r)=\{y\in\Psi(r);|y|=\inf\{|z|;z\in\Psi(r)\}\},~~\forall r\in \Bbb{R},
\end{eqnarray*}
$\Psi^0$ is monotonically increasing. Define the integral (see \cite[page:54]{B})
\begin{eqnarray*}
j(r):=\int_0^r\Psi^0(s)ds,~~\forall r\in \Bbb{R},
\end{eqnarray*}
the function $j$ is continuous on $(0,r)$ and convex on $\Bbb{R}$ satisfying $j=\partial\Psi$
\begin{eqnarray}\label{5.5}
0\leq j(r)\leq r \Psi^0(r),~~ \forall r\in \Bbb{R},
\end{eqnarray}
and from \cite[page:41, Proposition 2.3]{B}
\begin{eqnarray}\label{eqnal02}
\lim_{\lambda\rightarrow0}\Psi_\lambda(r)=\Psi^0(r),~~\forall r\in\Bbb{R}.
\end{eqnarray}
Recall from \cite[page:48, Theorem 2.9]{B} that $\Psi_\lambda=\partial j_\lambda$ with
\begin{eqnarray}\label{eq:005}
j_\lambda(r):=\inf\Big\{\frac{|r-\overline{r}|^2}{2\lambda}+j(\overline{r}); \overline{r}\in\mathbb{R}\Big\},\ \ \forall r\in \mathbb{R}.
\end{eqnarray}
Moreover,
\begin{eqnarray}
 &&j_\lambda\geq0,\nonumber\\
 &&\lim_{\lambda\rightarrow0}j_\lambda(r)=j(r), \forall r\in \Bbb{R},\label{5.7}\\
 &&j_\lambda(r)\leq j(r), \forall r\in \Bbb{R}.\label{5.8}
\end{eqnarray}
Consequently, for all $Z\in L^{m+1}((0,T)\times \Omega\times E)$
\begin{eqnarray}\label{eqnarray66}
\lim_{\lambda\rightarrow0}\sup\mathbb{E}\int_0^T\int_Ej_\lambda(X_\lambda)-j_\lambda(Z)d\mu dt\geq \mathbb{E}\int_0^T\int_Ej(X)-j(Z)d\mu dt.
\end{eqnarray}
Indeed, by \eqref{5.5}, \eqref{5.8} and \textbf{(H1)}
\begin{eqnarray*}
|j_\lambda(z)|\leq C(|z|^{m+1}+|z|\cdot1_{\{\mu(E)<\infty\}}),
\end{eqnarray*}
hence by Lebesgue's dominated convergence theorem and \eqref{5.7}
$$\lim_{\lambda\rightarrow0}\mathbb{E}\int_0^T\int_Ej_\lambda(Z)d\mu dt=\mathbb{E}\int_0^T\int_Ej(Z)d\mu dt.$$
Furthermore, by \eqref{eqnal01}, \eqref{eqnal02}, \textbf{(H1)} and because $X\in L^{2m}((0,T)\times \Omega\times E)$ we have as $\lambda\rightarrow0$, $\Psi_\lambda(X)\rightarrow\Psi^0(X)$ in $L^2((0,T)\times\Omega\times E)$, hence
\begin{eqnarray*}
 \!\!\!\!\!\!\!\!&&\lim_{\lambda\rightarrow0}\sup\mathbb{E}\int_0^T\int_Ej_\lambda(X_\lambda)-j(X)d\mu dt\nonumber\\
=\!\!\!\!\!\!\!\!&&\lim_{\lambda\rightarrow0}\sup\mathbb{E}\int_0^T\int_Ej_\lambda(X_\lambda)-j_\lambda(X)d\mu dt\nonumber\\
\geq\!\!\!\!\!\!\!\!&&\lim_{\lambda\rightarrow0}\sup\mathbb{E}\int_0^T\int_E\Psi_\lambda(X)(X_\lambda-X)d\mu dt\nonumber\\
=\!\!\!\!\!\!\!\!&&0.
\end{eqnarray*}
Hence by \eqref{eqnal16}, \eqref{eqnarray67} and \eqref{eqnarray66}, we have $\forall Z\in
L^{m+1}((0,T)\times\Omega\times E)$,
\begin{eqnarray*}
\mathbb{E}\int_0^T\int_E\eta(X-Z)d\mu dt\geq
\mathbb{E}\int_0^T\int_E\big(j(X)-j(Z)\big)d\mu dt.
\end{eqnarray*}
This yields
\begin{eqnarray}\label{eqnal17}
\mathbb{E}\int_0^T\int_E\eta(X-Z)d\mu dt\geq
\mathbb{E}\int_0^T\int_E\zeta(X-Z)d\mu dt,
\end{eqnarray}
for all $Z\in L^{m+1}((0,T)\times\Omega\times E)$ and $\zeta\in
L^{\frac{m+1}{m}}((0,T)\times \Omega\times E)$ such that $\zeta\in
\Psi(Z)$ a.e. on $(0,T)\times \Omega\times E$.

By virtue of assumption \textbf{(H1)}, $\Psi$ is maximal
monotone in $L^{m+1}((0,T)\times\Omega\times E)\times
L^{\frac{m+1}{m}}((0,T)\times \Omega\times E)$, so by
\cite[page:34, Theorem 2.2]{B} one knows that the equation
\begin{equation}\label{eq:8}
J(Z)+\Psi(Z)\ni J(X)+\eta,
\end{equation}
where $J(Z)=|Z|^{m-1}Z$, has a unique solution $Z\in L^{m+1}((0,T)\times\Omega\times E)$.

Now if in \eqref{eqnal17}, we take $Z$ to be the solution of \eqref{eq:8} and $\zeta:=J(X)-J(Z)+\eta$,
we obtain
\begin{eqnarray*}
\mathbb{E}\int_0^T\int_E(J(X)-J(Z))(X-Z)d\mu
dt\leq0,
\end{eqnarray*}
i.e.,
\begin{eqnarray*}
\mathbb{E}\int_0^T\int_E(|
X|^{m-1}X-| Z|^{m-1}
Z)(X- Z)d\mu dt\leq0.
\end{eqnarray*}
Since $J:r\rightarrow |r|^{m-1}r$ is strictly increasing, it follows that
\begin{eqnarray*}\label{eqnal9}
\mathbb{E}\int_0^T\int_E((|
X|^{m-1}X-| Z|^{m-1}
Z)(X- Z)d\mu dt=0.
\end{eqnarray*}
Hence $X=Z$ a.e. on $(0,T)\times\Omega\times E$, and thus by
\eqref{eq:8}, we have $\eta\in \Psi(X)$, $\Bbb{P}\otimes dt\otimes\mu$, $a.e.$.

It remains to prove \eqref{eqnal16}. In order to apply the It\^{o} formula from \cite[Theorem 4.2.5]{LR} to the equations \eqref{eq:3} and \eqref{5.22}, we need to use an appropriate Gelfand triple. Recall a special case of \cite[Proposition 3.1]{RRW} that
\begin{eqnarray}\label{lemma7.4}
\mathscr{F}_e\cap L^{\frac{m+1}{m}}(\mu)~\text{is~dense~both~in}~\mathscr{F}_e~\text{and}~L^{\frac{m+1}{m}}(\mu).
\end{eqnarray}
Define
\begin{eqnarray*}
V:=\{u\in L^{m+1}(\mu)|\exists ~C\in(0,\infty)~ \text{such~that}~|\mu(uv)|\leq C\|v\|_{\mathscr{F}_e},~ \forall~ v\in\mathscr{F}_e\cap L^{\frac{m+1}{m}}(\mu)\}.
\end{eqnarray*}
By \eqref{lemma7.4}, $V$ is a subspace of $\mathscr{F}_e^*$ and can be symbolically written as $V=L^{m+1}(\mu)\cap\mathscr{F}^*_{e}$. We note that $V$ is reflexive, since $L^{m+1}(\mu)$ and $\mathscr{F}^*_{e}$ are reflexive, hence so is $L^{m+1}(\mu)\times \mathscr{F}^*_{e}$. But
\begin{eqnarray*}
V\ni u\mapsto (u,\mu(u\ \cdot))\in L^{m+1}\times\mathscr{F}^*_{e}
\end{eqnarray*}
is a homeomorphic isomorphism, mapping $V$ onto a closed subspace of $L^{m+1}\times \mathscr{F}^*_{e}$, which is reflexive.

Recall special cases of \cite[Proposition 3.1]{RRW} and \cite[Lemma 3.3(ii)]{RRW} that
\begin{eqnarray}\label{lemma7.5}
V~\text{is~dense~both~in}~\mathscr{F}^*_{e}~\text{and}~L^{m+1}(\mu),
\end{eqnarray}
and for the map $L:=\overline{L}:\mathscr{F}_{e}\rightarrow\mathscr{F}^*_{e}$ defined in Lemma \ref{rrw} we have for all $v\in\mathscr{F}_{e}\cap L^{\frac{m+1}{m}}(\mu)$, $u\in V$,
\begin{eqnarray*}\label{eq:7}
\langle Lv,u \rangle_{\mathscr{F}^*_{e}}=-\mu(vu).
\end{eqnarray*}

Now we set $H:=\mathscr{F}^*_{e}$ and consider the Gelfand triple
\begin{eqnarray*}
V\subset H\subset V^*.
\end{eqnarray*}
Consider the operator
\begin{eqnarray*}
L:\mathscr{F}_{e}\cap L^{\frac{m+1}{m}}(\mu)\rightarrow \mathscr{F}^*_{e}\subset V^*,
\end{eqnarray*}
as $V^*$-valued, i.e., $Lv=-\mu(v\ \cdot)\in V^*$. Then by \eqref{lemma7.5}, $L$ is continuous w.r.t. the norm $|\cdot|_{\frac{m+1}{m}}$ on $\mathscr{F}_{e}\cap L^{\frac{m+1}{m}}(\mu)$, hence by \eqref{lemma7.4} has a unique continuous linear extension
\begin{eqnarray*}
\widetilde{L}: L^{\frac{m+1}{m}}(\mu)\rightarrow V^*,
\end{eqnarray*}
such that
\begin{eqnarray}\label{7.1}
_{V^*}\langle \widetilde{L}v,u\rangle_V=\mu(v u),~~\forall v\in L^{\frac{m+1}{m}}(\mu), u\in V.
\end{eqnarray}

Now we consider equation \eqref{5.22} in the large space $V^*$. Then, since $\eta\in L^{\frac{m+1}{m}}([0,T]\times\Omega\times E)$,
\begin{eqnarray*}
L\int_0^t\eta(s)ds=\widetilde{L}\int_0^t\eta(s)ds=\int_0^t\widetilde{L}\eta(s)ds\in V^*,
\end{eqnarray*}
so
\begin{eqnarray*}
X(t)=x+\int_0^t\widetilde{L}\eta(s)ds+\int_0^tB(s,X(s))dW(s), t\in[0,T],
\end{eqnarray*}
hence by the It\^{o} formula from \cite[Theorem 4.2.5]{LR} with the Gelfand triple $V=L^{m+1}(\mu)\cap\mathscr{F}^*_{e}\subset \mathscr{F}^*_{e}\subset V^*$, we have
\begin{eqnarray}\label{eqnarray38}
\!\!\!\!\!\!\!\!&&\frac{1}{2}\mathbb{E}\|X(t)\|^2_{\mathscr{F}_e^*}+\mathbb{E}\int_0^t\int_E
\eta\cdot X(s)d\mu ds\nonumber\\
=\!\!\!\!\!\!\!\!&&\frac{1}{2}\|
x\|^2_{\mathscr{F}_e^*}+\frac{1}{2}\mathbb{E}\int_0^t\|B(s,X(s))\|^2_{L_2(L^2(\mu),{\mathscr{F}_e^*})}ds.
\end{eqnarray}
By similar arguments (a special case with $m=1$) we get for \eqref{eq:3}
\begin{eqnarray*}
X_\lambda(t)=x+\int_0^t\widetilde{L}(\Psi_\lambda(X_\lambda(s))+\lambda X_\lambda(s))ds+\int_0^tB(s,X_\lambda(s))dW(s), t\in[0,T],
\end{eqnarray*}
hence
\begin{eqnarray}\label{eqnarray37}
\!\!\!\!\!\!\!\!&&\frac{1}{2}\mathbb{E}\| X_\lambda(t)\|^2_{\mathscr{F}^*_e}-\mathbb{E}\int_0^t~ _{V^*}\big\langle \widetilde{L}\big(\Psi_\lambda(X_\lambda(s))+\lambda X_\lambda(s)\big), X_\lambda(s)\big\rangle_{V} ds \nonumber\\
=\!\!\!\!\!\!\!\!&&\frac{1}{2}\|x\|^2_{\mathscr{F}^*_e}+\frac{1}{2}\mathbb{E}\int_0^t\|
B(s,X_\lambda(s))\|^2_{L_2(L^2(\mu),{\mathscr{F}^*_e})}ds.
\end{eqnarray}
By \eqref{eq:7}, we know that
\begin{eqnarray}\label{5.12}
-~ _{V^*}\big\langle \widetilde{L}\big(\Psi_\lambda(X_\lambda(s))+\lambda X_\lambda(s)\big), X_\lambda(s)\big\rangle_{V}=\int_E\big(\Psi_\lambda(X_\lambda(s))+\lambda X_\lambda(s)\big)\cdot X_\lambda(s)d\mu.
\end{eqnarray}
Letting $\lambda\rightarrow0$ in \eqref{eqnarray37} after plugging in \eqref{5.12}, using \eqref{5.1} and comparing with \eqref{eqnarray38}, we obtain \eqref{eqnal16} (even with $``="$ replacing $``\leq"$).

\medskip
\noindent \textbf{(Uniqueness)}~~Suppose $X_1$, $X_2$ are two strong
solutions to \eqref{eq:1}. We have with $\widetilde{L}$ that
\begin{equation*} \label{eq:10}
\left\{ \begin{aligned}
&d(X_1-X_2)-\widetilde{L}(\eta_1-\eta_2)dt=(B(t,X_1)-B(t,X_2))dW(t),\ \text{in}\ [0,T]\times E,\\
&X_1-X_2=0 \text{~on~} E,
\end{aligned} \right.
\end{equation*}
where $\eta_i\in \Psi(X_i)$, $i=1,2$, a.e. on $\Omega\times
(0,T)\times E$.

As above we may apply It\^{o}'s formula from \cite[Theorem 4.2.5]{LR} with the same Gelfand triple as used in \eqref{eqnarray38} to get
\begin{eqnarray}\label{eqnarray58}
&&\!\!\!\!\!\!\!\!\frac{1}{2}d\|X_1-X_2\|^2_{\mathscr{F}_e^*}-~ _{V^*}\big\langle
\widetilde{L}(\eta_1-\eta_2), X_1-X_2\big\rangle_{V}dt\nonumber\\
=&&\!\!\!\!\!\!\!\!\frac{1}{2}\big\|B(t,X_1)-B(t,X_2)\big\|^2_{L_2(L^2(\mu),\mathscr{F}_e^*)}dt\nonumber\\
&&\!\!\!\!\!\!\!\!+\big\langle X_1-X_2,
(B(s,X_1)-B(s,X_2))dW_t\big\rangle_{\mathscr{F}_e^*}.
\end{eqnarray}
Since $\Psi$ is monotone, by \eqref{eq:7} we have
\begin{eqnarray}\label{eqnarray65}
&&\!\!\!\!\!\!\!\!\mathbb{E}\int_0^T~ _{V^*}\big\langle
-\widetilde{L}(\eta_1-\eta_2),
X_1-X_2\big\rangle_{V}dt\nonumber\\
=&&\!\!\!\!\!\!\!\!\mathbb{E}\int_0^T\int_E
(\eta_1-\eta_2)\cdot( X_1-X_2)d\mu dt\geq0.
\end{eqnarray}
Therefore, integrating \eqref{eqnarray58} from 0 to $t$ and taking expectation, by
\eqref{eqnarray65} and Remark \ref{mainremark} (i), we obtain
\begin{eqnarray*}
  &&\mathbb{E}\|X_1-X_2\|^2_{\mathscr{F}_e^*}\leq C_1
\int_0^t\mathbb{E}\|X_1-X_2\|^2_{\mathscr{F}_e^*}ds,\ \
\forall t\in [0,T],
\end{eqnarray*}
and by Gronwall's inequality we get $X_1=X_2$, $\Bbb{P}-$a.s.. Thus Theorem \ref{theorem1} is proved.
\end{proof}
%%%%%%%%%%%%%%%%%%%%%%%%%%%%%%%%%%%%%%%%%%%%%%%%%%%%%%%%%%%%% the end of main theorem
\section{Applications}\label{Applications}
\setcounter{equation}{0}
 \setcounter{definition}{0}

%\textbf{Example 6.1. (Example for $\Psi$)}

\begin{example}\label{example for psi}\textbf{(Example for $\Psi$)}
Recall a known fact from \cite[page:54]{B}: if $\tilde{\Psi}:\Bbb{R}\rightarrow\Bbb{R}$ is an increasing function and if $\{r_i|i\in\Bbb{N}\}\subset \Bbb{R}$ is the set of all $r\in\Bbb{R}$ for which $\tilde{\Psi}$ is discontinuous in $r$, then one gets a maximal monotone multivalued function $\Psi:\Bbb{R}\rightarrow2^{\Bbb{R}}$ by filling the gaps, i.e., define
\begin{eqnarray*}
\Psi(r):=\left\{
           \begin{array}{ll}
             &\!\!\!\!\!\!\tilde{\Psi}(r),~~~~~~~~~~~~~~~~~~~~~~~~~~~~~~\text{for}~r\notin \{r_i|i\in\Bbb{N}\},\\
             &\!\!\!\!\!\![\tilde{\Psi}(r_j-0), \tilde{\Psi}(r_j+0)],~~~\text{else}.
           \end{array}
         \right.
\end{eqnarray*}

A concrete example of maximal monotone multivalued $\Psi$ satisfying \textbf{(H1)} could be constructed as following: for $m\in[1,\infty)$, define
\begin{eqnarray*}
\Psi(r):=\left\{
           \begin{array}{ll}
             &\!\!\!\!\!\!|r+2|^{m-1}(r+2)-2,~~\text{if}~r\in(-\infty,-2],\\
             &\!\!\!\!\!\!-2,~~~~~~~~~~~~~~~~~~~~~~~~~~~~~~\text{if}~r\in(-2,-1),\\
             &\!\!\!\!\!\![-2,-1],~~~~~~~~~~~~~~~~~~~~~\text{if}~r=-1,\\
             &\!\!\!\!\!\!|r|^{m-1}r,~~~~~~~~~~~~~~~~~~~~~~~~\text{if}~r\in(-1,1),\\
             &\!\!\!\!\!\![1,2],~~~~~~~~~~~~~~~~~~~~~~~~~~~\text{if}~r=1,\\
             &\!\!\!\!\!\!2,~~~~~~~~~~~~~~~~~~~~~~~~~~~~~~~~~\text{if}~r\in(1,2),\\
             &\!\!\!\!\!\!|r-2|^{m-1}(r-2)+2,~~\text{if}~r\in[2,\infty).
           \end{array}
         \right.
\end{eqnarray*}
Hence our result covers non-continuous nonlinearities $\Psi$, as is indicated in the title of the paper.

Another typical example for $\Psi$ satisfying \textbf{(H1)} is
\begin{eqnarray*}
\Psi(r):=|r|^{\gamma}\text{sign}(r),~~\forall r\in\Bbb{R},~~\gamma\in[0,1],
\end{eqnarray*}
where $\text{sign}:\Bbb{R}\rightarrow 2^{\Bbb{R}}$ is defined as
\begin{eqnarray*}
\text{sign}(r):=\left\{
           \begin{array}{ll}
             &\!\!\!\!\!\!1,~~~~~~~~~~~~\text{if}~r>0,\\
             &\!\!\!\!\!\![-1,1],~~~\text{if}~r=0,\\
             &\!\!\!\!\!\!-1,~~~~~~~~~\text{if}~r<0,
           \end{array}
         \right.
\end{eqnarray*}
if $L=\Delta$ on an open bounded domain of $\Bbb{R}^d$ ($d\geq3$), then \eqref{eq:1} corresponds to the fast diffusion equation (FDE) perturbed by multiplicative noise and the special case $\gamma=0$ is related to the SOC model (see \cite{BTW88}) mentioned in the introduction. Applications of the FDE have been proposed in different areas. The equation appears in plasma physics. One behavior of the solution to FDE is it decays to zero in finite time. The equation in dimensions $d\geq3$ with exponent $\gamma=\frac{d-2}{d+2}$ has an important application in geometry (the Yamabe flow) (see \cite[Section 7.5]{V1}), while the Okuda-Dawson law corresponds to $\gamma=\frac{1}{2}$ (see \cite{OD}), for more mathematical studies of the (stochastic) FDE we refer to \cite{GRSIAM, HP85, RRW} and references therein.
\end{example}

%\textbf{Example 6.2. (Example for $B$)}

\begin{example}\label{example for B}\textbf{(Example for $B$)}
For the convenience of the reader, here we recall that there is a standard large class of example for $B$ already described in \cite[page: 3914, Remark 2.9(iii)]{RW}, which satisfy \textbf{(H2)} and \textbf{(H3)}. Let $N\in\Bbb{N}\cup \{+\infty\}$ and $e_k\in L^2(\mu)\cap L^\infty(\mu)$, $1\leq k< N$, be an orthonormal system in $L^2(\mu)$ such that for every $1\leq k< N$ there exists $\xi_k\in(0,\infty)$ such that for all $\nu\in(0,\infty)$
\begin{eqnarray*}\label{eq:6.2}
| _{F^*_{1,2}}\langle x,e_k u\rangle_{F_{1,2}}|\leq\xi_k\|x\|_{F^*_{1,2,\nu}}\|u\|_{F_{1,2,\nu}},~~\forall~u\in F_{1,2},~x\in F^*_{1,2}.
\end{eqnarray*}
The above assumption means that each $e_k$ is a multiplier in $(F^*_{1,2},\|\cdot\|_{F^*_{1,2,\nu}})$ with norm independent of $\nu>0$. Choose $\mu_k\in(0,\infty)$ such that
\begin{eqnarray}\label{eq:6.3}
\sum_{k=1}^{\infty}\mu_k^2(\xi_k^2+|e_k|^2_\infty)<\infty,
\end{eqnarray}
and define for $x\in F^*_{1,2}$, $B(x)\in L_2(L^2(\mu), F^*_{1,2})$ by
\begin{eqnarray}\label{eq:6.1}
B(x)h:=\sum_{k=1}^{\infty}\mu_k\langle e_k,h\rangle x\cdot e_k,~~h\in L^2(\mu).
\end{eqnarray}
By extending $\{e_k|k\in\Bbb{N}\}$ to an orthonormal basis of $L^2(\mu)$, we can get \textbf{(H2)(ii)} as in \cite[page:3915]{RW}, \textbf{(H2)(i)} follows from the linearity of $B(x)$. From \eqref{eq:6.3} and \eqref{eq:6.1}, \textbf{(H3)(i)} and \textbf{(H3)(ii)} are obvious.
\end{example}

%\textbf{Example 6.3. (Example for local $\mathscr{E}$)}

\begin{example}\label{example for local E}\textbf{(Example for local $\mathscr{E}$)}
Suppose $(\mathscr{E}, F_{1,2})$ is a local transient Dirichlet form with generator $L$ such that it admits a carr\'e du champ $\Gamma$ (\cite[Definition 4.1.2]{BH}), which is a unique positive symmetric and continuous bilinear map from $F_{1,2}\times F_{1,2}$ into $L^1(\mu)$ such that
\begin{eqnarray*}
\mathscr{E}(uw,v)+\mathscr{E}(vw,u)-\mathscr{E}(w,uv)=\int w\Gamma(u,v)d\mu,~\forall u,v,w\in F_{1,2}.
\end{eqnarray*}
From \cite[Propostion 6.1.1]{BH}, we know that then
\begin{eqnarray}\label{gamma}
\mathscr{E}(u,v)=\frac{1}{2}\int\Gamma(u,v)d\mu,~~u,v\in F_{1,2},
\end{eqnarray}
which implies \textbf{(H4)(i)}.

By \cite[Corollary 7.1.2]{BH}, we know that for every Lipschitz function $\varphi:\Bbb{R}\rightarrow\Bbb{R}$ which satisfies $\varphi(0)=0$,
\begin{eqnarray*}\label{6.20}
\Gamma(\varphi(u),v)=\varphi'(u)\Gamma(u,v),~~\forall~u,v\in F_{1,2},
\end{eqnarray*}
where $\varphi'$ is any version of the derivatives (defined Lebesgue-a.e.) of $\varphi$. Furthermore, if $\varphi$ is nondecreasing, then
$$\Gamma(u,\varphi(u))=\varphi'(u)\Gamma(u,u)\geq0,~\forall~ u\in F_{1,2},$$
and
\begin{eqnarray*}\label{eq:6.03}
\Gamma(\varphi(u),\varphi(u))=\varphi'(u)\Gamma(u,\varphi(u))\leq \emph{esssup}_{r\in \Bbb{R}}~\varphi'(r)\Gamma(u,\varphi(u)),~\forall~ u\in F_{1,2},
\end{eqnarray*}
which implies \textbf{(H4)(ii)}.

There is an abundance of examples of such Dirichlet forms on very general state spaces $E$, as e.g. finite or infinite dimensional manifolds. A typical example is the Ornstein- Uhlenbeck operator on the Wiener space (\cite[page:83, Proposition 2.3.2]{BH}), for more details and examples we refer e.g. to \cite{BH, FOT, MR} and also \cite{H}. In the following, we briefly cite an example of manifold-valued transient Dirichlet form from \cite{PPA}.

Let $\widetilde{M}$ be a $C^\infty$ manifold and let $\mu$ be a measure with positive $C^\infty$ density with respect to the Lebesgue measure on each local chart. If $(\Xi_i; 1\leq i\leq q)$ are $C^\infty$ vector fields, we can consider the operator
\begin{eqnarray*}\label{example1}
\Gamma(u,v)=\sum_{i=1}^q(\Xi_iu)(\Xi_iv)
\end{eqnarray*}
and the form $\widetilde{\mathscr{E}}$ associated by \eqref{gamma}. We can close this form and obtain a Dirichlet space $\widetilde{\Bbb{D}}$. If now $M$ is an open subset of $\widetilde{M}$, one can let $(\mathscr{E},\Bbb{D})$ be the part of $(\widetilde{\mathscr{E}},\widetilde{\Bbb{D}})$ on $M$; this is the subset of functions, a quasicontinuous modification of which is zero on $\widetilde{M}\backslash M$. If $\widetilde{\mathscr{E}}$ is irreducible and $\widetilde{M}\backslash M$ has positive capacity, then $\mathscr{E}$ is transient. For more details and explanations, we refer to \cite[page:57, Example 1]{PPA} (see also \cite[page:57, Example 2]{PPA} for the example on Riemannian manifold).
\end{example}

%\textbf{Example 6.4. (Example for nonlocal $\mathscr{E}$)}
\begin{example}\label{example for nonlocal E}\textbf{(Example for nonlocal $\mathscr{E}$)}
As is well-known, under quite general assumptions according to the Beurling-Deny representation formula a Dirichlet can be written as the sum of a local Dirichlet form $\mathscr{E}^{(1)}$ (i.e. if it has a square field operator, it satisfies \eqref{6.20}) and a non-local part $\mathscr{E}^{(2)}$ (see \cite[Section 3.2]{FOT} or \cite{HZ} for details). A typical form of $\mathscr{E}^{(2)}$ is as follows
\begin{eqnarray*}
\mathscr{E}^{(2)}(u,v)=\int_E\int_E(u(x)-u(y))(v(x)-v(y)) J(x,dy)\mu(dx),~~u,v\in D(\mathcal{E}^{(2)}),
\end{eqnarray*}
where $J$ is a kernel on $(E,\mathcal{B})$ and $\mu$ is a $\sigma$-finite measure on $(E,\mathcal{B})$. Therefore,
\begin{eqnarray*}
\mathscr{E}^{(2)}(u,v)=\frac{1}{2}\int_E\Gamma(u,v)d\mu,~~u,v\in D(\mathscr{E}^{(2)}),
\end{eqnarray*}
where for $x\in E$
\begin{eqnarray*}
\Gamma(u,v)(x)=2\int_E(u(x)-u(y))(v(x)-v(y))J(x,dy).
\end{eqnarray*}
Clearly, $\Gamma$ does not satisfy \eqref{6.20}, but it satisfies our condition \textbf{(H4)(ii)}. Indeed, for every non-decreasing Lipschitz function $\varphi:\Bbb{R}\rightarrow\Bbb{R}$ with $\varphi(0)=0$ and $u\in D(\mathscr{E}^{(2)})$ we have
\begin{eqnarray*}
\Gamma(\varphi(u),\varphi(u))(x)=&&\!\!\!\!\!\!\!\!2\int_E\Big(\varphi(u(x))-\varphi(u(y))\Big)^2J(x,dy)\\
\leq&&\!\!\!\!\!\!\!\!2Lip\varphi\Big(\int_E1_{\{u(x)\geq u(y)\}}(u(x)-u(y))\big(\varphi(u(x))-\varphi(u(y))\big)J(x,dy)\\
&&\!\!\!\!\!\!\!\!~~~~~~~~+\int_E1_{\{u(x)<u(y)\}}(u(y)-u(x))\big(\varphi(u(y))-\varphi(u(x))\big)J(x,dy)\Big)\\
=&&\!\!\!\!\!\!\!\!2Lip\varphi\Gamma(u,\varphi(u)).
\end{eqnarray*}

There are plenty of examples of such Dirichlet forms, not only when the corresponding state space $E:=\Bbb{R}^d$, but also when $E$ is a general state space, as e.g. a fractal. Concrete Dirichlet forms on $\Bbb{R}^d$ will be given in Example \ref{nonlocal E}, while in this section, we recall an example from \cite{GY} with $E$ being a Sierpi\'{n}ski graph.

Let $c: E\times E\rightarrow [0,\infty)$ be conductance satisfying
\begin{eqnarray*}
&&c(x,y)=c(y,x),\\
&&\pi(x)=\sum_{y\in E}c(x,y)\in(0,+\infty),\\
&&c(x,y)>0\Leftrightarrow x\sim y,
\end{eqnarray*}
for all $x,y\in E$. Let $\mu: E\rightarrow (0,\infty)$ be a positive function given by
$$\mu(x)=\big(\frac{c}{3\lambda}\big)^{|x|},~~x\in E,$$
where $c\in(0,\lambda)\subseteq (0,1)$. Then $\mu$ can be regarded as a finite measure on $E$. Set a symmetric form on $L^2(E;\mu)$ by
\begin{eqnarray*}
           \left\{
           \begin{array}{ll}
             &\!\!\!\!\!\!\!\!\mathscr{E}(u,v)=\frac{1}{2}\sum_{x,y\in E}c(x,y)(u(x)-u(y))(v(x)-v(y)),\\
             &~\\
             &\!\!\!\!\!\!\!\!D(\mathscr{E})=\text{the}~(\mathscr{E})_1-\text{closure~of}~C_0(E),
           \end{array}
         \right.
\end{eqnarray*}
where $C_0(E)$ is the set of all functions with finite support. Then, $(\mathscr{E}, D(\mathscr{E}))$ is a regular Dirichlet form on $L^2(E;\mu)$ and it is transient. For the detailed proofs and more information about Sierpi\'{n}ski gasket, we refer to \cite{GY} and references therein.
\end{example}

%\textbf{Example 6.5 (Nonlocal $\mathscr{E}$ on $\Bbb{R}^d$)}

\begin{example}\label{nonlocal E}\textbf{(Nonlocal $\mathscr{E}$ on $\Bbb{R}^d$)}
Based on \cite[page: 31, Example 1.4.1]{FOT}, \cite[page:48, Example 1.5.2]{FOT} and Example \ref{example for nonlocal E}, one can give a large class of examples of transient nonlocal Dirichlet forms on $\Bbb{R}^d$ satisfying assumptions \textbf{(H4)}. Here we only briefly present the main framework. Let $\{\nu_t, t>0\}$ on $\Bbb{R}^d$ be a continuous symmetric convolution semigroup of probability measures, i.e.,
\begin{eqnarray*}
&&\nu_t\ast\nu_s=\nu_{t+s},~~t,s>0,\nonumber\\
&&\nu_t(A)=\nu_t(-A),~~A\in\mathcal{B}(\Bbb{R}^d),\nonumber\\
&&\lim_{t\downarrow0}\nu_t=\delta~~\text{weakly},\nonumber
\end{eqnarray*}
where $\nu_t\ast\nu_s(A)$ denotes the convolution $\int_{\Bbb{R}^d}\nu_t(A-y)\nu_s(dy)$ and $\delta$ is the Dirac measure concentrated at the origin. The celebrated L\'{e}vy-Khinchin formula (\cite{I}) reads as follows
\begin{eqnarray}\label{con1}
           \left\{
           \begin{array}{ll}
             &\!\!\!\!\!\!\!\!\hat{\nu}_t(x)\big(=\int_{\Bbb{R}^d}e^{i(x,y)}\nu_t(dy)\big)=e^{-t\psi(x)},\\
             &~\\
             &\!\!\!\!\!\!\!\!\psi(x)=\frac{1}{2}(Sx,x)+\int_{\Bbb{R}^d}(1-\cos(x,y))J(dy),
           \end{array}
         \right.
\end{eqnarray}
where
\begin{eqnarray}
&&S~\text{is~a~non-negative~definite}~d\times d~\text{symmetric~matrix},\label{con2}\\
&&J~\text{is~a~symmetric~measure~on}~\Bbb{R}^d\setminus\{0\}~\text{such~that}~\int_{\Bbb{R}^d\setminus\{0\}}\frac{|x|^2}{1+|x|^2}J(dx)<\infty.\label{con3}
\end{eqnarray}
Thus a continuous symmetric convolution semigroup $\{\nu_t,t>0\}$ is characterized by a pair $(S,J)$ satisfying \eqref{con2} and \eqref{con3} through the formula \eqref{con1}.

When $S$ vanishes, the Dirichlet form $\mathscr{E}$ on $L^2(\Bbb{R}^d)$ determined by $\{\nu_t, t>0\}$ is
\begin{eqnarray}
&&\!\!\!\!\!\!\!\!\mathscr{E}(u,v)\!=\!\frac{1}{2}\int_{\Bbb{R}^d\times(\Bbb{R}^d\setminus\{0\})}\!\!\!\!\big(u(x+y)\!-\!u(x)\big)\big(v(x+y)\!-\!v(x)\big)J(dy)dx\label{E}\\
&&\!\!\!\!\!\!\!\!D(\mathscr{E})\!=\!\{u\in L^2(\Bbb{R}^d)\!:\!\int_{\Bbb{R}^d\times(\Bbb{R}^d\setminus\{0\})}\!\!\!\!\big(u(x+y)\!-\!u(x)\big)\big(v(x+y)\!-\!v(x)\big)J(dy)dx<\infty\}\label{DE}.
\end{eqnarray}
Note that $J$ can be extended to a symmetric measure on $\Bbb{R}^d$ by setting $J(\{0\})=0$. From \cite[page:48, Example 1.5.2]{FOT} we know that, if $\frac{1}{\psi(x)}$ is locally integrable on $\Bbb{R}^d$, then $\mathscr{E}$ is transient.

A typical example (see \cite[page:49]{FOT}) is $\psi(x)=|x|^\alpha$ with $0<\alpha\leq 2$, which corresponds to
$$S=0~~~\text{and}~~~J(dy)=\frac{\alpha 2^{\alpha-1}\Gamma(\frac{\alpha+d}{2})}{\pi^{d/2}\Gamma(\frac{2-\alpha}{2})}|y|^{-d-\alpha}dy,$$
then the Dirichlet form $\mathscr{E}$ is transient if and only if $\alpha<d$. This example can be also described in the following way. Let ``$\hat{~~}$" resp. ``$\check{~~}$" denote Fourier transform, i.e.,
$$
\hat{\phi}(x)=(2\pi)^{-\frac{d}{2}}\int \exp[i\langle x,y\rangle_{\Bbb{R}^d}]\phi(y)dy,
$$
resp. its inverse. Define for $\alpha>0$
\begin{eqnarray}\label{L}
(-\Delta)^{\frac{\alpha}{2}}u:=\big(|x|^{\alpha}\hat{u}\big)^{\check{~}}~~~(\in L^2(\Bbb{R}^d;dx)),~~~u\in C_0^\infty(\Bbb{R}^d).
\end{eqnarray}
Then $(-\Delta)^{\frac{\alpha}{2}}$ is a symmetric linear operator on $L^2(\Bbb{R}^d;dx)$ with dense domain $C_0^\infty(\Bbb{R}^d)$. Then $(\mathscr{E}, D(\mathscr{E}))$ in \eqref{E} and \eqref{DE} is the closure of the form
$$
\Bbb{D}^{({\frac{\alpha}{2}})}(u,v):=\frac{1}{2}\int\hat{u}\overline{\hat{v}}|x|^{\alpha}dx,~~~u,v\in C_0^\infty(\Bbb{R}^d).
$$
For more details we refer to \cite[page:43]{MR}.

Finally, we would like to mention that if in \eqref{con1} we choose $\psi(x):=f(|x|^2)$, where $f$ is a Bernstein function (cf. \cite{SSV}) such that $\frac{1}{f(|x|^2)}$ is locally integrable, then in \eqref{L} we get $f(-\Delta)$ instead of $(-\Delta)^{\frac{\alpha}{2}}$, i.e., $L=-f(-\Delta)$ in \eqref{eq:1}.
\end{example}

\section{Appendix}\label{Appendix}
\setcounter{equation}{0}
 \setcounter{definition}{0}

\subsection{Auxiliary results}\label{Auxiliaryresults}
In this part we aim to prove \eqref{eqna19}, which has been used in the proof of Claims \ref{claim1} and \ref{claim2}.
\begin{lemma}\label{lemma7.1}
Let $\nu\in(0,1]$, $\varepsilon\in(0,1)$ and $\lambda\in(0,+\infty)$. For all $x\in F^*_{1,2}$, we have
\begin{eqnarray}\label{eqna16}
&&\!\!\!\!\!\!\!\!\big\langle(\nu-L)((\Psi_\lambda+\lambda I)(J_\varepsilon(x))),
x\big\rangle_{F^*_{1,2,\nu}}\nonumber\\
=&&\!\!\!\!\!\!\!\!
\big\langle(\Psi_\lambda+\lambda I)(J_\varepsilon(x)),J_\varepsilon(x)\big\rangle_2+\varepsilon\|(\nu-L)((\Psi_\lambda+\lambda I)(J_\varepsilon(x)))\|^2_{F^*_{1,2,\nu}}.
\end{eqnarray}
For all $x\in L^2(\mu)$,
\begin{eqnarray}\label{eqna17}
&&\!\!\!\!\!\!\!\!\big\langle(\nu-L)(\Psi_\lambda+\lambda I)(J_\varepsilon(x)),x\big\rangle_2\nonumber\\
=&&\!\!\!\!\!\!\!\!\big\langle(\nu-L)(\Psi_\lambda+\lambda I)(J_\varepsilon(x)),J_\varepsilon(x)\big\rangle_2+\varepsilon\big|(\nu-L)(\Psi_\lambda+\lambda I)(J_\varepsilon(x))\big|^2_2.
\end{eqnarray}
\end{lemma}

\begin{proof}
Recall from \eqref{eqnarray18} that
\begin{eqnarray*}
J_\varepsilon(x)+\varepsilon(\nu-L)\big((\Psi_\lambda+\lambda I)(J_\varepsilon(x))\big)=x,~~\forall x\in F_{1,2}.
\end{eqnarray*}
For $x\in F^*_{1,2}$, to prove \eqref{eqna16}, we rewrite
\begin{eqnarray*}
&&\!\!\!\!\!\!\!\!\big\langle(\nu-L)((\Psi_\lambda+\lambda I)(J_\varepsilon(x))),
x\big\rangle_{F^*_{1,2,\nu}}\nonumber\\
=&&\!\!\!\!\!\!\!\!\big\langle(\nu-L)((\Psi_\lambda+\lambda I)(J_\varepsilon(x))),
J_\varepsilon(x)\big\rangle_{F^*_{1,2,\nu}}\nonumber\\
&&\!\!\!\!\!\!\!\!+\big\langle(\nu-L)((\Psi_\lambda+\lambda I)(J_\varepsilon(x))),
\varepsilon(\nu-L)\big((\Psi_\lambda+\lambda I)(J_\varepsilon(x))\big)\big\rangle_{F^*_{1,2,\nu}}\nonumber\\
=&&\!\!\!\!\!\!\!\!\big\langle(\Psi_\lambda+\lambda I)(J_\varepsilon(x)),J_\varepsilon(x\big\rangle_2+\varepsilon\|(\nu-L)((\Psi_\lambda+\lambda I)(J_\varepsilon(x)))\|^2_{F^*_{1,2,\nu}}.
\end{eqnarray*}
The proof of \eqref{eqna17} is analogous due to the fact that
$J_\varepsilon$ is $\frac{1}{\sqrt{\nu\varepsilon\lambda}}$-Lipschitz in
$L^2(\mu)$, so $A^{\nu,\varepsilon}_\lambda\in L^2(\mu)$ if $x\in
L^2(\mu)$.
\end{proof}

\begin{lemma}\label{lemma7.2}
For each $x\in L^2(\mu)$ and $\varepsilon\in(0,1)$, there exists
$C_3>0$ such that $\forall\nu\in(0,1]$, $\lambda\in(0,+\infty)$, $t\in[0,T]$,
\begin{eqnarray}\label{eqna18}
\mathbb{E}|X^{\nu,\varepsilon}_\lambda(t)|^2_2+2\mathbb{E}\int_0^t\big\langle(\nu-L)(\Psi_\lambda+\lambda I)(J_\varepsilon(X^{\nu,\varepsilon}_\lambda(s))),J_\varepsilon(X^{\nu,\varepsilon}_\lambda(s))\big\rangle_2ds\leq
e^{C_3T}|x|^2_2.
\end{eqnarray}
\end{lemma}
\begin{proof}
Applying It\^{o}'s formula to $|X^{\nu,\varepsilon}_\lambda|^2_2$,
we obtain
\begin{eqnarray*}
&&\!\!\!\!\!\!\!\!d|X^{\nu,\varepsilon}_\lambda(t)|^2_2+2\big\langle(\nu-L)((\Psi_\lambda+\lambda I)(J_\varepsilon(X^{\nu,\varepsilon}_\lambda(t)))),(X^{\nu,\varepsilon}_\lambda(t)\big\rangle_2dt\nonumber\\
=&&\!\!\!\!\!\!\!\!\|B(t,X^{\nu,\varepsilon}_\lambda(t))\|_{L_2(L^2(\mu),L^2(\mu))}^2dt+2\big\langle
X^{\nu,\varepsilon}_\lambda,B(t,X^{\nu,\varepsilon}_\lambda(t))dW(t)\big\rangle_2
\end{eqnarray*}
which by \eqref{eqna17} yields,
\begin{eqnarray*}
&&\!\!\!\!\!\!\!\!d|X^{\nu,\varepsilon}_\lambda(t)|^2_2+2\big\langle(\nu-L)((\Psi_\lambda+\lambda I)(J_\varepsilon(X^{\nu,\varepsilon}_\lambda(t)))),
J_\varepsilon(X^{\nu,\varepsilon}_\lambda(t))\big\rangle_2dt\nonumber\\
&&\!\!\!\!\!\!\!\!+2\varepsilon\big|(\nu-L)((\Psi_\lambda+\lambda I)(J_\varepsilon(X^{\nu,\varepsilon}_\lambda(t))))\big|^2_2dt\nonumber\\
=&&\!\!\!\!\!\!\!\!\|B(t,X^{\nu,\varepsilon}_\lambda(t))\|_{L_2(L^2(\mu),L^2(\mu))}^2dt+2\big\langle
X^{\nu,\varepsilon}_\lambda,B(t,X^{\nu,\varepsilon}_\lambda(t))dW(t)\big\rangle_2.
\end{eqnarray*}
Taking expectation of both sides, by \textbf{(H3)(i)} we get
\begin{eqnarray*}
&&\!\!\!\!\!\!\!\!\mathbb{E}|X^{\nu,\varepsilon}_\lambda(t)|^2_2+2\mathbb{E}\int_0^t\big\langle(\nu-L)((\Psi_\lambda+\lambda I)(J_\varepsilon(X^{\nu,\varepsilon}_\lambda(s))))
,J_\varepsilon(X^{\nu,\varepsilon}_\lambda(s))\big\rangle_2ds\nonumber\\
&&\!\!\!\!\!\!\!\!+2\varepsilon\mathbb{E}\int_0^t\big|(\nu-L)((\Psi_\lambda+\lambda I)(J_\varepsilon(X^{\nu,\varepsilon}_\lambda(s))))\big|_2^2ds\nonumber\\
\leq&&\!\!\!\!\!\!\!\!|x|^2_2+C_3\mathbb{E}\int_0^t|X^{\nu,\varepsilon}_\lambda(s)|^2_2ds.
\end{eqnarray*}
Then by \eqref{eqnarray22} and Gronwall's lemma we get \eqref{eqna18} as claimed.
\end{proof}

\begin{proposition}\label{Proposition7.1}
For $x\in L^2(\mu)$, $t\in[0,T]$, $\varepsilon\in(0,1)$ and $\nu\in(0,1]$, we have
\begin{eqnarray}\label{eqna19}
\mathbb{E}\int_0^t\big\|(\nu-L)\big((\Psi_\lambda+\lambda I)(J_\varepsilon(X^{\nu,\varepsilon}_\lambda(s)))\big)\big\|^2_{F^*_{1,2,\nu}}ds\leq
\frac{1}{2}(\frac{1}{\lambda}+\lambda+C_5)e^{C_3T}|x|^2_2.
\end{eqnarray}
\end{proposition}
\begin{proof}
Let $x\in L^2(\mu)$. Then
\begin{eqnarray*}
&&\!\!\!\!\!\!\!\!\big\|(\nu-L)\big((\Psi_\lambda+\lambda I)(J_\varepsilon(x))\big)\big\|^2_{F^*_{1,2,\nu}}\nonumber\\
=&&\!\!\!\!\!\!\!\!\|(\Psi_\lambda+\lambda I)(J_\varepsilon(x))\|^2_{F_{1,2,\nu}}\nonumber\\
=&&\!\!\!\!\!\!\!\!\int\frac{1}{2}\Gamma\big((\Psi_\lambda+\lambda I)(J_\varepsilon(x)),(\Psi_\lambda+\lambda I)(J_\varepsilon(x))\big)d\mu\nonumber\\
&&\!\!\!\!\!\!\!\!+\nu\big\langle(\Psi_\lambda+\lambda I)(J_\varepsilon(x)),(\Psi_\lambda+\lambda I)(J_\varepsilon(x))\big\rangle_2\nonumber\\
\leq&&\!\!\!\!\!\!\!\!C_5\int\frac{1}{2}\Gamma\big(J_\varepsilon(x),(\Psi_\lambda+\lambda I)(J_\varepsilon(x))\big)d\mu+\nu(\frac{1}{\lambda}+\lambda)\langle(\Psi_\lambda+\lambda I)(J_\varepsilon(x)),J_\varepsilon(x)\rangle_2\nonumber\\
\leq&&\!\!\!\!\!\!\!\!(\frac{1}{\lambda}+\lambda+C_5)\langle
J_\varepsilon(x),(\Psi_\lambda+\lambda I)(J_\varepsilon(x))\rangle_{F_{1,2,\nu}}\nonumber\\
=&&\!\!\!\!\!\!\!\!(\frac{1}{\lambda}+\lambda+C_5)\langle(\nu-L)(\Psi_\lambda+\lambda I)(J_\varepsilon(x)),J_\varepsilon(x)\rangle_2,
\end{eqnarray*}
where in the first inequality we used \textbf{(H4)}, the fact that $r(\Psi_\lambda(r)+\lambda r)\geq0$ for all $r\in \Bbb{R}$ and $\Psi_\lambda$ is $\frac{1}{\lambda}$-Lipschitz (\cite[Page:41, Proposition 2.3 (ii)]{B}), the last equality comes from the fact that $(\Psi_\lambda+\lambda I)(J_\varepsilon(x))\in D(L)$. Now from \eqref{eqna18}, we get the assertion.
\end{proof}

\subsection{The $L^p$-It\^o formula in expectation}\label{itoformula}
The purpose in this section is to prove Theorem \ref{ito} below, which has been used in Lemmas \ref{lemma4} and \ref{lemma5}.

Let $\ell_2$ be the space of all square-summable sequences in $\Bbb{R}$ and $p\in[2,\infty)$. In addition, to the real-valued $L^p$-space, $L^p(\mu):=L^p(E,\mu)$ we consider the $\ell_2$-valued $L^p$-space $L^p(\mu;\ell_2):=L^p(E,\mu;\ell_2)$. We set
$$|g|_p^p:=|g|^p_{L^p(\mu;\ell_2)}=\int_E\|g(x)\|^p_{\ell_2}\mu(dx)=\int_E\Big(\sum_{k=1}^\infty|g_k(x)|^2\Big)^{\frac{p}{2}}\mu(dx).$$

Let $\mathscr{P}$ denote the predictable $\sigma$-algebra on $[0,T]\times\Omega$ corresponding to $(\Omega,\mathscr{F},(\mathscr{F}_t)_{t\geq0},\Bbb{P})$. For $p\in [2,\infty)$ we set
\begin{eqnarray*}
\Bbb{L}^p(T):=L^p([0,T]\times\Omega,\mathscr{P};L^p(\mu))
\end{eqnarray*}
and
\begin{eqnarray*}
\Bbb{L}^p(T;\ell_2):=L^p([0,T]\times\Omega,\mathscr{P};L^p(\mu;\ell_2)),
\end{eqnarray*}
equipped with its standard $L^p$-norms. Since $(E,\mathcal{B})$ is a standard measurable space, by definition there exists a complete metric $d$ on $E$, such that $(E,d)$ is separable, i.e., a Polish space, whose Borel $\sigma$-algebra coincides with $\mathcal{B}$. Below we fix this metric $d$ and denote the corresponding set of all bounded continuous functions by $C_b(E)$.

Let $\mathcal{E}$ be all $g=(g_k)_{k\in\Bbb{N}}\in L^\infty([0,T]\times\Omega;L^\infty(\mu;\ell_2)\cap L^1(\mu;\ell_2))$ such that there exists $j\in\Bbb{N}$ and bounded stopping times $\tau_0\leq\tau_1\leq\cdots\leq\tau_j\leq T$ such that
\begin{equation*}
    g_k=\left\{
      \begin{array}{ll}
        \sum_{i=1}^jg_k^i1_{(\tau_{i-1},\tau_i]}, & \hbox{if~$k\leq j$;} \\
        0, & \hbox{if $k>j$,}
      \end{array}
    \right.
\end{equation*}
where $g_k^i\in C_b(E)\cap L^1(\mu)$, $1\leq i\leq j$.

\begin{claim}\label{claim7.1}
$\mathcal{E}$ is dense in $\Bbb{L}^p(T;\ell_2)$ for all $p\in[2,\infty)$.
\end{claim}
\begin{proof}
Let $f=(f_k)_{k\in\Bbb{N}}\in L^q(T;\ell_2)$, with $q:=\frac{p}{p-1}$, be such that
\begin{eqnarray*}
_{\Bbb{L}^q(T;\ell_2)}\langle f,g\rangle_{\Bbb{L}^p(T;\ell_2)}=\Bbb{E}\int_{0}^{T}\int_E\sum_{k=1}^{\infty}f_kg_kd\mu ds=0 ~~\forall g\in\mathcal{E}.
\end{eqnarray*}
Now let $\sigma\leq\tau$ be two stopping times and $k\in\Bbb{N}$. Define $g\in\Bbb{L}^p(T;\ell_2)$ by $g=(g_k\delta_{ik})_{i\in\Bbb{N}}$, where
$$g_k:=g_k^kI_{(\sigma,\tau]}$$
and $g_k^k\in C_b(E)\cap L^1(\mu)$. Then $g\in\mathcal{E}$, hence
\begin{eqnarray*}
 0=&&\!\!\!\!\!\!\!\!_{\Bbb{L}^q(T;\ell_2)}\langle f,g\rangle_{\Bbb{L}^p(T;\ell_2)}\nonumber\\
  =&&\!\!\!\!\!\!\!\!\Bbb{E}\int_0^T\int_Ef_kg_k^kd\mu \ I_{(\sigma,\tau]}(t)dt,
\end{eqnarray*}
which implies that
\begin{eqnarray*}
\int_Ef_kg_k^kd\mu=0~~dt\otimes\Bbb{P}-a.s.,
\end{eqnarray*}
since all sets of the type $(\sigma,\tau]$ generate the $\sigma$-algebra $\mathscr{P}$ and since $f_k$ is $\mathscr{P}$-measurable. Therefore, since $C_b(E)\cap L^1(\mu)$ is dense in $L^p(\mu)$,
\begin{eqnarray*}
f_k=0~~\text{in}~~L^q(\mu)~~dt\otimes\Bbb{P}-a.s.,~~\text{for~all}~~k\in\Bbb{N}.
\end{eqnarray*}
Now the assertion follows by the Hahn-Banach theorem (\cite[page: 61, Corollary 4.23]{S}).
\end{proof}

\begin{remark}\label{remark7.1}
Let $\mathcal{S}$ be the set of all functions $f\in L^\infty([0,T]\otimes\Omega;L^\infty(\mu)\cap L^1(\mu))$ such that there exist $l\in\Bbb{N}$ and bounded stopping times $\tau'_0\leq\tau'_1\leq\cdots\leq\tau'_l\leq T$ such that $f=\sum_{i=1}^lf^i1_{(\tau'_{i-1},\tau'_i]}$, where $f^i\in C_b(E)\cap L^1(\mu)$, $1\leq i\leq l$. Similarly to Claim \ref{claim7.1}, one can prove that $\mathcal{S}$ is dense in $\Bbb{L}^p(T)$ for all $p\in[2,\infty)$.
\end{remark}

Define $\mathbf{M}:\mathcal{E}\longmapsto \bigcap_{p\geq1}L^p(\Omega;C([0,T];L^p(\mu)))$ as follows:
\begin{eqnarray}\label{eq:7.9}
\mathbf{M}(g)(t)=\int_0^tgdW(s)&&\!\!\!\!\!\!:=\sum_{k=1}^{\infty}\int_{0}^{t}g_kdW_k(s)\nonumber\\
&&\!\!\!\!\!\!=\sum_{i,k=1}^{j}g_k^i\big(W_k(t\wedge\tau_i)-W_k(t\wedge\tau_{i-1})\big),~t\in[0,T],~g\in\mathcal{E}.
\end{eqnarray}
Let us note that the right hand-side of \eqref{eq:7.9} is $\Bbb{P}$-a.s. for every $t\in[0,T]$ a continuous $\mu$-version of $\mathbf{M}(g)(t)\in L^p(E,\mu)$, which for every $x\in E$ is a continuous real-valued martingale and is equal to
\begin{eqnarray}\label{eq:7.7}
\sum_{k=1}^\infty\int_0^tg_k(s,x)dW_k(s),~x\in E, ~t\in[0,T].
\end{eqnarray}

\begin{claim}\label{claim7.2}
Let $p\in [2,\infty)$. Then $\mathbf{M}$ extends to a linear continuous map $\overline{\mathbf{M}}$ from $\Bbb{L}^p(T;\ell_2)$ to $L^p(\Omega;C([0,T];L^p(\mu)))$, such that $\overline{\mathbf{M}}(g)$ is a continuous martingale in $L^p(\mu)$ for all $g\in\Bbb{L}^p(T;\ell_2)$.
\end{claim}
\begin{proof}
We have
\begin{eqnarray*}\label{eq:7.10}
&&\!\!\!\!\!\!\!\!\Bbb{E}\Big[\sup_{t\in[0,T]}\int_E\big|\int_0^tgdW(s)\big|^pd\mu\Big]\nonumber\\
=&&\!\!\!\!\!\!\!\!\Bbb{E}\Big[\sup_{t\in[0,T]}\int_E\Big|\sum_{k=1}^{\infty}\int_{0}^{t}g_{k}(s,x)dW_k(s)\Big|^pd\mu\Big]\nonumber\\
\leq&&\!\!\!\!\!\!\!\!\int_E\Big[\Bbb{E}\sup_{t\in[0,T]}\Big|\sum_{k=1}^{\infty}\int_{0}^{t}g_{k}(s,x)dW_k(s)\Big|^p\Big]d\mu\nonumber\\
\leq&&\!\!\!\!\!\!\!\!c_p\int_E\Big[\Bbb{E}\Big\langle\sum_{k=1}^{\infty}\int_0^\cdot~g_{k}(s,x)dW_k(s)\Big\rangle^{\frac{p}{2}}_T\Big]d\mu\nonumber\\
=&&\!\!\!\!\!\!\!\!c_p\int_E\Bbb{E}\Big(\sum_{k=1}^{\infty}\int_0^Tg^2_{k}(s,x)ds\Big)^{\frac{p}{2}}d\mu\nonumber\\
=&&\!\!\!\!\!\!\!\!c_p\Bbb{E}\Bigg[\int_E\Big(\int_{0}^{T}|g(s,x)|^2_{\ell_2}ds\Big)^{\frac{p}{2}}d\mu\Bigg]^{\frac{2}{p}\cdot\frac{p}{2}}\nonumber\\
\leq&&\!\!\!\!\!\!\!\!c_p\Bbb{E}\Bigg[\int_{0}^{T}\Big(\int_E|g(s,x)|^p_{\ell_2}d\mu\Big)^{\frac{2}{p}}ds\Bigg]^{\frac{p}{2}}\nonumber\\
\leq&&\!\!\!\!\!\!\!\!c_pT^{\frac{p}{2}-1}\Bbb{E}\int_{0}^{T}\big|g(s,\cdot)\big|^p_{L^p(\mu;\ell_2)}ds,
\end{eqnarray*}
where we have used the BDG inequality applied to the real-valued martingale in \eqref{eq:7.7} in the third step, the assumption that $p\geq2$ and Minkowski's inequality in the sixth step and H\"{o}lder's inequality in the last step. Hence the first part of the assertion follows.

To prove the second let $g\in\Bbb{L}^p(T;\ell_2)$. It suffices to prove that for all $f\in L^q(\mu)$ with $q:=\frac{p}{p-1}$,
\begin{eqnarray*}
\int_Ef~\overline{\mathbf{M}}(g)(t)d\mu,~t\in[0,T],
\end{eqnarray*}
is a real-valued martingale (see e.g. \cite[Remark 2.2.5]{LR}). But since for some $g_n\in\mathcal{E}$, $n\in \Bbb{N}$, we have $\forall~t\in[0,T]$ that
\begin{eqnarray*}
\mathbf{M}(g_n)(t)~~\overrightarrow{n\rightarrow\infty}~~\overline{\mathbf{M}}(g)(t)~~\text{in}~~L^p(\Omega;L^p(\mu)),
\end{eqnarray*}
it follows that
\begin{eqnarray*}
\int_Ef~\mathbf{M}(g_n)(t)d\mu~~\overrightarrow{n\rightarrow\infty}~~\int_Ef~\overline{\mathbf{M}}(g)(t)d\mu~~\text{in}~~L^1(\Omega).
\end{eqnarray*}
So, we may assume that $g\in\mathcal{E}$. But in this case by \eqref{eq:7.9} it follows immediately that $\int_Ef~\mathbf{M}(g)(t)d\mu$, $t\in[0,T]$, is a real-valued martingale.
\end{proof}

\vspace{3mm}
Below we define for $g\in\Bbb{L}^p(T;\ell_2)$, $p\in[2,\infty)$,
\begin{eqnarray*}
\int_0^tg(s)dW(s):=\overline{\mathbf{M}}(g)(t),~~t\in[0,T],
\end{eqnarray*}
where $\overline{\mathbf{M}}$ is as in Claim \ref{claim7.2}.

Now we fix $p\in[2,\infty)$ and consider the following process
$$u:\Omega\times[0,T]\rightarrow L^p(\mu),$$
defined by
\begin{eqnarray}\label{eq:7.12}
u(t):=u(0)+\int_0^tf(s)ds+ \int_0^tg(s)dW(s),
\end{eqnarray}
where $u(0)\in L^p(\Omega,\mathcal{F}_0;L^p(\mu))$, $f\in\Bbb{L}^p(T)$ and $g\in\Bbb{L}^p(T;\ell_2)$.

\begin{theorem}\label{ito}\textbf{``It\^{o}-formula in expectation"}
Let $p\in[2,\infty)$, $f\in \Bbb{L}^p(T)$, $g\in \Bbb{L}^p(T;\ell_2)$. Let $u$ be as in \eqref{eq:7.12}. Then for all $t\in[0,T]$,
\begin{eqnarray}\label{eq:7.1}
\Bbb{E}|u(t,x)|_p^p=&&\!\!\!\!\!\!\!\!\Bbb{E}|u(0)|^p+\Bbb{E}\int_0^t\int_Ep|u(s,x)|^{p-2}u(s,x)f(s,x)\mu(dx) ds\nonumber\\
&&\!\!\!\!\!\!\!\!+\frac{1}{2}p(p-1)\Bbb{E}\int_0^t\int_E|u(s,x)|^{p-2} |g(s,x)|^2_{\ell_2}\mu(dx) ds.
\end{eqnarray}
\end{theorem}
\begin{remark}
In the case $E=\Bbb{R}^d$, $\mu=$Lebesgue measure, N. Krylov proved It\^{o}'s formula for the $L^p$-norm of a large class of $W^{1,p}$-valued stochastic processes in his fundamental paper \cite{NK}. In particular, Lemma 5.1 in that paper gives a pathwise It\^{o} formula for processes $u$ as in \eqref{eq:7.12}, which immediately implies \eqref{eq:7.1}. The proof, however, uses a smoothing technique by convoluting the process $u$ in $x$ with Dirac-sequence of smooth functions, which is not available in our more general case, where $(E,\mathcal{B})$ is just a standard measurable space with a $\sigma$-finite measure $\mu$, without further structural assumptions that we wanted to avoid to cover applications e.g. to underlying spaces $E$ which are fractals. Fortunately, the above It\^{o} formula in expectation is enough to prove all main results in this paper without any further assumptions. After the preparations above, its proof is quite simple.
\end{remark}

We recall the following well-known result (see e.g. Theorem 21.7 in \cite{HB}):
\begin{lemma}\label{lemma7.3}
Let $p\in[1,\infty)$, $v_n,v\in L^p(\mu)$ such that $v_n\rightarrow v$ in $\mu$-measure as $n\rightarrow\infty$ and
$$\lim_{n\rightarrow\infty}|v_n|_p=|v|_p.$$
Then
$$\lim_{n\rightarrow\infty}v_n=v~~\text{in}~~L^p(\mu).$$
\end{lemma}
\textbf{Proof of Theorem \ref{ito}}~~
By Claim \ref{claim7.1} and Remark \ref{remark7.1}, we can find $f_n\in\mathcal{S}$, $n\in\Bbb{N}$, and $g_n\in\Bbb{L}^p(T;\ell_2)$, $n\in\Bbb{N}$, such that as $n\rightarrow\infty$
\begin{eqnarray}\label{eqn7.66}
f_n\rightarrow f~~\text{in}~~\Bbb{L}^p(T),
\end{eqnarray}
and
\begin{eqnarray}\label{eqn7.6}
g_n\rightarrow g ~~\text{in}~~\Bbb{L}^p(T;\ell_2).
\end{eqnarray}

For $n\in\Bbb{N}$, define
\begin{eqnarray*}\label{eqn7.7}
u_n(t):=u(0)+\int_0^tf_n(s)ds+\int_0^tg_n(s)dW(s).
\end{eqnarray*}

By \eqref{eq:7.12}, \eqref{eqn7.66}, \eqref{eqn7.6} and Claim \ref{claim7.2}, it follows that as $n\rightarrow\infty$,
\begin{eqnarray}\label{eq:7.13}
\int_0^\cdot f_n(s)ds\!\!\!\!\!\!\!\!&&\rightarrow\int_0^\cdot f(s)ds,\nonumber\\
\int_0^\cdot g_n(s)dW(s)\!\!\!\!\!\!\!\!&&\rightarrow\int_0^\cdot g(s)dW(s),\\
u_n\!\!\!\!\!\!\!\!&&\rightarrow u,\nonumber
\end{eqnarray}
in $L^p(\Omega;C([0,T];L^p(\mu)))$.

Applying the It\^{o} formula to the real-valued semi-martingale $|u_n(t,x)|_p^p$ for each $x\in E$, and integrating w.r.t. $x\in E$ and $\omega\in \Omega$, we obtain
\begin{eqnarray}\label{eq:7.3}
\Bbb{E}\int_E|u_n(t,x)|^p\mu(dx)=&&\!\!\!\!\!\!\!\!\Bbb{E}|u(0)|^p+\Bbb{E}\int_E\int_0^tp|u_n(s,x)|^{p-2}u_n(s,x)\cdot f_n(s,x)ds\mu(dx)\nonumber\\
&&\!\!\!\!\!\!\!\!+\frac{1}{2}p(p-1)\Bbb{E}\int_E\int_0^t|u_n(s,x)|^{p-2}\cdot|g_n(s,x)|^2_{\ell_2}ds\mu(dx).
\end{eqnarray}
Note that by Lemma \ref{lemma7.3} and \eqref{eq:7.13}
\begin{eqnarray*}
|u_n(s)|^{p-2}u_n(s)&&\!\!\!\!\!\!\!\!\rightarrow |u(s)|^{p-2}u(s)~~\text{in}~~L^{\frac{p}{p-1}}(\mu),\\
|u_n(s)|^{p-2}&&\!\!\!\!\!\!\!\!\rightarrow |u(s)|^{p-2}~~\text{in}~~L^{\frac{p}{p-2}}(\mu),
\end{eqnarray*}
as $n\rightarrow\infty$. Hence by \eqref{eqn7.66} and \eqref{eqn7.6} we may pass to the limit $n\rightarrow\infty$ in \eqref{eq:7.3} to get \eqref{eq:7.1}.
\hspace{\fill}$\Box$

\section*{Funding} Michael R\"ockner is supported by the Deutsche Forschungsgemeinschaft (DFG, German Research Foundation) through CRC 1283. Weina Wu is supported by the National Natural Science Foundation of China (NSFC) (No.11901285), China Scholarship Council (CSC) (No.202008320239) and DFG through CRC 1283. Yingchao Xie is supported by the NSFC (No.11931004, No.11771187).

%\section*{Acknowledgements}

%%%%%%%%%%%%%%%%%%%%%%%%%%%%%%%%%%%%%%%%%%%%%%%%%%%%%%%%%%%%%%%%%%%
%%                                                               %%
%% Use the two commands below for producing your bibliography    %%
%% with bibtex, then comment again the commands and include the  %%
%% content of the .bbl file in this file below the commands.     %%
%%                                                               %%
%%%%%%%%%%%%%%%%%%%%%%%%%%%%%%%%%%%%%%%%%%%%%%%%%%%%%%%%%%%%%%%%%%%

%\bibliographystyle{amsplain}
%\bibliography{yourbibfilename}

% add below the content of your .bbl file produced by bibtex.

\end{document}